\documentclass{amsart}
\usepackage{amssymb,amsfonts,amsmath,graphicx}
\usepackage{mathtools, tensor}
\usepackage[alphabetic,y2k,lite]{amsrefs}
\usepackage{fullpage, mystyle}
\usepackage{MnSymbol}
\usepackage{pgfplots}
\usepackage{tikz}
\usetikzlibrary{shapes.geometric}
\usetikzlibrary{calc}
\usetikzlibrary{scopes}
\usetikzlibrary{math}
\usetikzlibrary{decorations.markings,decorations.pathreplacing}
\usepackage{tikz-cd}

\tikzset{
every picture/.style={line width=0.8pt, >=stealth,
                       baseline=-3pt,label distance=-3pt},
emptynode/.style={circle,minimum size=0pt, inner sep=0pt, outer sep=0},
dotnode/.style={fill=black,circle,minimum size=2.5pt, inner sep=1pt, outer
sep=0},
small_dotnode/.style={fill=black,circle,minimum size=2pt, inner sep=0pt, outer sep=0},
morphism/.style={fill=white,circle,draw,thin, inner sep=1pt, minimum size=15pt,
                 scale=0.8},
small_morphism/.style={fill=white,circle,draw,thin,inner sep=1pt,
                       minimum size=0pt, scale=0.8},
coupon/.style={draw,thin, inner sep=1pt, minimum size=18pt,scale=0.8},
semi_morphism/.style args={#1,#2}{
                  fill=white,semicircle,draw,thin, inner sep=1pt, scale=0.8,
                  shape border rotate=#1,
                  label={#1-90:#2}},
regular/.style={densely dashed}, 
edge/.style={very thick, draw=green, text=black},
overline/.style={preaction={draw,line width=#1 mm,white,-}},
overline/.default=2,
thin_overline/.style={preaction={draw,line width=#1 mm,white,-}},
thin_overline/.default=2,
thick_overline/.style={preaction={draw,line width=3mm,white,-}},
boundary/.style={thick,  draw=blue, text=black},
cell/.style={fill=black!10},
subgraph/.style={fill=black!30},
midarrow/.style={postaction={decorate},
                 decoration={
                    markings,
                    mark=at position #1 with {\arrow{>}},
                 }},
midarrow/.default=0.5,
midarrow_rev/.style={postaction={decorate},
                 decoration={
                    markings,
                    mark=at position #1 with {\arrow{<}},
                 }},
midarrow_rev/.default=0.5
}

\newcommand{\al}{\alpha}

\newcommand{\ga}{\gamma}

\newcommand{\ph}{\varphi}

\newcommand{\lmb}{\lambda}
\newcommand{\Lmb}{\Lambda}

\newcommand{\VV}{\mathbf{V}}
\newcommand{\WW}{\mathbf{W}}

\newcommand{\ZTV}{Z_\text{TV}}    
\newcommand{\ZCY}{{Z_\text{CY}}}

\newcommand{\hatZCY}{{\hat{Z}_\text{CY}}}

\newcommand{\Zel}{{\mathcal{Z}^\text{el}}}
\newcommand{\ZA}{{\cZ(\cA)}}

\newcommand{\torus}{{\mathbf{T}^2}}
\newcommand{\punctorus}{{\mathbf{T}_0^2}}
\newcommand{\tnsrbar}{{\, \overline{\tnsr} \,}}
\newcommand{\tnsrproj}[2]{\tensor[_{#1}]\tnsrbar{_{#2}}}

\newcommand{\cD}{{\mathcal{D}}}

\newcommand{\Ann}{\text{Ann}}
\newcommand{\hA}{{\hat{\cA}}}

\newcommand{\ihom}[2]{\Hom_{#1}^{#2}}

\newcommand{\cM}{\mathcal{M}}      

\newcommand{\hZSig}[1]{{\hatZCY(\Sigma_{#1})}}
\newcommand{\ZSig}[1]{{\ZCY(\Sigma_{#1})}}
\newcommand{\wdtld}{\widetilde}

\DeclareMathOperator{\Kar}{Kar} 
\DeclareMathOperator{\htr}{hTr} 
\DeclareMathOperator{\Skein}{Skein}

\newcommand{\st}{\; | \;}                               

\newcommand{\rdual}[1]{\tensor*[^{*}]{{#1}}{}}
\newcommand{\rvee}[1]{\prescript{\vee}{}#1}

\newcommand{\CYA}{{\ZCY(\text{Ann})}}
\newcommand{\hCYA}{{\hatZCY(\text{Ann})}}

\newcommand{\hP}{{\hat{P}}}
\newcommand{\cAA}{{\cA \boxtimes \cA}}
\newcommand{\cAAbop}{{\cA \boxtimes \cA^\bop}}
\newcommand{\onebar}{{\overline{\one}}}
\DeclareMathOperator{\Mat}{Mat}
\DeclareMathOperator{\Vctsp}{Vec}
\newcommand{\matvec}[1]{\Mat_{#1}(\Vctsp)}

\newcommand{\op}{\text{op}}
\newcommand{\hU}{{\hat{U}}}
\newcommand{\UCY}{{U_{\text{CY}}}}
\newcommand{\UUCY}{{U^2_{\text{CY}}}}

\newcommand{\thetabar}{{\overline{\theta}}}

\newcommand{\tnsrz}[1]{{\, \overline{\tnsr}_{#1} \,}}
\newcommand{\htnsr}{{\hat{\tnsr}}}
\newcommand{\tnsrst}{\tnsrz{\text{st}}}
\newcommand{\onest}{{\onebar_{\text{st}}}}
\newcommand{\POP}{{M_\text{POP}}}
\newcommand{\YPOP}{{M_\text{Y}}}

\newcommand{\defend}{\hfill $\triangle$}
\newcommand{\rmkend}{\hfill $\triangle$}

\newcommand{\cS}{{\mathcal{S}}}

\begin{document}

\title{Reduced Tensor Product on the Drinfeld Center}
\author{Ying Hong Tham}
   \address{Department of Mathematics, Stony Brook University, 
            Stony Brook, NY 11794, USA}
    \email{yinghong.tham@stonybrook.edu}
    \urladdr{http://www.math.sunysb.edu/\textasciitilde yinghong/}

\begin{abstract}
The annulus comes with a ``stacking" operation
which glues two annuli into one.
This provides a tensor product structure
on the category of boundary values $\CYA$
associated to the annulus in an extended
Crane-Yetter TQFT.
It is known that $\CYA \simeq \ZA$, the Drinfeld center
of the premodular category $\cA$ from which $\ZCY$
is constructed.
We give an explicit formula for the tensor product on
$\ZA$ that corresponds to the stacking operation
on $\CYA$.
\end{abstract}
\maketitle

\section{Introduction}

The Crane-Yetter state-sum $\ZCY$,
as defined in \ocite{CY}
and further generalized in \ocite{CKY},
is a 4D TQFT that depends on a choice of premodular category $\cA$,
that is, a ribbon fusion category.
When $\cA$ is modular, it is known
that $\ZCY$ is an invertible theory that detects
only the signature and Euler characteristic of
a closed 4-manifold \ocite{CKYeval}.

In \ocite{KT},
we studied a certain category-valued invariant
$\ZCY(\Sigma)$ of (not necessarily closed) surfaces $\Sigma$,
which has been shown to be the codimension-2 extension
of the Crane-Yetter TQFT (see \ocite{tham_thesis}).
Briefly, $\ZCY(\Sigma)$ has as objects
configurations of $\cA$-labeled marked points in
$\Sigma$,
and morphisms are $\cA$-colored graphs in
$\Sigma \times [0,1]$
(see \secref{s:cy}).

It was shown in \ocite{Cooke}, and later independently in \ocite{KT},
that $\ZCY$ satisfies an excision property:
when a surface $\Sigma$ is built from
two smaller surfaces $\Sigma_1$ and $\Sigma_2$
along a circular boundary,
we have that ($\Ann$ short for annulus)
\[
\ZCY(\Sigma) \simeq \ZCY(\Sigma_1) \boxtimes_{\CYA}
  \ZCY(\Sigma_2)
\]
(one can also consider surfaces glued along segments
of boundaries; see \ocite{KT}*{Theorem 7.5}).
Here $\CYA$ is endowed with the ``stacking'' tensor product,
which comes from the operation of gluing
two annuli along a boundary to get a new annulus,
and $\ZCY(\Sigma_1),\ZCY(\Sigma_2)$ have
$\CYA$-module structures by virtue of their
circular boundary component.

Thus, understanding the properties of $\CYA$,
in particular the stacking tensor product,
is crucial to the computation of $\ZCY(\Sigma)$
for arbitrary surfaces.
Moreover, the stacking product, and more generally the excision property,
is a gluing along 1-dimensional manifolds,
and suggests that Crane-Yetter theory can be further extended
to codimension 3
(see \rmkref{r:coincidence} for an example showing the interaction
between skeins in 3-manifolds, boundary values on 2-manifolds,
and the stacking product);
studying $\CYA$ may suggest a suitable type of object that
$\ZCY$ should associate to 1-manifolds.


It is known that, as braided tensor categories,
\[
\ZA \simeq \CYA
\]
where $\ZA$ is the Drinfeld center of $\cA$
equipped with the standard tensor product,
and $\CYA$ is equipped with the pair-of-pants product;
this is closely related to the fact that in Turaev-Viro theory,
$\ZA \simeq \ZTV(S^1)$.
(See \ocite{KT}*{Example 8.2} for a proof of equivalence
as abelian categories;
see \rmkref{r:tnsr_topology} for a brief discussion.)
Thus, we were motivated to describe this stacking
tensor product on $\ZA$ intrinsically,
and the result is the main construction of the paper,
the \emph{reduced tensor product}, denoted $\tnsrbar$
(see \defref{d:tnsr_red_ZA}).

The main result of this paper is
\thmref{t:main}, which says that the
\emph{reduced tensor product} on $\ZA$
encodes the \emph{stacking tensor product}
on $\CYA$,
expressed in the following 2-commutative diagram
of tensor functors:
\begin{equation}
\label{e:final_comm_diag}
\begin{tikzcd}
(\cA,\tnsr) \ar[r,"\htr"] \ar[d,"I"']
& (\hA,\htnsr) \ar[ld,"G"'] \ar[d,"\Kar"]
    \ar[r,"H"] \ar[r,"\simeq"']
& (\hCYA,\tnsrst) \ar[d, "\Kar"]
\\
(\ZA,\tnsrbar)
& (\Kar(\hA),\htnsr) \ar[l,"\Kar(G)"] \ar[l,"\simeq"']
    \ar[r, "\Kar(H)"'] \ar[r,"\simeq"]
& (\CYA,\tnsrst)
\end{tikzcd}
\end{equation}
where
\begin{itemize}
\item $(\ZA,\tnsrbar)$: the Drinfeld center with the
  reduced tensor product, see \secref{s:reduced},
  in particular \defref{d:tnsr_red_ZA};
\item $I$: two-sided adjoint of the forgetful functor $\ZA \to \cA$; see \prpref{p:center};
\item $\hA,\htr,G$: the horizontal trace of $\cA$ as an $\cA$-bimodule category;
this serves as an alternative path to realizing the reduced tensor product; see \defref{d:hA},\defref{d:htr},\prpref{p:hA_ZA};
\item $(\CYA,\tnsrst)$: category of boundary values
  with its stacking tensor product;
	$\hCYA$ is an intermediate construction for $\CYA$;
  see \defref{d:CYA};
\item $H$: sends an object $X$
	to the boundary value with one marked point
	labeled by $X$; see \eqnref{e:def_H};
\item $\Kar$ refers to Karoubi envelope.
\end{itemize}

We also describe a $\ZZ/2$-action on each category,
and show that the functors above
are $\ZZ/2$-equivariant.

As an application of \thmref{t:main},
we describe $\ZCY(\torus)$ purely algebraically
in terms of $\cA$,
where $\torus$ is the torus (see \corref{c:torus_ZA}).

Finally in the last two sections,
we consider these constructions in light of
additional assumptions on $\cA$,
in particular the extreme cases of
when $\cA$ is modular (\secref{s:modular})
and when $\cA$ is symmetric (\secref{s:symmetric}).
In these cases, the reduced tensor product has
simpler descriptions:
\begin{itemize}
\item
When $\cA$ is modular, one has
the equivalence $\ZA \simeq_{\tnsr,br} \cAAbop$
\ocite{Mu},
which has simple objects $X_i \boxtimes X_j^*$,
where $X_i,X_j \in \cA$ are simple.
Then the reduced tensor product on $\ZA$ induces
the following tensor product on $\cAA$
(see \eqnref{e:cAA_simple}):
\[
(X_i \boxtimes X_j^*) \tnsrbar (X_k \boxtimes X_l^*)
\simeq \delta_{j,k} X_i \boxtimes X_l^*
\;\;.
\]

\item
When $\cA$ is symmetric, in particular when
$\cA = \Rep(G)$ for a finite group $G$,
$\ZA$ can be described as the category of
$G$-equivariant bundles over $G$, where $G$ acts
by conjugation on the base
(see e.g. \ocite{BakK}*{Section 3.2}).
Then the reduced tensor product on $\ZA$
corresponds to the fiberwise tensor product of bundles
(see \thmref{t:ZA_DG_red}).
(The supergroup case is a little more complicated,
as it mixes the even and odd components.
See \defref{d:symm_red_super}.)
\end{itemize}

We are certainly not the first to
consider the stacking product.
In \ocite{BZBJ1}, they show that the value of
factorization homology $\int_{\Ann} \cA$,
as introduced in \ocite{AFR},
on the annulus is equivalent to the category of
modules over a certain algebra $\mathfrak{F}_\cA$
in $\cA$.
In a subsequent paper \ocite{BZBJ2},
they show that the stacking product on $\int_{\Ann} \cA$
corresponds to the relative tensor product
$M \tnsr_{\mathfrak{F}_\cA} N$
for two modules $M,N \in \mathfrak{F}_\cA -\mmod$.
Since $\ZCY$ satisfies excision,
it must be the same as $\int_{-} \cA$
(see \ocite{AF}), thus our definitions must agree
(but we have not worked out an explicit
correspondence).

We also note that Wasserman \ocite{wassermansymm}
defined a ``symmetric tensor product"
on $\ZA$ when $\cA$ is symmetric, and coincides with
the reduced tensor product defined in this paper.
He shows in \ocite{wasserman2fold}
that this symmetric tensor product,
together with the usual tensor product,
makes $\ZA$ a ``bilax 2-fold tensor category",
which is of interest to topologists as they
are closely related to iterated loop spaces
(see \ocite{wasserman2fold} and references therein).
The reduced tensor product defined in this paper
should fit into a similar structure,
but we will not pursue this connection in much detail
in this paper, merely commenting on it briefly in
\rmkref{r:2fold} and \rmkref{r:2fold_cy}.

Acknowledgments: The author thanks
David Ben-Zvi,
Jin-Cheng Guu,
David Jordan,
Louis Kauffman,
Alexander Kirillov,
Dennis Sullivan,
and
Thomas Wasserman
for many helpful discussions.

\subsection*{Notations, Conventions, and Lemmas}

We adopt notations and conventions from
\ocite{stringnet}, to which we refer the reader for
details and proofs.
Throughout the paper, we fix an algebraically closed field $\kk$
of characteristic zero. All vector spaces and linear maps will be 
considered over $\kk$.
$\cA$ will denote a premodular category over $\kk$.
In particular, $\cA$ is semisimple with finitely many
isomorphism classes of simple objects. We will denote by $\Irr(\cA)$ the
set of isomorphism classes of simple objects.
We fix a representative $X_i$ for each isomorphism class
$i\in \Irr(\cA)$; abusing language,
we will frequently use the same letter $i$ for denoting
both the isomorphism class and the choice of
representative $X_i$.
We will also denote by $\one=X_1$ the unit object in $\cA$
(which is simple).

Rigidity gives us an involution $-^*$ on $\Irr(\cA)$;
$X_i$'s are chosen so that $X_{i^*} = X_i^*$.
When there is little cause for confusion,
we will suppress the
associativity, unit, and pivotal morphisms
$\delta: X \simeq X^{**}$.

We denote the categorical dimensions of simple objects
by $d_i = \dim X_i$.
For each $i$, we will fix a choice of square root
$\sqrt{d_i}$ so that $\sqrt{d_1} = 1$
and $\sqrt{d_i} = \sqrt{d_{i^*}}$.
We also denote the dimension of $\cA$
by $\cD = \sum_{i\in \Irr(\cA)} d_i^2$,
and fix a square root $\sqrt{\cD}$.
Note that by results of \ocite{ENO2005},
$\cD\neq 0$.

Concatenation of objects will mean tensor product;
this will always be the ``standard" one,
e.g. the tensor product in $\cA$
or tensor product of vector spaces.
When we are constructing a non-standard tensor product,
we use a different symbol like $\tnsrbar$,
and duals are represented by $X^\vee$ or $\rvee X$.

We define the functor
$\eval{} : \cA^{\boxtimes n} \to \Vctsp$
by
\begin{equation} \label{e:vev}
\eval{V_1,\ldots,V_n} = \Hom_\cA(\one,V_1\cdots V_n)
\;\;.
\end{equation}
The pivotal structure gives functorial isomorphisms
\begin{equation} \label{e:cyclic}
z:\eval{V_1,\ldots,V_n} \simeq
  \eval{V_n,V_1,\ldots,V_{n-1}}
\end{equation}
such that $z^n = \id$ (see \ocite{BakK}*{Section 5.3});
so up to canonical isomorphism,
$\eval{V_1,\ldots,V_n}$ only depends on the cyclic
order of $V_1,\ldots,V_n$.

There is a non-degenerate pairing
$\mathbf{\text{ev}}: \eval{V_1,\ldots,V_n} \tnsr
\eval{V_n^*,\ldots,V_1^*} \to \kk$
obtained by post-composing with evaluation maps.
When two nodes in a graph are labeled by
the same Greek letter, say $\al$,
it stands for a summation over
a pair of dual bases:
\begin{equation}\label{e:summation_convention}
\begin{tikzpicture}
\node[morphism] (ph) at (0,0) {$\al$};
\draw[->] (ph)-- +(240:1cm) node[pos=0.7, left] {$V^{*}_1$} ;
\draw[->] (ph)-- +(180:1cm);
\draw[->] (ph)-- +(120:1cm);
\draw[->] (ph)-- +(60:1cm);
\draw[->] (ph)-- +(0:1cm);
\draw[->] (ph)-- +(-60:1cm) node[pos=0.7, right] {$V^{*}_n$};
\node[morphism] (ph') at (3,0) {$\al$};
\draw[->] (ph')-- +(240:1cm) node[pos=0.7, left] {$V_n$} ;
\draw[->] (ph')-- +(180:1cm);
\draw[->] (ph')-- +(120:1cm);
\draw[->] (ph')-- +(60:1cm);
\draw[->] (ph')-- +(0:1cm);
\draw[->] (ph')-- +(-60:1cm) node[pos=0.7, right] {$V_1$};
\end{tikzpicture}
\quad = \sum_\al\quad
\begin{tikzpicture}
\node[morphism] (ph) at (0,0) {$\ph_\al$};
\draw[->] (ph)-- +(240:1cm) node[pos=0.7, left] {$V^{*}_1$} ;
\draw[->] (ph)-- +(180:1cm);
\draw[->] (ph)-- +(120:1cm);
\draw[->] (ph)-- +(60:1cm);
\draw[->] (ph)-- +(0:1cm);
\draw[->] (ph)-- +(-60:1cm) node[pos=0.7, right] {$V^{*}_n$};
\node[morphism] (ph') at (3,0) {$\ph^\al$};
\draw[->] (ph')-- +(240:1cm) node[pos=0.7, left] {$V_n$} ;
\draw[->] (ph')-- +(180:1cm);
\draw[->] (ph')-- +(120:1cm);
\draw[->] (ph')-- +(60:1cm);
\draw[->] (ph')-- +(0:1cm);
\draw[->] (ph')-- +(-60:1cm) node[pos=0.7, right] {$V_1$};
\end{tikzpicture}
\end{equation}
where $\ph_\al\in \eval{V_1,\dots, V_n}$,
$\ph^\al\in \eval{V_n^*,\dots, V_1^*}$
are dual bases with respect to the pairing 
$\mathbf{\text{ev}}$.

Diagrams represent morphisms from top to bottom.
The braiding $c$ and twist operators $\theta$
are depicted by right-hand twists:
\[
c_{X,Y} = 
\begin{tikzpicture}
\draw (-0.3,0.3) .. controls +(down:0.2cm) and +(up:0.2cm) .. (0.3,-0.3)
  node[pos=0,above] {$X$}
  node[below] {$X$};
\draw[overline={1.5}]
(0.3,0.3) .. controls +(down:0.2cm) and +(up:0.2cm) .. (-0.3,-0.3)
  node[pos=0,above] {$Y$}
  node[below] {$Y$};
\end{tikzpicture}
\;\;,\;\;
\theta_X =
\begin{tikzpicture}
\draw
  (0,0.3) .. controls +(down:0.4cm) and +(down:0.1cm) ..
  (0.2,0)
  node[pos=0,above] {$X$};
 \draw[overline={1.25}]
  (0.2,0) .. controls +(up:0.1cm) and +(up:0.4cm) ..
  (0,-0.3)
  node[below] {$X$};
\end{tikzpicture}
\;\;.
\]
We will denote $\thetabar := \theta^\inv$.

A dashed line stands for the regular coloring,
i.e. the sum of all colorings
by simple objects $i$, each taken with coefficient $d_i$:
\begin{equation} \label{e:regular_color}
\begin{tikzpicture}
\draw[regular] (0,0.3) -- (0,-0.3);
\end{tikzpicture}
=\sum_{i\in \Irr(\cA)} d_i \quad 
\begin{tikzpicture}
\draw[midarrow={0.55}] (0,0.3)--(0,-0.3)
  node[pos=0.5, right] {$i$};
\end{tikzpicture}
\;\;.
\end{equation}
An oriented edge labeled $X$ is the same as
the oppositely oriented edge labeled $X^*$.

Note that although we will be discussing 
various pivotal multifusion categories,
the morphisms are described in terms of morphisms
in $\cA$, in particular
all morphisms depicted graphically are
of morphisms in $\cA$,
unless specified otherwise.

Finally, we record some facts and lemmas that are useful
for computations. We give no proofs,
referring readers to \ocite{stringnet},\ocite{KT}.

\begin{equation} \label{e:combine}
\begin{tikzpicture}
\node[small_morphism] (T) at (0,-0.4) {$\al$};
\node[small_morphism] (B) at (0,0.4) {$\al$};
\node at (0,1) {$\dots$};
\node at (0,-1) {$\dots$};
\draw[midarrow] (T)-- +(-120:0.8cm)
  node[pos=0.5,left] {\small $V_1$};
\draw[midarrow] (T)-- +(-60:0.8cm)
  node[pos=0.5,right] {\small $V_n$};
\draw[midarrow_rev] (B)-- +(120:0.8cm)
  node[pos=0.5,left] {\small $V_1$};
\draw[midarrow_rev] (B)-- +(60:0.8cm)
  node[pos=0.5,right] {\small $V_n$};
\draw[regular] (T) -- (B); 
\end{tikzpicture}
=
\begin{tikzpicture}
\node at (0,0) {$\dots$};
\draw[midarrow] (-0.3,1)-- (-0.3,-1)
  node[pos=0.5,left] {\small $V_1$};
\draw[midarrow] (0.3,1)-- (0.3,-1)
  node[pos=0.5,right] {\small $V_n$};
\end{tikzpicture}
\end{equation}

\begin{equation} \label{e:sliding}
\textrm{Sliding lemma:}\;\;\;\;\;\;
\begin{tikzpicture}
\draw[regular] (0,0) circle(0.4cm);
\path[subgraph] (0,0) circle(0.3cm); 
\draw (0,1.2)..controls +(-90:0.8cm) and +(90:0.5cm) ..
  (-0.55, 0) ..controls +(-90:0.5cm) and +(90:0.8cm) ..
  (0,-1.2);
 \end{tikzpicture}
\quad=\quad
\begin{tikzpicture}
\draw[regular] (0,0) circle(0.4cm);
\node[small_morphism] (top) at (0,0.45) {\tiny $\al$};
\node[small_morphism] (bot) at (0,-0.45) {\tiny $\al$};
\path[subgraph] circle(0.3cm); 
\draw (0,1.2)--(top) (bot)--(0,-1.2);    
\end{tikzpicture}
\quad=\quad 
 \begin{tikzpicture}
 \draw[regular] (0,0) circle(0.4cm);
 \path[subgraph] (0,0) circle(0.3cm); 
 \draw (0,1.2)..controls +(-90:0.8cm) and +(90:0.5cm) ..
   (0.55, 0) ..controls +(-90:0.5cm) and +(90:0.8cm) ..
   (0,-1.2);
 \end{tikzpicture}
\end{equation}
where the shaded region can contain anything.


\begin{equation} \label{e:al_natural}
\text{For }
f: V_1 \to W_1,\;\;\;\;
\begin{tikzpicture}
\node[small_morphism] (al) at (0,0) {$\al$};
\draw[->] (al)-- +(240:1cm) node[pos=0.7, left] {$V^{*}_1$} ;
\draw[->] (al)-- +(180:1cm);
\draw[->] (al)-- +(120:1cm);
\draw[->] (al)-- +(60:1cm);
\draw[->] (al)-- +(0:1cm);
\draw[->] (al)-- +(-60:1cm) node[pos=0.7, right] {$V^{*}_n$};
\end{tikzpicture}
\;\;\;\;
\begin{tikzpicture}
\node[small_morphism] (al) at (0,0) {$\al$};
\draw[->] (al)-- +(240:1cm) node[pos=0.7, left] {$V_n$} ;
\draw[->] (al)-- +(180:1cm);
\draw[->] (al)-- +(120:1cm);
\draw[->] (al)-- +(60:1cm);
\draw[->] (al)-- +(0:1cm);
\draw[midarrow={0.3},->] (al)-- +(-60:1cm) node[right] {$W_1$};
\node[small_morphism] at (-60:0.6cm) {\small $f$};
\end{tikzpicture}
=
\begin{tikzpicture}
\node[small_morphism] (be) at (0,0) {$\beta$};
\draw[midarrow={0.3},->] (be)-- +(240:1cm)
  node[left] {$V^{*}_1$} ;
\node[small_morphism] at (-120:0.6cm) {\tiny $f^*$};
\draw[->] (be)-- +(180:1cm);
\draw[->] (be)-- +(120:1cm);
\draw[->] (be)-- +(60:1cm);
\draw[->] (be)-- +(0:1cm);
\draw[->] (be)-- +(-60:1cm)
  node[pos=0.7,right] {$V^{*}_n$};
\end{tikzpicture}
\;\;
\begin{tikzpicture}
\node[small_morphism] (be) at (0,0) {$\beta$};
\draw[->] (be)-- +(240:1cm) node[pos=0.7, left] {$V_n$} ;
\draw[->] (be)-- +(180:1cm);
\draw[->] (be)-- +(120:1cm);
\draw[->] (be)-- +(60:1cm);
\draw[->] (be)-- +(0:1cm);
\draw[->] (be)-- +(-60:1cm) node[pos=0.7,right] {$W_1$};
\end{tikzpicture}
\end{equation}

\begin{equation} \label{e:charge_conservation}
\text{When $\cA$ is modular,}
\;\;
\frac{1}{\cD}
\begin{tikzpicture}
\draw[midarrow={0.9}] (0,0) -- (0,-0.5);
\node at (0.15,-0.4) {\tiny $i$};
\draw[overline={1.5},regular] (0,0.1) circle (0.25cm);
\draw[overline={1.5}] (0,0) -- (0,0.5);
\end{tikzpicture}
=
\delta_{i,1} \id_{X_i}
\end{equation}

\section{Reduced Tensor Product on $\ZA$}
\label{s:reduced}

Recall that $\ZA$, the Drinfeld center of $\cA$,
is the category with

Objects: pairs $(X,\ga)$, where $X\in \cA$ and $\ga$
is a half-braiding, i.e. a natural isomorphism of functors  
$\ga_A : A \tnsr X\to X \tnsr A$, $A\in \cC$
satisfying natural compatibility conditions.

Morphisms: $\Hom_\ZA((X,\ga), (X',\ga'))=\{f\in \Hom_\cA(X,X')\st f\ga=\ga' f\}$.

We recall some well-known properties of $\ZA$.
These do not require the braiding on $\cA$,
only its spherical fusion structure.

The following is standard (see e.g.
\ocite{EGNO}*{Corollary 8.20.14}):
\begin{definition}
$\ZA$ is modular, with tensor product
\begin{equation}
(X,\ga) \tnsr (Y,\mu) := (X\tnsr Y,\ga \tnsr \mu)
  \label{e:tnsr_std}
\end{equation}
where
\begin{equation}
  (\ga \tnsr \mu)_A := (\id_X \tnsr \mu_A) \circ (\ga_A \tnsr \id_Y)
  \label{e:tnsr_hfbrd_std}
\end{equation}
\label{d:tnsr_std}
and left dual given by
\begin{equation}
(X,\ga)^* = (X^*,\ga^*),
\text{ where }
(\ga^*)_A =
(\ga_{\rdual{A}})^* = 
\begin{tikzpicture}
\node[small_morphism] (ga) at (0,0) {\tiny $\ga$};
\draw (ga) .. controls +(-45:0.5cm) and +(0:0.5cm) ..
  (-0.3,-0.9) .. controls +(180:0.5cm) and +(down:1cm) .. (-1,1)
  node[above] {$A$};
\draw (ga) .. controls +(135:0.5cm) and +(180:0.5cm) ..
  (0.3,0.9) .. controls +(0:0.5cm) and +(up:1cm) .. (1,-1)
  node[below] {$A$};
\draw (ga) .. controls +(-120:0.5cm) and +(0:0.2cm) ..
  (-0.25,-0.6) .. controls +(180:0.2cm) and +(down:1cm) .. (-0.5,1)
  node[above] {$X^*$};
\draw (ga) .. controls +(60:0.5cm) and +(180:0.2cm) ..
  (0.25,0.6) .. controls +(0:0.2cm) and +(up:1cm) .. (0.5,-1)
  node[below] {$X^*$};
\end{tikzpicture}
,
\end{equation}
and similarly the right dual is
$\rdual{(X,\ga)} = (\rdual{X}, \rdual{\ga})$,
where $(\rdual{\ga})_A = \rdual{(\ga_{A^*})}$.

The pivotal structure is given by that of $\cA$.
\defend
\end{definition}

The following is taken from \ocite{stringnet}*{Theorem 8.2}:

\begin{proposition}\label{p:center}\par\noindent
Let $F : \ZA \to \cA$ be the natural forgetful functor
$F: (X,\ga) \mapsto X$. 
Then it has a two-sided adjoint functor
$I : \cA \to \ZA$, given by 
\begin{equation}
\label{e:induction}
I(A)=\big( \bigoplus_{i\in \Irr(\cA)} X_i \tnsr A \tnsr X_i^*, \Gamma \big)
\end{equation}
where $\Gamma$ is the half-braiding given by
\begin{equation}
\Gamma_B = 
\displaystyle{\sum_{i,j \in \Irr(\cA)}\sqrt{d_i}\sqrt{d_j}}\quad 
\begin{tikzpicture}
\draw (0,1)--(0,-1);
\node[above] at (0,1) {$A$};
\node[below] at (0,-1) {$A$};
\node[small_morphism] (L) at (-0.5,0) {$\al$};
\node[small_morphism] (R) at (0.5,0) {$\al$};
\draw (L)-- +(0,1) node[above] {$i$}; \draw (L)-- +(0,-1) node[below] {$j$}; 
\draw (R)-- +(0,1)node[above] {$i^*$}; \draw (R)-- +(0,-1)node[below] {$j^*$}; 
\draw (-1, 1) .. controls +(down:0.5cm) and +(135:0.5cm) .. (L);
\draw (1, -1) .. controls +(up:0.5cm) and +(-45:0.5cm) .. (R);
\node[above] at (-1,1) {$B$};
\node[below] at (1,-1) {$B$};
\end{tikzpicture}
.
\end{equation}
We refer the reader to \ocite{stringnet} for more details.
\end{proposition}

\begin{proposition}
The adjoint functor $I: \cA \to \ZA$ above is dominant.
More explicitly,
the object $(X,\ga)$ is a direct summand of $I(X)$,
given by the projection $P_{(X,\ga)}$, described below:

\begin{equation}
P_{(X,\ga)} :=
\sum_{i,j \in \Irr(\cA)} \frac{\sqrt{d_i}\sqrt{d_j}}{\cD}
\begin{tikzpicture}
\tikzmath{
  \toprow = 1;
  \bottom = -1;
  \lf = -0.7;
  \rt = 0.7;
}
\node[small_morphism] (ga) at (0,0.3) {$\ga$}; 
\node[small_morphism] (ga2) at (0,-0.3) {$\ga$};
\draw (ga) -- (0,\toprow) node[above] {$X$};
\draw (ga) -- (ga2);
\draw (ga2) -- (0,\bottom) node[below] {$X$}; 
\draw[midarrow_rev] (ga) to[out=180,in=-90] (\lf,\toprow);
\node at (-0.6,0.35) {\tiny $i$};
\draw[midarrow] (ga) to[out=0,in=-90] (\rt,\toprow);
\draw[midarrow] (ga2) to[out=180,in=90] (\lf,\bottom);
\node at (-0.6,-0.25) {\tiny $j$};
\draw[midarrow_rev] (ga2) to[out=0,in=90] (\rt,\bottom);
\end{tikzpicture}
.
\end{equation}
\end{proposition}

\subsection{The Reduced Tensor Product, $\tnsrbar$}
\par \noindent
\\

From here on, we will assume that $\cA$ is premodular.
We define a different monoidal structure on $\ZA$,
which we call the \emph{reduced tensor product},
denoted by $\tnsrbar$;
we emphasize that \emph{braiding} is required
to define this monoidal structure.
We will see in the coming sections that
the definitions and results proved here
have topological origin,
but we first present them purely algebraically.

We note that many of the definitions and formulas here
can be found in \ocite{tham}*{Section 3},
where we used them to define a tensor product
on a different category, $\Zel(\cA)$.
It is no coincidence that there is overlap,
as $\Zel(\cA)$ is the category associated to the once-punctured
torus, $\ZCY(\punctorus)$ (see \ocite{KT}*{Proposition 9.5}),
and both the tensor product defined here
and the one in \ocite{tham} have similar topological origins
(compare \ocite{tham}*{Remark 3.21} and \rmkref{r:tnsr_topology}).
We also note again that \ocite{wassermansymm}
defines a similar tensor product that coincides with ours
when $\cA$ is symmetric.

The definition of $\tnsrbar$ will be given in several steps.

\begin{definition}
\label{d:tnsr_red}
Let $X,Y$ be objects in $\cA$,
and let $\ga,\mu$ be half-braidings
on $X,Y$ respectively.
The \emph{reduced tensor product of $X$ and $Y$
with respect to $\ga,\mu$}
is defined as the image of the
projection $Q_{\ga,\mu}:X \tnsr Y \to X \tnsr Y$ defined
as follows:
\begin{equation} \label{e:Qproj}
X \tnsrproj{\ga}{\mu} Y :=
\im (Q_{\ga,\mu})
\hspace{5pt},\hspace{5pt}
Q_{\ga,\mu} :=
\frac{1}{\cD}
\begin{tikzpicture}
\draw (0.3,-0.8) -- (0.3,0.8)
  node[above] {$Y$};
\draw[regular, overline] (0,0) circle (0.5cm); 
\node[small_morphism] (n2) at (0.3,0.4) {\tiny $\mu$};
\draw[overline] (-0.3,-0.8) -- (-0.3,0.8)
  node[above] {$X$};
\node[small_morphism] (n1) at (-0.3,-0.4) {\tiny $\ga$};
\end{tikzpicture}
\hspace{10pt}.
\end{equation}
It is easy to check that $Q^2_{\ga,\mu} = Q_{\ga,\mu}$.
\defend
\end{definition}
(Essentially \ocite{tham}*{Definition 3.14};
compare \ocite{wassermansymm}*{Equation (11)}.)

There is an accompanying definition of $\tnsrbar$
for half-braidings:
\begin{definition}
Let $\ga,\mu$ be half-braidings on $X,Y$ respectively.
Define $\ga \tnsrbar \mu$ to be natural transformation
$- \tnsr XY \to XY \tnsr -$ given by
\[
(\ga \tnsrbar \mu)_A =
\frac{1}{\cD}
\begin{tikzpicture}
\draw (0.3,-0.8) -- (0.3,0.8)
  node[above] {$Y$}
  node[pos=0,below] {$Y$};
\draw[regular, overline] (0,0) circle (0.5cm); 
\node[small_morphism] (n2) at (0.3,0.4) {\tiny $\mu$};
\draw[overline] (-0.3,-0.8) -- (-0.3,0.8)
  node[above] {$X$}
  node[pos=0,below] {$X$};
\node[small_morphism] (n1) at (-0.3,-0.4) {\tiny $\ga$};
\node[small_morphism] (a1) at (-0.5,0) {\tiny $\al$};
\draw (a1) to[out=150,in=-90] (-1,0.8)
  node[above] {$A$};
\node[small_morphism] (a2) at (0.5,0) {\tiny $\al$};
\draw (a2) to[out=-30,in=90] (1,-0.8)
  node[below] {$A$};
\end{tikzpicture}
\;\;.
\]
\defend
\end{definition}

In general, $\ga \tnsrbar \mu$ fails to be a
half-braiding, but only insofar as it is not an isomorphism.
Observe that $\ga \tnsrbar \mu$ commutes with
$Q_{\ga,\mu}$, so it descends to a
natural transformation
$- \tnsr (X\tnsrproj{\ga}{\mu} Y)
\to (X\tnsrproj{\ga}{\mu} Y) \tnsr -$.
This is in fact a half-braiding
on $X\tnsrproj{\ga}{\mu} Y$:

\begin{lemma}
Let $\ga,\mu$ be half-braidings on $X,Y$ respectively.
Consider the half-braidings $\ga \tnsr c^\inv$ and $c \tnsr \mu$
on $X \tnsr Y$, where recall $c$ is the braiding on $\cA$.
Observe that the projection $Q_{\ga,\mu}$ intertwines
both $\ga \tnsr c^\inv$ and $c \tnsr \mu$,
thus they restrict to half-braidings
on $X \tnsrproj{\ga}{\mu} Y$.
Then as half-braidings on
$X \tnsrproj{\ga}{\mu} Y$,
we have
\begin{equation*}
\ga \tnsrbar \mu =
  \ga \tnsr c^\inv = c \tnsr \mu\
\;\;.
\end{equation*}
\label{l:hfbrd_red}
\end{lemma}

\begin{proof}
This follows from the following computation:
\begin{equation}
\label{e:hfbrd_sym}
\begin{tikzpicture}
\draw (0.3,-1) -- (0.3,1);
\draw[regular, overline] (0,0.2) circle (0.5cm); 
\node[small_morphism] (n2) at (0.3,0.6) {\tiny $\mu$};
\draw[overline] (-0.3,-1) -- (-0.3,1);
\node[small_morphism] (n1) at (-0.3,-0.2) {\tiny $\ga$};
\node[small_morphism] (n3) at (-0.3,-0.7) {\tiny $\ga$};
\draw (n3) to[out=150,in=-90] (-1,1);
\draw[overline] (n3) to[out=-20,in=90] (1,-1);
\end{tikzpicture}
=
\begin{tikzpicture}
\draw (0.3,-1) -- (0.3,1);
\draw[regular, overline] (0,0.2) circle (0.5cm); 
\node[small_morphism] (n2) at (0.3,0.6) {\tiny $\mu$};
\draw[overline] (-0.3,-1) -- (-0.3,1);
\node[small_morphism] (n1) at (-0.3,-0.2) {\tiny $\ga$};
\node[small_morphism] (n3) at (-0.3,-0.7) {\tiny $\ga$};
\draw (n3) to[out=150,in=-90] (-1,1);
\node[small_morphism] (a1) at (0.1,-0.25) {\tiny $\al$};
\draw (n3) to[out=-20,in=-90] (a1);
\node[small_morphism] (a2) at (0.5,0.2) {\tiny $\al$};
\draw (a2) to[out=-50,in=90] (1,-1);
\end{tikzpicture}
=
\begin{tikzpicture}
\draw (0.3,-1) -- (0.3,1);
\draw[regular, overline] (0,0) circle (0.5cm); 
\node[small_morphism] (n2) at (0.3,0.4) {\tiny $\mu$};
\draw[overline] (-0.3,-1) -- (-0.3,1);
\node[small_morphism] (n1) at (-0.3,-0.4) {\tiny $\ga$};
\node[small_morphism] (a1) at (-0.5,0) {\tiny $\al$};
\draw (a1) to[out=150,in=-90] (-1,1);
\node[small_morphism] (a2) at (0.5,0) {\tiny $\al$};
\draw (a2) to[out=-30,in=90] (1,-1);
\end{tikzpicture}
=
\begin{tikzpicture}
\node[small_morphism] (n3) at (0.3,0.7) {\tiny $\mu$};
\draw (n3) to[out=160,in=-90] (-1,1);
\draw (n3) to[out=-30,in=90] (1,-1);
\draw (n3) -- (0.3,1);
\draw (n3) -- (0.3,-1);
\draw[regular, overline] (0,-0.2) circle (0.5cm); 
\node[small_morphism] (n2) at (0.3,0.2) {\tiny $\mu$};
\draw[overline] (-0.3,-1) -- (-0.3,1);
\node[small_morphism] (n1) at (-0.3,-0.6) {\tiny $\ga$};
\end{tikzpicture}
\;\;.
\end{equation}
\end{proof}
(Essentially \ocite{tham}*{Lemma 3.16};
compare \ocite{wassermansymm}*{Lemma 10}.)

\begin{definition}
\label{d:tnsr_red_ZA}
Let $(X,\ga),(Y,\mu)$ be objects in $\ZA$.
Their \emph{reduced tensor product}
is defined as follows:
\begin{equation*}
  (X,\ga) \tnsrbar (Y,\mu) := (X \tnsrproj{\ga}{\mu} Y, \ga \tnsrbar \mu)
\end{equation*}

and for $f : (X,\ga) \to (X,\ga'), g : (Y,\mu) \to (Y',\mu')$,
their reduced tensor product is
\[
  f \tnsrbar g := Q_{\ga',\mu'} \circ (f \tnsr f') \circ Q_{\ga,\mu}
\]
or more simply, it is $f \tnsr f'$ restricted to $X \tnsrproj{\ga}{\mu} Y$.
\defend
\end{definition}

\begin{lemma}
The reduced tensor product of \defref{d:tnsr_red_ZA}
is associative. More precisely,
if $a : (X_1 \tnsr X_2) \tnsr X_3 \simeq X_1 \tnsr (X_2 \tnsr X_3)$
is the associativity constraint of $\cA$,
and $\ga_1,\ga_2,\ga_3$ are half-braidings on $X_1,X_2,X_3$
respectively,
then $a$ restricts to an isomorphism
\[
  a : (X_1 \tnsrproj{\ga_1}{\ga_2} X_2) \tnsrproj{\ga_1 \tnsrbar \ga_2}{\ga_3} X_3
  \simeq
  X_1 \tnsrproj{\ga_1}{\ga_2 \tnsrbar \ga_3} (X_2 \tnsrproj{\ga_2}{\ga_3} X_3)
\]
and hence
\[
  a : \big( (X_1,\ga_1) \tnsrbar (X_2,\ga_2) \big) \tnsrbar (X_3,\ga_3)
  \simeq
  (X_1,\ga_1) \tnsrbar \big( (X_2,\ga_2) \tnsrbar (X_3,\ga_3) \big)
\;\;.
\]
Furthermore, $a$ is natural in $(X_i,\ga_i)$,
and satisfies the pentagon equation.
\label{l:assoc}
\end{lemma}

\begin{proof}
Follows easily from \lemref{l:hfbrd_red}.
See also \ocite{tham}*{Corollary 3.17},
and compare \ocite{wassermansymm}*{Lemma 17}.
\end{proof}

\begin{proposition}
\label{p:reduced_pivotal}
$(\ZA, \tnsrbar)$ is a pivotal multifusion category.
More precisely,
\begin{itemize}
\item the associativity constraint is given by
  the associativity constraint of $\cA$
  (see \lemref{l:assoc});
\item the unit object, denoted $\onebar$, is $I(\one)$
  (see \prpref{p:center}),
  with left and right unit constraints given by
\begin{equation}
l_{(X,\ga)} :=
\sum_i \frac{\sqrt{d_i}}{\sqrt{\cD}}
\begin{tikzpicture}
\node[small_morphism] (ga) at (0,0) {\tiny $\ga$};
\draw (ga) -- (0,0.8) node[above] {$X$};
\draw (ga) -- (0,-0.6) node[below] {$X$};
\draw[midarrow_rev] (ga) to[out=180,in=-45] (-1,0.8);
\draw[overline] (ga) to[out=0,in=-45] (-0.5,0.8);
\node at (-0.8,0.35) {\tiny $i$};
\node[above] at (-0.75,0.8) {$I(\one)$};
\end{tikzpicture}
\hspace{10pt}
r_{(X,\ga)} := \sum_i \frac{\sqrt{d_i}}{\sqrt{\cD}}
\begin{tikzpicture}
\node[small_morphism] (ga) at (0,0) {\tiny $\ga$};
\draw (ga) -- (0,-0.6) node[below] {$X$};
\draw (ga) to[out=180,in=-135] (0.5,0.8);
\draw[midarrow] (ga) to[out=0,in=-135] (1,0.8);
\node at (0.8,0.3) {\tiny $i$};
\draw[overline] (ga) -- (0,0.8) node[above] {$X$};
\node[above] at (0.75,0.8) {$I(\one)$};
\end{tikzpicture}
\;\; ;
\end{equation}
\item the duals are given by
\begin{equation*}
(X,\ga)^\vee := (X^*,\ga^\vee),
\hspace{10pt}
\rvee (X,\ga) := (\rdual{X}, \rvee \ga)
\end{equation*}

where

\begin{equation} \label{e:brd_duals}
\ga^\vee := 
\begin{tikzpicture}
\node[small_morphism] (ga) at (0,0) {\tiny $\ga$};
\draw (ga) -- (-0.5,0.8);
\draw[overline={1.5}]
  (ga) .. controls +(-120:0.3cm) and +(down:0.2cm) ..
  (-0.25,0) .. controls +(up:0.2cm) and +(down:0.5cm) ..
  (0,0.8);
\node[above] at (0,0.8) {$X^*$};
\draw 
  (ga) .. controls +(60:0.3cm) and +(up:0.2cm) ..
  (0.25,0) .. controls +(down:0.2cm) and +(up:0.5cm) ..
  (0,-0.8);
\draw[overline] (ga) -- (0.5,-0.8);
\end{tikzpicture}
=
\begin{tikzpicture}
\node[small_morphism] (ga) at (0,0) {\tiny $\ga *$};
\draw (ga) -- (0,-0.8);
\draw[overline={1.5}] (ga) .. controls +(120:0.5cm) and +(150:1.5cm) .. (0.5,-0.8);
\draw (ga) .. controls +(-60:0.5cm) and +(-30:1.5cm) .. (-0.5,0.8);
\draw[overline={1.5}] (ga) -- (0,0.8) node[above] {$X^*$};
\end{tikzpicture}
=
\begin{tikzpicture}
\node[small_morphism] (ga) at (0,0) {\tiny $\ga *$};
\draw (ga) -- (0,0.8);
\node[above] at (0,0.8) {$X^*$};
\draw (ga) .. controls +(-60:0.4cm) and +(-80:1.5cm) .. (-0.5,0.8);
\draw[thin_overline={1}] (ga) -- (0,-0.8);
\draw[thin_overline={1}] (ga) .. controls +(120:0.4cm) and +(100:1.5cm) .. (0.5,-0.8);
\end{tikzpicture}
, \hspace{10pt}
\rvee \ga
:=
\begin{tikzpicture}
\node[small_morphism] (ga) at (0,0) {\tiny $\ga$};
\draw
  (ga) .. controls +(-100:0.3cm) and +(down:0.3cm) ..
  (0.25,0) .. controls +(up:0.2cm) and +(down:0.4cm) ..
  (0,0.8);
\node[above] at (0,0.8) {$\rdual{X}$};
\draw[overline={1.5}] (ga) -- (0.5,-0.8);
\draw (ga) -- (-0.5,0.8);
\draw [overline={1.5}]
  (ga) .. controls +(80:0.3cm) and +(up:0.3cm) ..
  (-0.25,0) .. controls +(down:0.2cm) and +(up:0.4cm) ..
  (0,-0.8);
\end{tikzpicture}
=
\begin{tikzpicture}
\node[small_morphism] (ga) at (0,0) {\tiny $* \ga$};
\draw (ga) -- (0,0.8);
\node[above] at (0,0.8) {$\rdual{X}$};
\draw (ga) .. controls +(-60:0.4cm) and +(-80:1.5cm) .. (-0.5,0.8);
\draw[overline={1.5}] (ga) -- (0,-0.8);
\draw[overline={1.5}] (ga) .. controls +(120:0.4cm) and +(100:1.5cm) .. (0.5,-0.8);
\end{tikzpicture}
=
\begin{tikzpicture}
\node[small_morphism] (ga) at (0,0) {\tiny $* \ga$};
\draw (ga) -- (0,-0.8);
\draw[overline={1.5}] (ga) .. controls +(120:0.5cm) and +(150:1.5cm) .. (0.5,-0.8);
\draw (ga) .. controls +(-60:0.5cm) and +(-30:1.5cm) .. (-0.5,0.8);
\draw[overline={1.5}] (ga) -- (0,0.8) node[above] {$X^*$};
\end{tikzpicture}
\end{equation}

with evaluation and coevaluation maps given by
(the projections $Q_{\ga^\vee,\ga}$ are implicit)

\begin{equation}
\ev_{(X,\ga)} :=
\sum_i \frac{\sqrt{d_i}}{\sqrt{\cD}}
\begin{tikzpicture}
\node[above] at (-0.2,0.7) {$X^*$ \;};
\node[above] at (0.2,0.7) {\;\; $X$};
\draw (-0.2,0.7) .. controls +(down:1cm) and +(180:0.1cm) ..
  (0,-0.3) .. controls +(0:0.1cm) and +(down:1cm) ..  (0.2,0.7);
\node[morphism] (brd) at (0,0.2) {\tiny $\ga^\vee \tnsrbar \ga$};
\draw[midarrow={0.8}] (brd) .. controls +(170:0.7cm) and +(up:0.3cm) .. (-0.1,-0.7);
\node at (-0.3,-0.5) {\tiny $i$};
\draw (brd) .. controls +(-10:0.6cm) and +(up:0.2cm) .. (0.1,-0.7);
\node[below] at (0,-0.7) {$I(\one)$};
\end{tikzpicture}
=
\sum_i \frac{\sqrt{d_i}}{\sqrt{\cD}}
\begin{tikzpicture}
\node[small_morphism] (ga) at (0,0.1) {\tiny $\ga$};
\draw[midarrow={0.8}] (ga) .. controls +(120:0.5cm) and +(up:1cm) .. (-0.3,-0.7);
\node at (-0.4,-0.5) {\tiny $i$};
\draw[overline] (ga) .. controls +(-150:0.3cm) and +(down:1cm) .. (-0.7,0.7);
\draw (ga) .. controls +(30:0.3cm) and +(down:0.4cm) .. (0.4,0.7);
\draw[midarrow_rev={0.7}] (ga) .. controls +(-60:0.3cm) and +(up:0.3cm) .. (0,-0.7);
\node at (0.15,-0.5) {\tiny $i$};
\node[above] at (-0.7,0.7) {$X^*$};
\node[above] at (0.4,0.7) {$X$};
\node[below] at (-0.15,-0.7) {$I(\one)$};
\end{tikzpicture}
, \hspace{10pt}
\coev_{(X,\ga)} :=
\sum_i \frac{\sqrt{d_i}}{\sqrt{\cD}}
\begin{tikzpicture}
\node[small_morphism] (gav) at (0,-0.1) {\tiny $\ga$};
\draw[midarrow_rev={0.7}] (gav) .. controls +(120:0.3cm) and +(down:0.3cm) .. (0,0.7);
\node at (-0.15,0.5) {\tiny $i$};
\draw (gav) .. controls +(-150:0.3cm) and +(up:0.4cm) .. (-0.4,-0.7);
\draw (gav) .. controls +(30:0.3cm) and +(up:1cm) .. (0.7,-0.7);
\draw[overline, midarrow={0.8}] (gav) .. controls +(-60:0.5cm) and +(down:1cm) .. (0.3,0.7);
\node at (0.4,0.5) {\tiny $i$};
\node[above] at (0.15,0.7) {$I(\one)$};
\node[below] at (-0.4,-0.7) {$X$};
\node[below] at (0.7,-0.7) {$X^*$};
\end{tikzpicture}
\end{equation}
and similarly for right duals;
\item the pivotal structure is given by that of $\cA$.
\end{itemize}
\label{t:ZA_reduced}
\end{proposition}
(Compare \ocite{wassermansymm}*{Theorem 24}.)

\begin{proof}
The inverses of the unit constraints are given by reflecting the diagrams
vertically. For example, we have
\begin{equation*}
\sum_{i,j} \frac{\sqrt{d_i} \sqrt{d_j}}{\cD}
\begin{tikzpicture}
\draw (0,-1) -- (0,1);
\node[small_morphism] (ga) at (0,0.3) {\tiny $\ga$};
\draw[midarrow_rev] (ga) to[out=-170,in=-45] (-1,1);
\draw[overline] (ga) to[out=10,in=-45] (-0.5,1);
\node at (-0.8,0.4) {\tiny $i$};
\node[small_morphism] (ga2) at (0,-0.3) {\tiny $\ga$};
\draw[midarrow] (ga2) to[out=170,in=45] (-1,-1);
\draw[overline] (ga2) to[out=-10,in=45] (-0.5,-1);
\node at (-0.8,-0.4) {\tiny $j$};
\end{tikzpicture}
=
\sum_{i,j} \frac{\sqrt{d_i} \sqrt{d_j}}{\cD}
\begin{tikzpicture}
\draw (0,-1) -- (0,1);
\node[small_morphism] (a1) at (-0.8,0) {\tiny $\al$};
\node[small_morphism] (a2) at (0.4,0) {\tiny $\al$};
\draw[midarrow_rev] (a1) to[out=90,in=-90] (-1,1);
\node at (-1.1,0.5) {\tiny $i$};
\draw[midarrow] (a1) to[out=-90,in=90] (-1,-1);
\node at (-1.1,-0.5) {\tiny $j$};
\draw[midarrow={0.3},overline] (a2) to[out=90,in=-90] (-0.5,1);
\node at (0.4,0.5) {\tiny $i$};
\draw[midarrow_rev={0.3},overline] (a2) to[out=-90,in=90] (-0.5,-1);
\node at (0.4,-0.5) {\tiny $j$};
\draw[regular] (a1) -- (a2);
\node[small_morphism] at (0,0) {\tiny $\ga$};
\end{tikzpicture}
=
\sum_{i,j} \frac{\sqrt{d_i} \sqrt{d_j}}{\cD}
\begin{tikzpicture}
\node[small_morphism] (b1) at (-0.8,-0.2) {\tiny $\beta$};
\node[small_morphism] (b2) at (0.4,-0.2) {\tiny $\beta$};
\node[small_morphism] (ga) at (0,0.2) {\tiny $\ga$};
\draw[regular] (b1) to[out=160,in=180] (ga);
\draw[regular] (b2) to[out=20,in=0] (ga);
\draw[overline,midarrow_rev] (b1) to[out=90,in=-90] (-1,1);
\node at (-1.05,0.5) {\tiny $i$};
\draw[midarrow] (b1) to[out=-95,in=90] (-1,-1);
\node at (-1.05,-0.5) {\tiny $j$};
\draw (ga) -- (0,1);
\draw (ga) -- (0,-1);
\draw[overline,midarrow={0.4}] (b2) .. controls +(up:1cm) and +(down:0.5cm) .. (-0.5,1);
\node at (0.4,0.5) {\tiny $i$};
\draw[overline,midarrow_rev={0.3}] (b2) to[out=-90,in=90] (-0.5,-1);
\node at (0.4,-0.6) {\tiny $j$};
\end{tikzpicture}
=
Q_{\Gamma,\ga}
=
\id_{I(\one) \tnsrbar (X,\ga)}
\;\;.
\end{equation*}

The last two forms of $\ga^\vee$ are equivalent
from the following consideration:
pulling the diagonal strands across the $\ga^*$ morphism
accumulates a (left) right-hand twist,
i.e. (inverse) Drinfeld morphism,
and these cancel out by the naturality of half-braidings.

It is also easy to check that the (co)evaluation maps
described have the desired properties.
The only potentially confusing thing is that
the ``empty strand" in graphical calculus is no longer
the unit of $(\ZA, \tnsrbar)$,
so the unit object and unit constraints have to be explicitly included.
For example, (all dot nodes represent $\ga$,
and a sum over $i,j \in \Irr(\cA)$ is implicit)

\begin{equation*}
\ev_{(X,\ga)} \circ \coev_{(X,\ga)}
=
\frac{d_i d_j}{\cD^2}
\begin{tikzpicture}
\draw (1,0) -- (1,1.2) node[above] {\small $X$};
\node[dotnode] (g1) at (1,1) {};
\node[dotnode] (g2) at (0.2,0.3) {};
\draw (0,0) to[out=90,in=-160] (g2);
\draw (g2) to[out=20,in=90] (0.5,0);
\draw[midarrow] (g1) .. controls +(150:0.2cm) and +(120:0.2cm) .. (g2);
\draw[overline] (g2) .. controls +(-60:0.2cm) and +(-120:0.2cm) ..
  (0.6,0.5) .. controls +(60:0.3cm) and +(-30:0.4cm) .. (g1);
\node at (0.5,1) {\tiny $i$};
\draw (0,0) -- (0,-1.2);
\node[dotnode] (g3) at (0,-1) {};
\node[dotnode] (g4) at (0.8,-0.3) {};
\draw[midarrow] (g3) .. controls +(-30:0.2cm) and +(-60:0.2cm) .. (g4);
\draw (g4) .. controls +(120:0.2cm) and +(60:0.2cm) ..
  (0.4,-0.5) .. controls +(-120:0.3cm) and +(150:0.4cm) .. (g3);
\draw (1,0) to[out=-90,in=20] (g4);
\draw[overline] (g4) to[out=-160,in=-90] (0.5,0);
\draw[overline] (g3) -- (0,0);
\node at (0.5,-1) {\tiny $j$};
\end{tikzpicture}
=
\frac{d_i d_j}{\cD^2}
\begin{tikzpicture}
\draw (0,-1.2) .. controls +(up:1.2cm) and +(down:1.2cm) .. (1,1.2)
  node[above] {\small $X$};
\node[dotnode] (g1) at (0.93,0.7) {};
\node[dotnode] (g2) at (0.7,0.25) {};
\node[dotnode] (g3) at (0.3,-0.25) {};
\node[dotnode] (g4) at (0.07,-0.7) {};
\draw (g2) .. controls +(160:0.5cm) and +(160:0.5cm) .. (g4);
\draw[midarrow=0.2] (g4) .. controls +(-20:0.5cm) and +(-20:0.5cm) .. (g2);
\draw[overline, midarrow=0.2] (g1) .. controls +(160:0.5cm) and +(160:0.5cm) .. (g3);
\draw[overline] (g1) .. controls +(-20:0.5cm) and +(-20:0.5cm) .. (g3);
\node at (0.8, 0.9) {\tiny $i$};
\node at (0.3, -0.9) {\tiny $j$};
\end{tikzpicture}
=
\frac{1}{\cD^2}
\begin{tikzpicture}
\draw (0,-1.2) to (0,1.2) node[above] {\small $X$};
\draw[regular] (0,0.6) circle (0.4cm);
\draw[regular] (0,-0.6) circle (0.4cm);
\node[dotnode] at (0,1) {};
\node[dotnode] at (0,0.2) {};
\node[dotnode] at (0,-0.2) {};
\node[dotnode] at (0,-1) {};
\end{tikzpicture}
=
\frac{1}{\cD^2}
\begin{tikzpicture}
\draw (0.5,-1.2) to (0.5,1.2) node[above] {\small $X$};
\draw[regular] (0,0.6) circle (0.4cm);
\draw[regular] (0,-0.6) circle (0.4cm);
\end{tikzpicture}
=
\id_{(X,\ga)}
\;\;.
\end{equation*}

As for the pivotal structure, it is straightforward
to check that $\ga^{\vee \vee} = \ga^{**}$.
\end{proof}

\begin{remark}
$\tnsrbar$ is not braided;
indeed, we will see later
(e.g. \eqnref{e:cAA_not_symm})
that there can be objects
$(X,\ga),(Y,\mu) \in \ZA$ such that
$(X,\ga) \tnsrbar (Y,\mu) \not\simeq 0$
and $(Y,\mu) \tnsrbar (X,\ga) \simeq 0$.
In particular,
this necessitates the ``multi" in multifusion.
\rmkend
\end{remark}

When considering the usual monoidal structure on $\ZA$,
the forgetful functor $F: \ZA \to \cA$
is naturally a tensor functor, but its adjoint,
$I : \cA \to \ZA$ is not.
With the reduced tensor product, however,
$F$ is clearly not tensor,
but $I$ is:

\begin{proposition}
\label{p:I_tensor}
The functor $I$ from \prpref{p:center} is a tensor functor.
More precisely, for $X,Y \in \cA$, define the morphism
\begin{equation*}
J_{X,Y} : I(X) \tnsrbar I(Y) \to I(X \tnsr Y)
\end{equation*}
by
\begin{equation*}
J_{X,Y} :=
\sqrt{d_i} \sqrt{d_j} \sqrt{d_k}
\begin{tikzpicture}
\node[above] at (-0.4,0.7) {\tiny $I(X)$};
\node[above] at (0.4,0.7) {\tiny $I(Y)$};
\node[below] at (0,-0.7) {\tiny $I(X\tnsr Y)$};
\draw (0.4,0.7) .. controls +(down:0.5cm) and +(up:0.5cm) .. (0.1,-0.7);
\node[small_morphism] (al1) at (-0.6,0) {\tiny $\alpha$};
\node[small_morphism] (al2) at (0.6,0) {\tiny $\alpha$};
\draw[midarrow_rev] (al1) .. controls +(100:0.3cm) and +(down:0.3cm) .. (-0.5,0.7);
\node at (-0.7,0.4) {\tiny $i$};
\draw[midarrow] (al1) .. controls +(down:0.3cm) and +(up:0.3cm) .. (-0.2,-0.7);
\node at (-0.6,-0.4) {\tiny $j$};
\draw[midarrow_rev={0.9}] (al1) .. controls +(45:0.3cm) and +(down:0.3cm) .. (0.3,0.7);
\draw[overline] (-0.4,0.7) .. controls +(down:0.5cm) and +(up:0.5cm) .. (-0.1,-0.7);
\draw[midarrow={0.8}] (al2) .. controls +(125:0.3cm) and +(down:0.3cm) .. (0.5,0.7);
\draw[midarrow_rev] (al2) .. controls +(down:0.3cm) and +(up:0.3cm) .. (0.2,-0.7);
\node at (0.6,-0.4) {\tiny $j$};
\draw[thin_overline={1.5},midarrow={0.9}]
  (al2) .. controls +(70:0.3cm) and +(down:0.3cm) .. (-0.3,0.7);
\node at (0.65,0.6) {\tiny $k$};
\end{tikzpicture}
\;\;.
\end{equation*}

Then
\[
(I,J) : (\cA, \tnsr) \to (\ZA,\tnsrbar)
\]
is a pivotal tensor functor.
\end{proposition}

\begin{proof}
Straightforward computations;
we refer the reader to \ocite{tham}*{Prop 3.19},
where we prove essentially the same result
(the only thing new here is the pivotal structure).
\end{proof}

Observe that for a half-braiding $\ga$ on $X$,
$(\rdual{\ga})^\vee$ is also a half-braiding on $X$.
This gives us an anti-tensor automorphism of $\ZA$:

\begin{proposition}
\label{p:aut_ZA}
There is a tensor equivalence
\begin{align*}
U: (\ZA,\tnsrbar) &\simeq_\tnsr (\ZA,\tnsrbar^\op)
\\
(X,\ga) &\mapsto (X,(\rdual{\ga})^\vee)
\;\;.
\end{align*}

Furthermore, $\thetabar: U^2 \simeq \id$,
so $U$ generates a $\ZZ/2$-action on $\ZA$
(but not tensor action).
\end{proposition}

\begin{proof}
Denote $\wdtld{\ga} := (\rdual{\ga})^\vee$.

First we check that $U(\onebar) \simeq \onebar$.
Recall $\onebar = (\bigoplus X_i X_i^*, \Gamma)$.
We have (sum over $i,j\in \Irr(\cA)$ is implicit)
\[
\wdtld{\Gamma} =
\sqrt{d_i}\sqrt{d_j}
\begin{tikzpicture}
\tikzmath{
  \tp = 0.8;
  \bt = -0.8;
  \lf = -0.3;
  \rt = 0.3;
}
\node[small_morphism] (a1) at (\lf,0) {\tiny $\al$};
\node[small_morphism] (a2) at (\rt,0) {\tiny $\al$};
\draw (a2) .. controls +(30:0.5cm) and +(down:0.3cm) ..
  (-0.7,\tp);
\draw[overline={1.5}] (a1) -- (\lf,\tp)
  node[above] {\small $X_i$};
\draw (a1) -- (\lf,\bt)
  node[below] {\small $X_j$};
\draw[overline={1.5}] (a2) -- (\rt,\tp)
  node[above] {\small $X_i^*$};
\draw (a2) -- (\rt,\bt)
  node[below] {\small $X_j^*$};
\draw[overline={1.5}] (a1) .. controls +(-150:0.5cm) and +(up:0.3cm) ..
  (0.7,\bt);
\end{tikzpicture}
=
\sqrt{d_i}\sqrt{d_j}
\begin{tikzpicture}
\tikzmath{
  \tp = 0.8;
  \bt = -0.8;
  \lf = -0.3;
  \rt = 0.3;
}
\node[small_morphism] (a1) at (\lf,0) {\tiny $\al$};
\node[small_morphism] (a2) at (\rt,0) {\tiny $\al$};
\draw (a2) .. controls +(120:0.3cm) and +(down:0.3cm) ..
  (-0.7,\tp);
\coordinate (mm) at (-0.5,0);
\draw (a1) .. controls +(45:0.4cm) and +(up:0.4cm) .. (mm);
\draw[overline={1}] (a1) -- (\lf,\tp)
  node[above] {\small $X_i$};
\draw (a1) -- (\lf,\bt)
  node[below] {\small $X_j$};
\draw (a2) -- (\rt,\tp)
  node[above] {\small $X_i^*$};
\draw (a2) -- (\rt,\bt)
  node[below] {\small $X_j^*$};
\draw[overline={1.5}] (mm) .. controls +(down:0.5cm) and +(up:0.3cm) ..
  (0.7,\bt);
\end{tikzpicture}
=
\sqrt{d_i}\sqrt{d_j}
\begin{tikzpicture}
\tikzmath{
  \tp = 0.8;
  \bt = -0.8;
  \lf = -0.3;
  \rt = 0.3;
}
\node[small_morphism] (a1) at (\lf,0) {\tiny $\al$};
\node[small_morphism] (a2) at (\rt,0) {\tiny $\al$};
\draw (a2) .. controls +(120:0.8cm) and +(down:0.2cm) ..
  (-0.7,\tp);
\draw[overline={1.5}] (a1) -- (\lf,\tp)
  node[above] {\small $X_i$};
\draw (a1) -- (\lf,\bt)
  node[below] {\small $X_j$};
\draw (a2) -- (\rt,\tp)
  node[above] {\small $X_i^*$};
\draw (a2) -- (\rt,\bt)
  node[below] {\small $X_j^*$};
\draw[overline={1.5}] (a1) .. controls +(-45:0.3cm) and +(up:0.3cm) ..
  (0.7,\bt);
\node[small_morphism] at (\lf,0.35) {\tiny $\thetabar$};
\node[small_morphism] at (\lf,-0.4) {\tiny $\theta$};
\end{tikzpicture}
\]
where recall $\theta$ is the twist operator,
and $\thetabar$ is its inverse.
Then it is easy to see that
$c^\inv \circ (\thetabar \tnsr \id) : U(\onebar) \simeq \onebar$
is an isomorphism.

Next we construct the tensor structure on $U$.
Let us compute $U((X,\ga)) \tnsrbar^\op U((Y,\mu))$.
$Y \tnsrproj{\wdtld{\mu}}{\wdtld{\ga}} X$
is the image of the projection
$Q_{\wdtld{\mu},\wdtld{\ga}}$,
which equals
\[
Q_{\wdtld{\mu},\wdtld{\ga}}
=
\frac{1}{\cD}
\begin{tikzpicture}
\draw (0.3,-0.8) -- (0.3,0.8)
  node[above] {$X$};
\draw[regular, overline] (0,0) circle (0.5cm); 
\node[small_morphism] (n2) at (0.3,0.4) {\tiny $\wdtld{\ga}$};
\draw[overline] (-0.3,-0.8) -- (-0.3,0.8)
  node[above] {$Y$};
\node[small_morphism] (n1) at (-0.3,-0.4) {\tiny $\wdtld{\mu}$};
\end{tikzpicture}
=
\frac{1}{\cD}
\begin{tikzpicture}
\draw[overline] (-0.3,-0.8) -- (-0.3,0.8)
  node[above] {$Y$};
\draw[regular, overline] (0,0) circle (0.5cm); 
\draw[overline] (0.3,-0.8) -- (0.3,0.8)
  node[above] {$X$};
\node[small_morphism] (n2) at (0.3,-0.4) {\tiny $\ga$};
\node[small_morphism] (n1) at (-0.3,0.4) {\tiny $\mu$};
\end{tikzpicture}
\;\;.
\]
Observe that the braiding gives an isomorphism
$c_{Y,X} : Y \tnsrproj{\wdtld{\mu}}{\wdtld{\ga}} X
\simeq X \tnsrproj{\ga}{\mu} Y$,
and in fact intertwines the half-braidings
$\wdtld{\mu} \tnsrbar \wdtld{\ga}$ and
$\wdtld{\ga \tnsrbar \mu}$ (see \eqnref{e:hfbrd_sym}):

\[
c_{Y,X} \circ (\wdtld{\mu} \tnsrbar \wdtld{\ga}) =
\begin{tikzpicture}
\begin{scope}[shift={(0,0.25)}]
\draw (0.3,-1) -- (0.3,1) node[above] {$X$};
\draw[regular, overline] (0,0) circle (0.5cm); 
\node[small_morphism] (n2) at (0.3,0.4) {\tiny $\wdtld{\ga}$};
\draw[overline] (-0.3,-1) -- (-0.3,1) node[above] {$Y$};
\node[small_morphism] (n1) at (-0.3,-0.4) {\tiny $\wdtld{\mu}$};
\node[small_morphism] (a1) at (-0.5,0) {\tiny $\al$};
\draw (a1) to[out=150,in=-90] (-1,1);
\node[small_morphism] (a2) at (0.5,0) {\tiny $\al$};
\draw (a2) to[out=-30,in=90] (1,-1.5);
\draw (-0.3,-1) .. controls +(down:0.2cm) and +(up:0.2cm) .. (0.3,-1.5);
\draw[overline={1.5}] (0.3,-1) .. controls +(down:0.2cm) and +(up:0.2cm) .. (-0.3,-1.5);
\end{scope}
\end{tikzpicture}
=
\begin{tikzpicture}
\begin{scope}[shift={(0,0.25)}]
\draw[overline={1.5}] (-0.3,-1) -- (-0.3,1)
  node[above] {$Y$};
\draw[regular, overline] (0,0) circle (0.5cm); 
\draw[overline={1.5}] (0.3,-1) -- (0.3,1)
  node[above] {$X$};
\node[small_morphism] (n2) at (0.3,-0.4) {\tiny $\ga$};
\node[small_morphism] (n1) at (-0.3,0.4) {\tiny $\mu$};
\node[small_morphism] (al1) at (0,0.5) {\tiny $\al$};
\node[small_morphism] (al2) at (0,-0.5) {\tiny $\al$};
\draw[overline={1.5}]
  (al2) .. controls +(down:0.3cm) and +(up:0.8cm) .. (1,-1.5);
\draw (al1) .. controls +(up:0.3cm) and +(down:0.3cm) .. (-1,1);
\draw[overline={1.5}] (-0.3,1) -- (-0.3,0.5);
\draw (-0.3,-1) .. controls +(down:0.2cm) and +(up:0.2cm) .. (0.3,-1.5);
\draw[overline={1.5}] (0.3,-1) .. controls +(down:0.2cm) and +(up:0.2cm) .. (-0.3,-1.5);
\end{scope}
\end{tikzpicture}
=
\begin{tikzpicture}
\begin{scope}[shift={(0,-0.25)}]
\node[small_morphism] (a2) at (0.5,0) {\tiny $\al$};
\draw (a2) .. controls +(30:0.3cm) and +(0:0.7cm) ..
  (0,0.8) .. controls +(180:0.5cm) and +(down:0.2cm) .. (-1,1.5);
\draw[overline={1.5}] (0.3,-1) -- (0.3,1);
\draw[regular, overline] (0,0) circle (0.5cm); 
\node[small_morphism] (n2) at (0.3,0.4) {\tiny $\mu$};
\draw[overline] (-0.3,-1) -- (-0.3,1);
\node[small_morphism] (n1) at (-0.3,-0.4) {\tiny $\ga$};
\node[small_morphism] (a1) at (-0.5,0) {\tiny $\al$};
\node[small_morphism] (a2) at (0.5,0) {\tiny $\al$};
\draw[overline={1.5}]
  (a1) .. controls +(-150:0.3cm) and +(180:0.7cm) ..
  (0,-0.8) .. controls +(0:0.5cm) and +(up:0.2cm) .. (1,-1);
\draw (-0.3,1.5) .. controls +(down:0.2cm) and +(up:0.2cm) .. (0.3,1)
  node[pos=0,above] {$Y$};
\draw[overline={1.5}] (0.3,1.5) .. controls +(down:0.2cm) and +(up:0.2cm) .. (-0.3,1)
  node[pos=0,above] {$X$};
\end{scope}
\end{tikzpicture}
=
\wdtld{\ga \tnsrbar \mu} \circ c_{Y,X}
\;\;.
\]
It follows easily that
$W_{(X,\ga),(Y,\mu)} := c_{Y,X}$
is a tensor structure on $U$.

Finally, observe that
$\wdtld{\wdtld{\ga}} = (\theta_X \tnsr \id)
\circ \ga \circ (\id \tnsr \thetabar_X)$,
as already hinted at in the computation of
$\wdtld{\Gamma}$ above.
So $\thetabar_X:
U^2((X,\ga)) = (X,\wdtld{\wdtld{\ga}})
\simeq (X,\ga)$ is our desired natural isomorphism
$U^2 \simeq \id$.
This is in fact a natural isomorphism of
tensor functors, due to the identity
$\thetabar_{X\tnsr Y} \circ
(c_{Y,X} \circ c_{X,Y})
= \thetabar_X \tnsr \thetabar_Y$.
\end{proof}

We also note that the above proposition holds true
for $\ZA$ with the standard tensor product.
We will not use this fact in the rest of the paper,
but we state it anyway for completeness:

\begin{proposition}
There is a tensor equivalence
\begin{align*}
U: (\ZA,\tnsr) &\simeq_\tnsr (\ZA,\tnsr^\op)
\\
(X,\ga) &\mapsto (X,(\rdual{\ga})^\vee)
\;\;.
\end{align*}

\end{proposition}

\begin{proof}
Essentially the same as before;
even the tensor structure and natural isomorphism
are the same.
We just show one computation,
checking $c_{Y,X}$ intertwines
$\wdtld{\mu} \tnsr \wdtld{\ga}$ and $\wdtld{\ga \tnsr \mu}$:
\[
c_{Y,X} \circ (\wdtld{\mu} \tnsr \wdtld{\ga}) =
\begin{tikzpicture}
\tikzmath{
  \tp = 0.8;
  \bt = -0.8;
  \lf = -0.3;
  \rt = 0.3;
}
\draw
  (\lf,\tp) -- (\lf,0)
  .. controls +(down:0.4cm) and +(up:0.4cm) .. (\rt,\bt);
\draw[overline={1.5}]
  (\rt,\tp) -- (\rt,0)
  .. controls +(down:0.4cm) and +(up:0.4cm) .. (\lf,\bt);
\coordinate (mu) at (\lf,0.5);
\coordinate (ga) at (\rt,0.3);
\draw (-0.7,\tp) .. controls +(down:0.1cm) and +(-150:0.3cm) ..
  (mu) -- (ga)
  .. controls +(-10:0.3cm) and +(up:0.8cm) .. (0.7,\bt);
\node[small_morphism] at (mu){\tiny $\wdtld{\mu}$};
\node[small_morphism] at (ga) {\tiny $\wdtld{\ga}$};
\end{tikzpicture}
=
\begin{tikzpicture}
\tikzmath{
  \tp = 0.8;
  \bt = -0.8;
  \lf = -0.3;
  \rt = 0.3;
}
\draw
  (\lf,\tp) .. controls +(down:0.4cm) and +(up:0.4cm) ..
  (\rt,0) -- (\rt,\bt);
\coordinate (mu) at (\rt,-0.3);
\coordinate (ga) at (\lf,-0.5);
\draw[overline={1.5}]
  (-0.7,\tp) .. controls +(down:0.8cm) and +(150:0.3cm) ..
  (mu) .. controls +(-30:0.2cm) and +(-30:0.8cm) ..
  (0,-0.4) .. controls +(150:0.8cm) and +(150:0.2cm) ..
  (ga) .. controls +(-30:0.5cm) and +(up:0.2cm) ..
  (0.7,\bt);
\draw[overline={1.5}]
  (\rt,\tp) .. controls +(down:0.4cm) and +(up:0.4cm) ..
  (\lf,0) -- (\lf,\bt);
\node[small_morphism] at (mu){\tiny $\wdtld{\mu}$};
\node[small_morphism] at (ga) {\tiny $\wdtld{\ga}$};
\end{tikzpicture}
=
\begin{tikzpicture}
\tikzmath{
  \tp = 0.8;
  \bt = -0.8;
  \lf = -0.3;
  \rt = 0.3;
}
\coordinate (mu) at (\rt,-0.2);
\coordinate (ga) at (\lf,-0.4);
\draw (mu) -- (\rt,\bt);
\draw (ga) -- (\lf,\bt);
\draw[overline={1.5}]
  (-0.7,\tp) .. controls +(down:0.8cm) and +(30:0.8cm) ..
  (mu) -- (ga)
  .. controls +(-150:0.8cm) and +(up:0.2cm) ..
  (0.7,\bt);
\draw[overline={1.5}]
  (\lf,\tp) .. controls +(down:0.4cm) and +(up:0.4cm) ..
  (\rt,0) -- (mu);
\draw[overline={1.5}]
  (\rt,\tp) .. controls +(down:0.4cm) and +(up:0.4cm) ..
  (\lf,0) -- (ga);
\node[small_morphism] at (mu){\tiny $\wdtld{\mu}$};
\node[small_morphism] at (ga) {\tiny $\wdtld{\ga}$};
\end{tikzpicture}
=
\wdtld{\ga \tnsr \mu} \circ c_{Y,X}
\;\;.
\]
\end{proof}

\begin{remark} \label{r:2fold}
In \ocite{wasserman2fold},
Wasserman showed that the Drinfeld center
of a symmetric $\cA$ is a ``bilax 2-fold
tensor category", which loosely means
a category with two monoidal structures that
almost commute. More precisely,
it consists of a pair of natural transformations
\[
\eta : ((X,\ga) \tnsr (X',\ga')) \tnsrbar
  ((Y,\mu) \tnsr (Y',\mu'))
\rightleftharpoons
((X,\ga) \tnsrbar (Y,\mu)) \tnsr
  ((X',\ga') \tnsrbar (Y',\mu'))
: \zeta
\]
such that $\eta \circ \zeta = \id$,
together with several morphisms
relating the units $\one$ and $\onebar$,
and they satisfy a cocktail of compatibility axioms.
We claim that $\ZA$ also has such a
structure when $\cA$ is not symmetric.
The only difference to the structure maps is
where we have to distinguish the braiding $c$
from its inverse:
$\eta = \id_X \tnsr c^\inv_{X',Y} \tnsr \id_{Y'}$
and
$\zeta = \id_X \tnsr c_{Y,X'} \tnsr \id_{Y'}$
(with the various projections implicit).
The proof that they satisfy the various compatibility
axioms is a lengthy calculation that we do not share
here.
There is a more topological approach,
see \rmkref{r:2fold_cy}.
We do not prove this claim in this paper,
as it will take us too far afield.
\rmkend
\end{remark}


\section{Horizontal Trace of $\cA$}
\label{s:htr}

In this section, we describe another way to obtain the
reduced tensor product,
via an intermediate category between $\cA$ and $\ZA$
known as the \emph{horizontal trace} of $\cA$,
as defined in \ocite{BHLZ}*{Section 2.4},
which is a generalization of Ocneanu's tube algebra
\ocite{Ocn}.

The horizontal trace of a braided category
is naturally a tensor category.
The Karoubi envelope of the horizontal trace of $\cA$
is equivalent to the Drinfeld center of $\cA$.
The main result of this section is \prpref{p:hA_ZA_tnsr},
which states that this equivalence
is also a tensor equivalence.

We begin with an exposition on the horizontal trace
mostly taken from \ocite{KT},
where we also considered a minor generalization
to bimodule categories $\cM$;
here we only consider $\cM = \cA$.

\begin{definition}
\label{d:hA}
Consider $\cA$ as a bimodule category over itself
by left and right multiplication.
Its \emph{horizontal trace},
denoted $\htr(\cA)$ or simply $\hA$,
is the category with the following objects and morphisms:

Objects: same as in $\cA$

Morphisms: $\Hom_\hA(X,X') := \bigoplus_A \ihom{\cA}{A}(X,X')/\sim$,
where $\ihom{\cA}{A}(X,X') := \Hom_\cA(A \tnsr X, X' \tnsr A)$,
the sum is over all objects $A\in \cA$,
and $\sim$ is the equivalence relation generated by 
the following:

For any $\psi\in \ihom{\cA}{B,A}(X,X') := \Hom_\cA(B \tnsr X,X' \tnsr A)$
  and $f\in \Hom_\cA(A,B)$, we have 

\begin{equation}
\ihom{\cA}{A}(X,X') \ni
\begin{tikzpicture}
\node[small_morphism] (psi) at (0,0) {\small $\psi$};
\draw (psi) -- +(0,0.8) node[above] {$X$};
\draw (psi) -- +(0,-0.8) node[below] {$X'$};
\coordinate (L) at (-0.8,0.6);
\coordinate (R) at (0.8,-0.6);
\node[left] at (L) {$A$};
\node[right] at (R) {$A$};
\draw (L) -- (psi);
\draw (psi) -- (R);
\node[small_morphism] at (-0.48,0.36) {\tiny $f$};
\node at (-0.35,0.05) {\tiny $B$};
\end{tikzpicture}
\quad \sim \quad
\begin{tikzpicture}
\node[small_morphism] (psi) at (0,0) {\small $\psi$};
\draw (psi) -- +(0,0.8) node[above] {$X$};
\draw (psi) -- +(0,-0.8) node[below] {$X'$};
\coordinate (L) at (-0.8,0.6);
\coordinate (R) at (0.8,-0.6);
\node[left] at (L) {$B$};
\node[right] at (R) {$B$};
\draw (L) -- (psi);
\draw (psi) -- (R);
\node[small_morphism] at (0.48,-0.36) {\tiny $f$};
\node at (0.35,-0.1) {\tiny $A$};
\end{tikzpicture}
\in \ihom{\cA}{B}(X,X')
\end{equation}
or in other words, $\Hom_\hA(X,X') =
\displaystyle\int^A \ihom{\cA}{A}(X,X')$.
\defend
\end{definition}

\begin{definition}
\label{d:htr}
Let $\htr: \cA \to \hA$
be the natural inclusion functor
that is identity on objects,
and on morphisms is the natural map
$\Hom_\cA(X,X') = \ihom{\cA}{\one}(X,X') \to \Hom_\hA(X,X')$.
\defend
\end{definition}

The adjective ``inclusion" further justified by the
following proposition
(along with the fact that $I$ is faithful),
which implies $\htr$ is faithful also.

\begin{proposition}{\ocite{KT}*{Theorem 3.9}}
\label{p:hA_ZA}
Let $G: \hA \to \ZA$ be defined as follows:
on objects, $G(X) = I(X)$, and on morphisms,
for $\psi \in \ihom{\cA}{A}(X,X')$,

\begin{equation}
G(\psi) = 
\sum_{i,j \Irr(\cA)} \sqrt{d_i}\sqrt{d_j}
\begin{tikzpicture}
\node[small_morphism] (psi) at (0,0) {$\psi$}; 
\node[small_morphism] (L) at (-0.7,0) {$\al$};
\node[small_morphism] (R) at (0.7,0) {$\al$};
\draw (psi)-- +(0,1) node[above] {$G(X)$};
\draw (psi)-- +(0,-1) node[below] {$G(X')$}; 
\draw[midarrow_rev] (L)-- +(0,1) node[pos=0.5,left] {\tiny $i$};
\draw[midarrow] (L) -- +(0,-1) node[pos=0.5,left] {\tiny $j$}; 
\draw[midarrow] (R) -- +(0,1) node[pos=0.5,right] {\tiny $i$};
\draw[midarrow_rev] (R)-- +(0,-1) node[pos=0.5,right] {\tiny $j$};
\draw[midarrow={0.7}] (L) -- (psi);
\node at (-0.4, -0.2) {\tiny $A$};
\draw[midarrow_rev={0.7}] (R) -- (psi);
\node at (0.4, -0.2) {\tiny $A$};
\end{tikzpicture}
\;\;.
\end{equation}

Then the extension to the Karoubi envelope
is an equivalence of abelian categories:
\[
  \Kar(G) : \Kar(\hA) \simeq \ZA
	\;\;.
\]
Under this equivalence, the natural functor $\htr: \cA \to \hA$
is identified with $I: \cA \to \ZA$,
i.e. we have the commutative diagram
\begin{equation}
\label{e:comm_diag}
\begin{tikzcd}
\cA \ar[r,"\htr"] \ar[d,"I"']
& \hA \ar[ld,"G"'] \ar[d,"\Kar"]
\\
\ZA
& \Kar(\hA) \ar[l,"\Kar(G)"] \ar[l,"\simeq"']
\end{tikzcd}
\;\;.
\end{equation}
\end{proposition}

We note that in \ocite{KT}, the proposition
establishes an equivalence
$\Kar(\htr(\cM)) \simeq \cZ(\cM)$,
where $\htr(\cM)$ is the horizontal trace
of an $\cA$-bimodule category,
and $\cZ(\cM)$ is the center of $\cM$
\ocite{GNN}*{Definition 2.1}
which is analogous to the Drinfeld center.

It is useful to construct an inverse to $\Kar(G)$:

\begin{proposition}
\label{p:G_inv}
An inverse to $\Kar(G)$ is given by:
\begin{align*}
\Kar(G)^\inv : \ZA &\simeq \Kar(\hA) \\
(X,\ga) &\mapsto (X, \hP_\ga)
\end{align*}
where
\[
\hP_\ga := \sum_{i \in \Irr(\cA)} \frac{d_i}{\cD} \ga_{X_i} =
\frac{1}{\cD}
\begin{tikzpicture}
\node[small_morphism] (ga) at (0,0) {\tiny $\ga$};
\draw (ga) -- (0,0.6) node[above] {\small $X$};
\draw (ga) -- (0,-0.6);
\draw[regular] (ga) -- +(135:0.6cm);
\draw[regular] (ga) -- +(-45:0.6cm);
\end{tikzpicture}
\]
and on morphisms, for $f\in \Hom_\ZA((X,\ga),(Y,\mu))$,
\[
f \mapsto \hP_\mu \circ f = f \circ \hP_\ga
\;\;.
\]
\end{proposition}

\begin{proof}
Straightforward.
\end{proof}

When $\cA$ is premodular,
in particular when $\cA$ is \emph{braided},
$\hA$ has a natural monoidal structure.
This was discussed in \ocite{KT}*{Example 8.2};
here we spell it out more explicitly.

\begin{proposition}
There is a tensor product on $\hA$,
denoted $\htnsr$:
on objects, it is simply the same as $\cA$,
and on morphisms,

\begin{equation*}
\begin{tikzpicture}
\node at (0,3) {$\htnsr: $};
\end{tikzpicture}
\begin{tikzpicture}
\node at (0,3) {$\ihom{\cA}{A}(X,X')$};
\node at (0,2) {$X$};
\node at (0,0) {$X'$};
\node[small_morphism] (ph) at (0,1) {$\ph$};
\draw (ph) -- (0,0.2);
\draw (ph) -- (0,1.8);
\draw (ph) -- (-0.5,1.5) node[pos=1,above left] {$A$};
\draw (ph) -- (0.5,0.5) node[pos=1,below right] {$A$};
\end{tikzpicture}
\begin{tikzpicture}
\node at (0,1) {$\boxtimes$};
\node at (0,3) {$\boxtimes$};
\end{tikzpicture}
\begin{tikzpicture}
\node at (0,3) {$\ihom{\cA}{B}(Y,Y')$};
\node at (0,2) {$Y$};
\node at (0,0) {$Y'$};
\node[small_morphism] (ph) at (0,1) {\small $\psi$};
\draw (ph) -- (0,0.2);
\draw (ph) -- (0,1.8);
\draw (ph) -- (-0.5,1.5) node[pos=1,above left] {$B$};
\draw (ph) -- (0.5,0.5) node[pos=1,below right] {$B$};
\end{tikzpicture}
\begin{tikzpicture}
\node at (0,1) {$\mapsto$};
\node at (0,3) {$\longrightarrow$};
\end{tikzpicture}
\begin{tikzpicture}
\node at (0,3) {$\ihom{\cA}{AB}(XY, X'Y'$)};
\node at (0,2) {$XY$};
\node at (0,0) {$X'Y'$};
\node[small_morphism] (ph1) at (-0.2,0.9) {\tiny $\ph$};
\node[small_morphism] (ph2) at (0.2,1.1) {\tiny $\psi$};
\draw (ph2) -- (0.2,1.8);
\draw (ph2) -- (0.2,0.2);
\draw (ph2) -- +(150:1cm);
\draw (ph2) -- +(-30:0.8cm);
\draw[overline] (ph1) -- (-0.2,1.8);
\draw (ph1) -- (-0.2,0.2);
\draw (ph1) -- +(150:0.8cm);
\draw[overline] (ph1) -- +(-30:1cm);
\node at (-1.3,1.7) {$AB$};
\node at (1.3,0.3) {$AB$};
\end{tikzpicture}
\;\;.
\end{equation*}

The unit object is the same as that of $\cA$.

Furthermore, this tensor product structure
is compatible with the rigid and pivotal structures
of $\cA$.

In other words, this is an
extension of the tensor product on $\cA$,
so that the inclusion functor $\htr: \cA \to \hA$
is a pivotal tensor functor.
\end{proposition}

\begin{proof}
Straightforward.
\end{proof}

It is useful to work out the left dual of a morphism
(the right dual being similar):
\begin{equation}
\label{e:hA_dual}
\begin{tikzpicture}
\node at (0,1.5) {$\ihom{\cA}{A}(X,X')$};
\node[small_morphism] (ph) at (0,0) {\small $\ph$};
\draw (ph) -- (0,0.6) node[above] {$X$};
\draw (ph) -- (0,-0.6) node[below] {$X'$};
\draw (ph) -- (-0.6,0.3) node[left] {$A$};
\draw (ph) -- (0.6,-0.3) node[right] {$A$};
\end{tikzpicture}
\begin{tikzpicture}
\node at (0,1.5) {$\to$};
\node at (0,0) {$\mapsto$};
\end{tikzpicture}
\begin{tikzpicture}
\node at (0,1.5) {$\ihom{\cA}{A}(X'^*,X^*)$};
\node[small_morphism] (ph) at (0,0) {\tiny $\ph$};
\draw (ph) .. controls +(up:0.3cm) and +(up:0.3cm) ..
  (0.2,0) .. controls +(down:0.2cm) and +(up:0.2cm) ..
  (0,-0.6)
  node[below] {$X^*$};
\draw (ph) -- (-0.6,0.3) node[left] {$A$};
\draw[overline={1.5}] (ph) -- (0.6,-0.3)
  node[right] {$A$};
\draw[overline={1.25}]
  (ph) .. controls +(down:0.3cm) and +(down:0.3cm) ..
  (-0.2,0) .. controls +(up:0.2cm) and +(down:0.2cm) ..
  (0,0.6)
  node[above] {$X'^*$};
\end{tikzpicture}
\;\;.
\end{equation}

In particular, when $X' = X$, and $\ph=\ga$ is a half-braiding on $X$,
the right side is nothing but $\ga^\vee$.

The following proposition is an upgrade of
\prpref{p:hA_ZA}:

\begin{proposition}
\label{p:hA_ZA_tnsr}
When $\ZA$ is endowed with the reduced tensor product,
the equivalence in \prpref{p:hA_ZA} is an equivalence
of pivotal multifusion categories.
More precisely, let $J$ be the tensor structure on $I$
from \prpref{p:I_tensor}.
Then
\[
(G,J) : (\hA, \htnsr) \to (\ZA, \tnsrbar)
\]
is a pivotal tensor functor,
and hence its completion to $\Kar(\hA)$
is a pivotal tensor equivalence.
\end{proposition}

\begin{proof}
As \prpref{p:I_tensor} takes care of objects,
it remains to check naturality of $J$ with respect to morphisms
of $\hA$. This is easy to do:
for example, for $\ph \in \ihom{\cA}{A}(X,X')$,
we see that
$J_{X',Y} \circ (I(\ph) \tnsrbar I(\id_Y))
=
I(\ph \tnsr \id_Y) \circ J_{X,Y}$ by

\begin{equation*}
\sqrt{d_i} \sqrt{d_j} \sqrt{d_k} d_l
\begin{tikzpicture}
\tikzmath{ \toprow = 1; \bottom = -1; }
\node[above] at (0,\toprow) {\tiny $I(X) \tnsrbar I(Y)$};
\node[below] at (0,\bottom) {\tiny $I(X'\htnsr Y)$};
\node[small_morphism] (be1) at (-0.6,0) {\tiny $\beta$};
\node[small_morphism] (be2) at (0.6,0) {\tiny $\beta$};
\node[small_morphism] (al1) at (-0.7,0.7) {\tiny $\al$};
\node[small_morphism] (al2) at (-0.1,0.7) {\tiny $\al$};
\node[small_morphism] (ph) at (-0.4,0.7) {\tiny $\ph$};
\draw (0.4,\toprow) .. controls +(down:0.5cm) and +(up:0.5cm) .. (0.1,\bottom);
\draw (be1) .. controls +(100:0.3cm) and +(down:0.3cm) .. (al1);
\node at (-0.8,0.4) {\tiny $l$};
\draw (be1) .. controls +(down:0.3cm) and +(up:0.3cm) .. (-0.3,\bottom);
\node at (-0.6,-0.5) {\tiny $j$};
\draw (be1) .. controls +(45:0.3cm) and +(down:0.3cm) .. (0.3,\toprow);
\draw (ph) -- (-0.4,\toprow);
\draw[thin_overline={1.5}] (ph) .. controls +(down:0.5cm) and +(up:0.5cm) .. (-0.1,\bottom);
\draw (be2) .. controls +(125:0.3cm) and +(down:0.3cm) .. (0.5,\toprow);
\draw (be2) .. controls +(down:0.3cm) and +(up:0.3cm) .. (0.3,\bottom);
\draw[thin_overline={1.5}]
  (be2) .. controls +(70:0.3cm) and +(down:0.3cm) .. (al2);
\node at (0.6,0.6) {\tiny $k$};
\draw (al1) -- (-0.7,\toprow);
\node at (-0.8,0.9) {\tiny $i$};
\draw (al2) -- (-0.1,\toprow);
\draw (al1) -- (ph);
\draw (ph) -- (al2);
\end{tikzpicture}
=
\sqrt{d_i} \sqrt{d_j} \sqrt{d_k}
\begin{tikzpicture}
\tikzmath{ \toprow = 1; \bottom = -1; }
\node[above] at (0,\toprow) {\tiny $I(X) \tnsrbar I(Y)$};
\node[below] at (0,\bottom) {\tiny $I(X'\htnsr Y)$};
\node[small_morphism] (ga1) at (-0.6,0) {\tiny $\ga$};
\node[small_morphism] (ga2) at (0.6,0) {\tiny $\ga$};
\node[small_morphism] (ph) at (-0.2,0) {\tiny $\ph$};
\draw (0.4,\toprow) .. controls +(down:0.5cm) and +(up:0.5cm) .. (0.1,\bottom);
\draw (ga1) .. controls +(100:0.3cm) and +(down:0.3cm) .. (-0.6,\toprow);
\node at (-0.75,0.5) {\tiny $i$};
\draw (ga1) .. controls +(down:0.3cm) and +(up:0.3cm) .. (-0.3,\bottom);
\node at (-0.6,-0.5) {\tiny $j$};
\draw (ga1) .. controls +(45:0.3cm) and +(down:0.3cm) .. (0.3,\toprow);
\draw[thin_overline={1.5}] (ph) .. controls +(100:0.3cm) and +(down:0.5cm) .. (-0.4,\toprow);
\draw (ph) .. controls +(-80:0.5cm) and +(up:0.5cm) .. (-0.1,\bottom);
\draw (ga2) .. controls +(125:0.3cm) and +(down:0.3cm) .. (0.5,\toprow);
\draw (ga2) .. controls +(down:0.3cm) and +(up:0.3cm) .. (0.3,\bottom);
\draw[thin_overline={1.5}]
  (ga2) .. controls +(70:0.3cm) and +(down:0.3cm) .. (-0.2,\toprow);
\node at (0.6,0.6) {\tiny $k$};
\draw (ga1) -- (ph);
\draw[thin_overline={1.5}] (ph) -- (ga2);
\node at (0.1,0.1) {\tiny $A$};
\end{tikzpicture}
=
\sqrt{d_i} \sqrt{d_j} \sqrt{d_k} d_l
\begin{tikzpicture}
\tikzmath{ \toprow = 1; \bottom = -1; }
\node[above] at (0,\toprow) {\tiny $I(X) \tnsrbar I(Y)$};
\node[below] at (0,\bottom) {\tiny $I(X'\htnsr Y)$};
\node[small_morphism] (be1) at (-0.6,0) {\tiny $\beta$};
\node[small_morphism] (be2) at (0.6,0) {\tiny $\beta$};
\node[small_morphism] (al1) at (-0.4,-0.7) {\tiny $\al$};
\node[small_morphism] (al2) at (0.4,-0.7) {\tiny $\al$};
\node[small_morphism] (ph) at (-0.1,-0.7) {\tiny $\ph$};
\draw (0.4,\toprow) .. controls +(down:0.5cm) and +(up:0.5cm) .. (0.1,\bottom);
\draw (be1) .. controls +(100:0.3cm) and +(down:0.3cm) .. (-0.6,\toprow);
\node at (-0.75,0.5) {\tiny $i$};
\draw (be1) .. controls +(down:0.3cm) and +(up:0.3cm) .. (al1);
\node at (-0.6,-0.45) {\tiny $l$};
\draw (be1) .. controls +(45:0.3cm) and +(down:0.3cm) .. (0.3,\toprow);
\draw[thin_overline={1.5}] (ph) .. controls +(up:0.3cm) and +(down:0.5cm) .. (-0.4,\toprow);
\draw (ph) -- (-0.1,\bottom);
\draw (be2) .. controls +(125:0.3cm) and +(down:0.3cm) .. (0.5,\toprow);
\draw (be2) .. controls +(down:0.3cm) and +(up:0.3cm) .. (al2);
\draw[thin_overline={1.5}]
  (be2) .. controls +(70:0.3cm) and +(down:0.3cm) .. (-0.2,\toprow);
\node at (0.6,0.6) {\tiny $k$};
\draw (al1) -- (-0.4,\bottom);
\node at (-0.5,-0.9) {\tiny $j$};
\draw (al2) -- (0.4,\bottom);
\draw (al1) -- (ph);
\draw[thin_overline={1.5}] (ph) -- (al2);
\end{tikzpicture}
\end{equation*}
using \eqnref{e:al_natural} and \eqnref{e:combine}.
\end{proof}

\begin{corollary}
The inverse functor $\Kar(G)^\inv$ defined in
\prpref{p:G_inv}
is naturally a tensor functor.
\end{corollary}
\begin{proof}
This follows from the general fact that
(weak) inverses to monoidal functors are automatically monoidal.
We provide some details for concreteness.

The tensor structure on $\Kar(G)^\inv$
is a natural isomorphism
\[
(X,\hP_\ga) \tnsr (Y,\hP_\mu)
= (XY,\hP_\ga \tnsr \hP_\mu)
\simeq (X\tnsrproj{\ga}{\mu} Y,\hP_{\ga\tnsrbar\mu})
\]
which is given by the natural projection
$Q_{\ga,\mu} : XY \to X\tnsrproj{\ga}{\mu} Y$.
This follows from two simple observations:
\[
\hP_\ga \tnsr \hP_\mu = \hP_{\ga\tnsrbar\mu}
\in \End_\hA(XY)
\]
(using \eqnref{e:combine}),
which implies $(XY,\hP_\ga \tnsr \hP_\mu)
= (XY,\hP_{\ga\tnsrbar\mu})$,
and
\[
\hP_{\ga\tnsrbar\mu} =
\hP_\ga \circ Q_{\ga,\mu} =
\hP_\ga \circ Q_{\ga,\mu} \circ Q_{\ga,\mu} =
\hP_{\ga\tnsrbar\mu} \circ Q_{\ga,\mu}
\in \End_\hA(XY)
\]
(see \eqnref{e:hfbrd_sym})
which implies $Q_{\ga,\mu} : (XY,\hP_{\ga\tnsrbar\mu})
\to (X\tnsrproj{\ga}{,\mu} Y,\hP_{\ga\tnsrbar\mu})$
is an isomorphism.
It is easy to check that this satisfies
the hexagon axiom for tensor structure.

The isomorphism $\Kar(G)^\inv (\onebar)
\simeq \one$ is given by
\begin{equation}
\label{e:unit_isom_G}
\sum_{i\in \Irr(\cA)}
\frac{\sqrt{d_i}}{\sqrt{\cD}}
\begin{tikzpicture}
\draw (-0.2,0.5) .. controls +(down:0.5cm) and +(-30:0.5cm) .. (-0.6,0.3)
  node[pos=0,above] {\small $X_i$};
\draw (0.2,0.5) .. controls +(down:0.5cm) and +(150:0.5cm) .. (0.6,-0.3)
  node[pos=0,above] {\small $X_i^*$};
 \node at (0,-0.5) {\small $\one$};
\end{tikzpicture}
\;\;.
\end{equation}
\end{proof}

Next we consider a $\ZZ/2$-action,
related to \prpref{p:aut_ZA}.
Let $\wdtld{\cdot}: \ihom{\cA}{A}(X,X')
\to \ihom{\cA}{A^*}(X,X')$
be the map
\[
\psi \mapsto \wdtld{\psi} :=
\begin{tikzpicture}
\node[small_morphism] (psi) at (0,0) {\tiny $\psi$};
\draw (psi) -- (0,-0.5) node[below] {$X'$};
\draw (psi) .. controls +(-30:0.25cm) and +(-10:0.4cm) ..
  (0,0.25) -- +(170:0.4cm)
  node[left] {$A^*$};
\draw[overline={1.5}]
  (psi) .. controls +(150:0.25cm) and +(170:0.4cm) ..
  (0,-0.25) -- +(-10:0.4cm)
  node[right] {$A^*$};
 \draw[overline={1.5}] (psi) -- (0,0.5)
 	node[above] {$X$};
\end{tikzpicture}
\;\;.
\]
This is very similar to the definition of $\wdtld{\ga}$
in proof of \prpref{p:aut_ZA}.
Indeed, if $\ga$ is a half-braiding on $A$,
then $\ga_A \in \ihom{\cA}{A}(X,X)$,
and $\wdtld{\ga_A} = (\wdtld{\ga})_{A^*}$.

This defines an endomorphism
$\wdtld{\cdot} : \Hom_\hA(X,X') \simeq
\Hom_\hA(X,X')$,
and in fact is part of an
automorphism of $\hA$:
(compare \prpref{p:aut_ZA})

\begin{proposition}
\label{p:aut_hA}
Let $\hU : \hA \to \hA$ be the functor
that is identity on objects,
and acts by $\wdtld{\cdot}$ on morphisms.
Then there is a tensor equivalence
\[
(\hU,c) : (\hA,\htnsr) \simeq_\tnsr (\hA, \htnsr^\op)
\;\;.
\]

Furthermore,
$\thetabar : \hU^2 \simeq \id$,
so $\hU$ generates a $\ZZ/2$-action on $\hA$.
(but not tensor action).

The tensor equivalence and $\ZZ/2$-action
naturally extends to the Karoubi envelope.
\end{proposition}

\begin{proof}
Similar to \prpref{p:aut_ZA}.
We note again $\thetabar : \hU^2 \simeq \id$
is a natural isomorphism of tensor functors.
It is also useful to observe that
$\Kar(\hU)((X,\hP_\ga)) = (X,\hP_{\wdtld{\ga}})$.
\end{proof}

\begin{proposition}
$\Kar(G) : \Kar(\hA) \simeq \ZA$
respects the $\ZZ/2$-actions.
More precisely, there is a natural isomorphism
$u: U \circ \Kar(G) \simeq \Kar(G) \circ \hU$
that makes $\Kar(G)$ a $\ZZ/2$-equivariant functor.
\end{proposition}

\begin{proof}
The natural isomorphism is given as follows:
for $(X,p) \in \Kar(\hA)$,
$u_{(X,p)} : \im(G(p)) \to \im(G(\wdtld{p}))$
is given by
\[
u_{(X,p)} =
\begin{tikzpicture}
\tikzmath{
  \a = 0.4;
  \b = -0.4;
  \lf = -0.3;
  \rt = 0.3;
}
\draw (\rt,\a) .. controls +(down:0.2cm) and +(up:0.4cm) .. (\lf,\b)
  node[pos=0,above] {\small $i^*$}
  node[below] {\small $i^*$};
\draw[overline={2}] (0,\a) -- (0,\b)
  node[pos=0,above] {\small $X$}
  node[below] {\small $X$};
\draw[overline={1}] (\lf,\a) .. controls +(down:0.2cm) and +(up:0.4cm) .. (\rt,\b)
  node[pos=0,above] {\small $i$}
  node[below] {\small $i$};
\node[small_morphism] at (0.25,-0.2) {\tiny $\thetabar$};
\end{tikzpicture}
\]
where the $i^*$ strand is at the bottom,
with the projections $G(p), G(\wdtld{p})$
implicit.
It is easy to check
$u_{(X,p)} \circ G(p) = G(\wdtld{p}) \circ u_{(X,p)}$
- intuitively, $u$ ``drags'' the middle strand around,
introducing a half-twist to turn $p$ into $\wdtld{p}$,
so this is an isomorphism.
It is also easy to check that it intertwines
the half-braidings $\wdtld{\Gamma}$ and $\Gamma$
(essentially the same computation for
$U(\onebar) \simeq \onebar$ in the proof of
\prpref{p:aut_ZA}).

Next we need to check the following diagram of
natural isomorphisms commutes:
\[
\begin{tikzcd}
U^2 G \ar[r,"Uu"] \ar[d,"\thetabar G"]
& U G \hU \ar[r,"u\hU"]
& G \hU^2 \ar[d,"G(\thetabar)"]
\\
G \ar[rr,equal]
& & G
\end{tikzcd}
\;\;.
\]
The left vertical arrow is a full inverse twist
on three strands, and commutativity
follows from the three strand version of the identity
relating braiding and twist.
\end{proof}

\section{Crane-Yetter theory on the Annulus}
\label{s:cy}

As mentioned in the introduction,
we were motivated by the study of
an extended Crane-Yetter theory $\ZCY$,
in particular its value on surfaces.
There the annulus played an important role,
so this section is devoted to relating
the previous sections to $\ZCY(\Ann)$;
more precisely, we describe a tensor product
on $\ZCY(\Ann)$,
and prove the main result of the paper,
\thmref{t:main},
which states that
$\ZCY(\Ann) \simeq \ZA$ as pivotal multifusion categories.

We review the definition of and some properties
about $\ZCY(\Sigma)$ from \ocite{KT},
to which we refer the reader for more details.
For most of this section, we will take 
$\Sigma = \Ann = S^1 \times (0,1)$,
the annulus.

A \emph{boundary value} on a surface $\Sigma$
is a configuration $B$ of finitely many framed points,
where each point $b\in B$ is labeled with an object
$V_b \in \cA$
(here framed point means a trivialization of its normal
bundle, or simply a tangent vector).
We denote such a boundary value by $(B,\{V_b\})$.

An \emph{$\cA$-colored ribbon graph} $\Gamma$
in an oriented 3-manifold $M$
is a ribbon graph in $M$ where each oriented edge $\ee$
is assigned an object $V_\ee \in \cA$
(with opposite orientations of the same edge
assigned mutually dual objects),
and each vertex is assigned some morphism
in $\Hom(\one, V_{\ee_1} \tnsr \cdots \tnsr V_{\ee_k})$,
$\ee_i$'s taken with outgoing orientation,
and edges ordered arbitrarily.
(Some care should be taken when identifying
$\Hom$'s for different choices of orderings,
since $\cA$ may not be symmetric.
See \ocite{KT}*{Definition 5.1}
for a more careful approach using the
Reshetikhin-Turaev invariant \ocite{RT}.)
Such graphs are allowed to meet the boundary transversally,
in which case it leaves an ``imprint'' on the boundary,
$\del \Gamma = \Gamma \cap \del M$.
Note that the coloring and framing of $\Gamma$
induce colorings and framings on $\del \Gamma$,
so we call $\del \Gamma$
the \emph{boundary value of $\Gamma$}.


Given a boundary value $(B,\{V_b\})$ on $\del M$,
one can consider the vector space generated by
$\cA$-colored ribbon graphs $\Gamma$ in $M$ with
boundary value $\del \Gamma = (B, \{V_b\})$,
modulo isotopy and certain local relations.
These local relations essentially arise from
taking a small ball $D$ in $M$,
``evaluating" $\Gamma \cap D$ using
the Reshetikhin-Turaev invariant \ocite{RT},
and replacing $\Gamma \cap D$ with another graph
with the same evaluation.
The resulting quotient vector space is called
the \emph{skein module of $M$ with boundary value $(B,\{V_b\})$},
denoted $\Skein(M;(B,\{V_b\}))$.
We will mostly consider $M = \Sigma \times [0,1]$,
where $\Sigma \times \{0\}$ is thought of as the
incoming boundary,
and $\Sigma \times \{1\}$ the outgoing boundary.

\begin{definition}
\label{d:CYA}
$\hZSig{}$ is the category with the following objects
and morphisms:

Objects: boundary values on $\Sigma$,

Morphisms: for boundary values $\VV = (B,\{V_b\}), \VV'$,
\[
\Hom_{\hZSig{}}(\VV,\VV')
:= \Skein(\Sigma \times [0,1]; \overline{\VV},\VV')
\]
is the skein module of $\Sigma \times [0,1]$
with boundary value $\overline{\VV} = (B,\{V_b^*\})$
on the
incoming boundary $\Sigma \times \{0\}$,
and boundary value $\VV'$ on the
outgoing boundary $\Sigma \times \{1\}$.

$\ZSig{}$ is defined to be the
Karoubi envelope of $\hZSig{}$.
\defend
\end{definition}

$\Ann$, as the product of something with the interval,
has a ``stacking" structure,
in that we can include two annuli $S^1 \times (0,1/2)$
and $S^1 \times (1/2,1)$ in $\Ann$.
In particular, two objects $\VV_1, \VV_2$ can be included
side by side in a bigger annulus, and similarly morphisms
$f:\VV_1 \to \VV_1'$, $g:\VV_2 \to \VV_2'$
can be placed side by side.
The unit object with respect to this tensor product
is the empty configuration $\onest = (\emptyset, \{\})$.
This gives us a stacking tensor product structure on $\hCYA$
that is in fact rigid and pivotal:

\begin{proposition}[\ocite{KT}*{Proposition 6.7}]
The stacking tensor product described above,
denoted by $\tnsrst$,
makes $\hCYA$ a pivotal multifusion category:
\begin{itemize}
\item For an object $\VV = (B, \{V_b\})$,
its left dual is the object $\VV^\vee
:= (\theta(B), \{V_b^\vee\})$,
where $\theta$ is the operation on $\Ann$ that flips $(0,1)$,
and has the following evaluation and coevaluation maps:
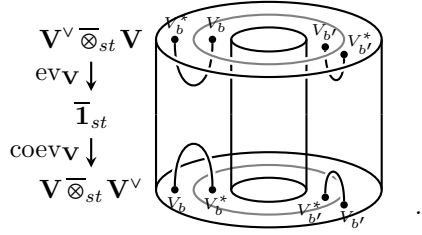
\[
\begin{tikzpicture}
\node (a1) at (0,0) {$\VV^\vee \tnsrst \VV$};
\node (a2) at (0,-1) {$\onest$};
\node (a3) at (0, -2) {$\VV \tnsrst \VV^\vee$};
\draw[->] (a1) -- (a2) node[left,pos=0.5] {$\ev_\VV$};
\draw[->] (a2) -- (a3) node[left,pos=0.5] {$\coev_\VV$};
\end{tikzpicture}
\begin{tikzpicture}
\draw (0,0) ellipse (0.5cm and 0.16cm);
\node[dotnode] (Vb) at (-0.75, 0) {};
\node[dotnode] (Vbp) at (-1.25, 0) {};
\node[dotnode] (Vb1) at (0.75, -0.1) {};
\node[dotnode] (Vb1p) at (1, -0.2) {};
\draw (Vb) .. controls +(down:0.8cm) and +(down:0.8cm) .. (Vbp);
\draw (Vb1) .. controls +(down:0.5cm) and +(down:0.5cm) .. (Vb1p);
\draw (-0.5, 0) -- (-0.5, -1);
\draw (0.5, 0) -- (0.5, -1);
\draw[overline={1}, gray] (0,0) ellipse (1cm and 0.33cm);
\draw[overline={1}] (0,0) ellipse (1.5cm and 0.5cm);
\draw (-1.5, 0) -- (-1.5, -1);
\draw (1.5, 0) -- (1.5, -1);
\node at (-0.7, 0.2) {\tiny $V_b$};
\node at (-1.22,0.2) {\tiny $V_b^*$};
\node at (0.75, 0.1) {\tiny $V_{b'}$};
\node at (1.22, -0.05) {\tiny $V_{b'}^*$};
\draw (0,-2) ellipse (0.5cm and 0.16cm);
\node[dotnode] (Vbp) at (-0.75, -2) {};
\node[dotnode] (Vb) at (-1.25, -2) {};
\node[dotnode] (Vb1p) at (0.75, -2.1) {};
\node[dotnode] (Vb1) at (1, -2.2) {};
\draw[gray] (0,-2) ellipse (1cm and 0.33cm);
\draw (0,-2) ellipse (1.5cm and 0.5cm);
\draw[overline={1}] (Vb) .. controls +(up:0.8cm) and +(up:0.8cm) .. (Vbp);
\draw[overline={1}] (Vb1) .. controls +(up:0.5cm) and +(up:0.5cm) .. (Vb1p);
\draw (-0.5, -1) -- (-0.5, -2);
\draw (0.5, -1) -- (0.5, -2);
\draw (-1.5, -1) -- (-1.5, -2);
\draw (1.5, -1) -- (1.5, -2);
\node at (-0.7, -2.2) {\tiny $V_b^*$};
\node at (-1.22,-2.2) {\tiny $V_b$};
\node at (0.55, -2.3) {\tiny $V_{b'}^*$};
\node at (1.12, -2.4) {\tiny $V_{b'}$};
\end{tikzpicture}
\;\;\;
\begin{tikzpicture}
\node at (0,-2.3) {.};
\end{tikzpicture}
\]
(The gray lines indicate $S^1 \times \{1/2\}$
which separates $\VV^\vee$ from $\VV$
in the tensor product, and play no role in defining
these morphisms.)
We think of the outside half of the annulus as the left side
in the tensor product.
The right dual is $\rvee{\VV}:= (\theta(B), \{\rdual{V_b}\})$,
with essentially the same (co)evaluation maps.

\item The pivotal map
$\delta_\VV : \VV \to \VV^{\vee\vee}$
is simply the graph with vertical strands running down,
and one node on each strand labeled with $\delta_{V_b}$.
\end{itemize}

$\CYA$, as the Karoubi envelope of $\hCYA$,
naturally inherits these structures.
\end{proposition}

It is not hard to see that the left dual of a morphism
is obtained by turning the solid annulus ``inside-out";
for example,

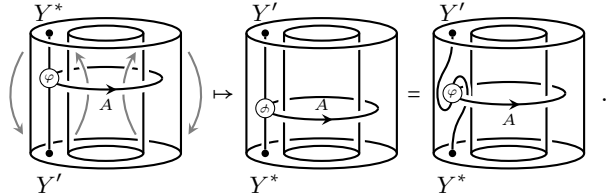
\begin{equation*}
\begin{tikzpicture}
\tikzmath{
  \toprow = 0.8;
  \bottom = -0.8;
  \midrow = 0.2;
  \bigx = 1; 
  \bigy = 0.2;
  \smlx = 0.5;
  \smly = 0.1;
  \medx = 0.75;
  \medy = 0.15;
}
\draw (0,\bottom) ellipse (\bigx cm and \bigy cm);
\draw (\medx,\midrow) arc(0:180:\medx cm and \medy cm);
\draw[overline={1.5}] (-0.5,\bottom) -- (-0.5,\toprow);
\draw[overline={1.5}] (0.5,\bottom) -- (0.5,\toprow);
\draw (0,\bottom) ellipse (\smlx cm and \smly cm); 
\draw[gray,overline={1},->] (-0.4,-0.55)
  to[out=60,in=-60] (-0.4,0.55);
\draw[gray,overline={1},->] (0.4,-0.55)
  to[out=120,in=-120] (0.4,0.55);
\draw[gray,->] (-1.1,0.55)
  to[out=-120,in=120] (-1.1,-0.55);
\draw[gray,->] (1.1,0.55)
  to[out=-60,in=60] (1.1,-0.55);
\node[dotnode] (top) at (-\medx,\toprow) {}; 
\node[dotnode, label={[shift={(0.,-0.6)}] \small $Y'$}] (bottom) at (-\medx,\bottom) {};
\draw[overline={1.5}, midarrow_rev]
  (\medx,\midrow) arc(0:-180:\medx cm and \medy cm)
  node[pos=0.5, below] {\tiny $A$};
\draw[overline={1}] (top) -- (bottom);
\node[small_morphism] at (-\medx,\midrow) {\tiny $\ph$};
\draw[overline={1}] (0,\toprow) ellipse (\bigx cm and \bigy cm);
\draw (0,\toprow) ellipse (\smlx cm and \smly cm);
\draw (-1,\bottom) -- (-1,\toprow);
\draw (1,\bottom) -- (1,\toprow);
\node[label={\small $Y^*$}] at (-\medx,\toprow) {};
\end{tikzpicture}
\mapsto
\begin{tikzpicture}
\tikzmath{
  \toprow = 0.8;
  \bottom = -0.8;
  \midrow = -0.2;
  \bigx = 1; 
  \bigy = 0.2;
  \smlx = 0.5;
  \smly = 0.1;
  \medx = 0.75;
  \medy = 0.15;
}
\draw (0,\bottom) ellipse (\bigx cm and \bigy cm);
\draw (\medx,\midrow) arc(0:180:\medx cm and \medy cm);
\draw[overline={1.5}] (-0.5,\bottom) -- (-0.5,\toprow);
\draw[overline={1.5}] (0.5,\bottom) -- (0.5,\toprow);
\draw (0,\bottom) ellipse (\smlx cm and \smly cm); 
\node[dotnode] (top) at (-\medx,\toprow) {}; 
\node[dotnode, label={[shift={(0,-0.6)}] \small $Y^*$}] (bottom) at (-\medx,\bottom) {};
\draw[overline={1.5}, midarrow_rev]
  (\medx,\midrow) arc(0:-180:\medx cm and \medy cm)
  node[pos=0.5, above] {\tiny $A$};
\draw[overline={1}] (top) -- (bottom);
\node[small_morphism] at (-\medx,\midrow)
  {\tiny $\rotatebox[origin=c]{180}{\(\ph\)}$};
\draw[overline={1}] (0,\toprow) ellipse (\bigx cm and \bigy cm);
\draw (0,\toprow) ellipse (\smlx cm and \smly cm);
\draw (-1,\bottom) -- (-1,\toprow);
\draw (1,\bottom) -- (1,\toprow);
\node[label={\small $Y'$}] at (-\medx,\toprow) {};
\end{tikzpicture}
=
\begin{tikzpicture}
\tikzmath{
  \toprow = 0.8;
  \bottom = -0.8;
  \midrow = 0;
  \bigx = 1; 
  \bigy = 0.2;
  \smlx = 0.5;
  \smly = 0.1;
  \medx = 0.75;
  \medy = 0.15;
}
\draw (0,\bottom) ellipse (\bigx cm and \bigy cm);
\draw (\medx,\midrow) arc(0:180:\medx cm and \medy cm);
\draw[overline={1.5}] (-0.5,\bottom) -- (-0.5,\toprow);
\draw[overline={1.5}] (0.5,\bottom) -- (0.5,\toprow);
\draw (0,\bottom) ellipse (\smlx cm and \smly cm); 
\node[dotnode] (top) at (-\medx,\toprow) {}; 
\node[dotnode, label={[shift={(0.,-0.6)}] \small $Y^*$}] (bottom) at (-\medx,\bottom) {};
\draw[overline={1}]
  (top) .. controls +(down:0.5cm) and +(up:0.3cm) ..
  (-0.95,\midrow) .. controls +(down:0.3cm) and +(down:0.3cm) ..
  (-\medx,\midrow) .. controls +(up:0.3cm) and +(up:0.3cm) ..
  (-0.55,\midrow) .. controls +(down:0.3cm) and +(up:0.5cm) .. (bottom);
\draw[overline={1.5}, midarrow_rev]
  (\medx,\midrow) arc(0:-180:\medx cm and \medy cm)
  node[pos=0.5, below] {\tiny $A$};
\node[small_morphism] at (-\medx,\midrow) {\tiny $\ph$};
\draw[overline={1}] (0,\toprow) ellipse (\bigx cm and \bigy cm);
\draw (0,\toprow) ellipse (\smlx cm and \smly cm);
\draw (-1,\bottom) -- (-1,\toprow);
\draw (1,\bottom) -- (1,\toprow);
\node[label={\small $Y'$}] at (-\medx,\toprow) {};
\end{tikzpicture}
\;\;.
\end{equation*}

The gray arrows indicate the ``inside-out" operation.
\footnote{Imagine pulling your hand out of the sleeve of
a tight sweater, so that
the sleeve will be pulled inside out.
The second diagram is
inside-out the same way the sleeve is.}
Note the upside-down `$\ph$' in the second diagram;
the last diagram can be turned into the second
by pulling on the upward and downward strands,
forcing the $\ph$ node to turn upside-down.

The last diagram is reminiscent of duals in $\hA$
(see \eqnref{e:hA_dual}).
In particular, when $Y' = Y$,
and $\ph$ is a half-braiding on $Y$,
the extra bending in the last diagram can be incorporated into
the node to become $\ph^\vee$ (see
\prpref{p:reduced_pivotal}).

In \ocite{KT}, in Example 8.2, we provided an explicit equivalence
$H = H_p:\hA \simeq \hCYA$, where $p\in \Ann$, given as follows:
\footnote{
$H_p$ is dependent on the choice of a point $p\in \Ann$,
but all $H_p$ are naturally isomorphic (by non-unique
natural isomorphism).
One can consider the full subcategory with objects
of the form $(\{p\},\{X\})$. This is strictly speaking not
a tensor subcategory, since the tensor product of such objects
would have two marked points.
However, one can put a different tensor product,
$(\{p\},\{X\}) \tnsrst' (\{p\},\{Y\}) = (\{p\},\{XY\})$,
and the inclusion can be made a pivotal tensor equivalence.
}

\begin{equation}
\label{e:def_H}
\begin{tikzpicture}
\node at (0,2) {$\hA$};
\node at (0,1.2) {$Y$};
\node at (0,-1.2) {$Y'$};
\node[small_morphism] (ph) at (0,0) {$\vphi$};
\draw (ph) -- (0,-1);
\draw (ph) -- (0,1);
\draw (ph) -- (-0.5,0.5) node[left] {$A$};
\draw (ph) -- (0.5,-0.5) node[right] {$A$};
\end{tikzpicture}
\begin{tikzpicture}
\node at (0,0) {$\mapsto$};
\node at (0,-1.2) {$\mapsto$};
\node at (0,1.2) {$\mapsto$};
\node at (0,2) {$\overset{H}{\simeq}$};
\end{tikzpicture}
\begin{tikzpicture}
\node at (0,2) {$\hCYA$};
\tikzmath{
  \toprow = 0.7;
  \bottom = -0.7;
  \midrow = 0;
  \bigx = 1; 
  \bigy = 0.2;
  \smlx = 0.5;
  \smly = 0.1;
  \medx = 0.75;
  \medy = 0.15;
}
\draw (0,1.2) ellipse (\bigx cm and \bigy cm);
\draw (0,1.2) ellipse (\smlx cm and \smly cm);
\node[dotnode, label={[shift={(0,0.1)}] \small $Y$}] at (-\medx,1.2) {};
\draw (0,-1.2) ellipse (\bigx cm and \bigy cm);
\draw (0,-1.2) ellipse (\smlx cm and \smly cm);
\node[dotnode, label={[shift={(0,-0.6)}] \small $Y'$}] at (-\medx,-1.2) {};
\draw (0,\bottom) ellipse (\bigx cm and \bigy cm);
\draw (\medx,\midrow) arc(0:180:\medx cm and \medy cm);
\draw[overline={1.5}] (-0.5,\bottom) -- (-0.5,\toprow);
\draw[overline={1.5}] (0.5,\bottom) -- (0.5,\toprow);
\draw (0,\bottom) ellipse (\smlx cm and \smly cm); 
\node[dotnode] (top) at (-\medx,\toprow) {}; 
\node[dotnode] (bottom) at (-\medx,\bottom) {};
\draw[overline={1.5}, midarrow_rev]
  (\medx,\midrow) arc(0:-180:\medx cm and \medy cm)
  node[pos=0.5, below] {\tiny $A$};
\draw[overline={1},midarrow={0.3},midarrow={0.9}]
  (top) -- (bottom);
\node[small_morphism] at (-\medx,\midrow) {\tiny $\ph$};
\draw[overline={1}] (0,\toprow) ellipse (\bigx cm and \bigy cm);
\draw (0,\toprow) ellipse (\smlx cm and \smly cm);
\draw (-1,\bottom) -- (-1,\toprow);
\draw (1,\bottom) -- (1,\toprow);
\node[style={fill=white,circle,inner sep=0pt,
  minimum size=0pt, scale=0.8}] at (-0.95,0.4) {\tiny $Y$};
\node[style={fill=white,circle,inner sep=0pt,
  minimum size=0pt, scale=0.8}] at (-0.95,-0.4) {\tiny $Y'$};
\node at (-0.6,1.2) {\tiny $p$};
\node at (-0.6,-1.2) {\tiny $p$};
\end{tikzpicture}
\;\;\;.
\end{equation}

It is not hard to see that this equivalence is in fact tensor
and pivotal.
For example, for $X,Y \in \Obj \hA$,
the tensor product in $\hCYA$, $H(X) \tnsrst H(Y)$,
is an object with two marked points, labeled with $X$ and $Y$.
The tensor structure on $H$ would be a trivalent
graph connecting $H(X) \tnsrst H(Y)$ to $H(XY)$,
with the unique vertex labeled by
$\id_{XY}$
(which is naturally identified with
$\coev \in \Hom_\cA(\one,XY(XY)^*)$
).
The unit object $\onest$ in $\hCYA$, i.e. the empty configuration,
is isomorphic to the object $H(\one) = (\{p\},\{\one\})$.
Hence we have the following:

\begin{theorem}
\label{t:main}
We have the following commutative diagram,
where all functors are pivotal tensor functors:
\begin{equation*}
\begin{tikzcd}
(\cA,\tnsr) \ar[r,"\htr"] \ar[d,"I"']
& (\hA,\htnsr) \ar[ld,"G"'] \ar[d,"\Kar"]
    \ar[r,"H"] \ar[r,"\simeq"']
& (\hCYA,\tnsrst) \ar[d, "\Kar"]
\\
(\ZA,\tnsrbar)
& (\Kar(\hA),\htnsr) \ar[l,"\Kar(G)"] \ar[l,"\simeq"']
    \ar[r, "\Kar(H)"'] \ar[r,"\simeq"]
& (\CYA,\tnsrst)
\end{tikzcd}
\;\;\;
\begin{tikzcd}
\textcolor{white}{.}
\\
.
\end{tikzcd}
\end{equation*}
In other words, the reduced tensor product $\tnsrbar$ on $\ZA$
encodes the stacking tensor product on $\CYA$.
\end{theorem}

Together with \prpref{p:G_inv},
the pivotal tensor equivalence $(\ZA,\tnsrbar) \simeq \CYA$
is given by

\begin{align}
\Kar(H) \circ \Kar(G)^\inv :
  (\ZA,\tnsrbar) &\simeq \CYA
\\
\begin{tikzpicture}
\node (X) at (0,0.65) {$(X,\ga)$};
\node (Y) at (0,-0.65) {$(Y,\mu)$};
\draw[midarrow={0.55}] (X) -- (Y) node[pos=0.5,right] {$f$};
\end{tikzpicture}
&
\begin{tikzpicture}
\node at (0,0.65) {$\mapsto$};
\node at (0,0) {$\mapsto$};
\node at (0,-0.65) {$\mapsto$};
\end{tikzpicture}
\begin{tikzpicture}
\begin{scope}[shift={(0,-1)}]
\draw (0,0) ellipse (1cm and 0.2cm);
\draw[regular] (0.75,1.5) arc(0:180:0.75cm and 0.15cm);
\draw[regular] (0.75,0.5) arc(0:180:0.75cm and 0.15cm);
\draw[overline={1}] (-0.5,0) -- (-0.5,2);
\draw[overline={1}] (0.5,0) -- (0.5,2);
\draw (0,0) ellipse (0.5cm and 0.1cm); 
\node[dotnode] (top) at (-0.75,2) {}; 
\node[dotnode] (bottom) at (-0.75,0) {};
\node at (-0.75,-0.3) {\small $Y$};
\draw[regular, overline={1.5}]
  (0.75,1.5) arc(0:-180:0.75cm and 0.15cm);
\draw[regular, overline={1.5}]
  (0.75,0.5) arc(0:-180:0.75cm and 0.15cm);
\draw[overline={1}] (top) -- (bottom);
\node[small_morphism] at (-0.75,1.5) {\tiny $\ga$};
\node[small_morphism] at (-0.75,1) {\tiny $f$};
\node[small_morphism] at (-0.75,0.5) {\tiny $\mu$};
\draw[overline={1}] (0,2) ellipse (1cm and 0.2cm);
\draw (0,2) ellipse (0.5cm and 0.1cm);
\draw (-1,0) -- (-1,2);
\draw (1,0) -- (1,2);
\node at (-0.75,2.3) {\small $X$};
\draw[decoration={brace,amplitude=3pt},decorate] (-1.2,1.3) -- (-1.2,2);
\node at (-1.55,1.65) {im};
\draw[decoration={brace,amplitude=3pt},decorate] (-1.2,0) -- (-1.2,0.7);
\node at (-1.55,0.35) {im};
\end{scope}
\end{tikzpicture}
\;\;\;.
\end{align}

Once again this functor doesn't send unit to unit;
the isomorphism is given by (compare \eqnref{e:unit_isom_G})

\[
\sum_{i\in \Irr(\cA)}
\frac{\sqrt{d_i}}{\sqrt{\cD}}
\begin{tikzpicture}
\begin{scope}[shift={(0,-0.5)}]
\draw (0,0) ellipse (1cm and 0.2cm);
\draw (0.75,0.5) arc(0:180:0.75cm and 0.15cm);
\draw[overline={1}] (-0.5,0) -- (-0.5,1);
\draw[overline={1}] (0.5,0) -- (0.5,1);
\draw (0,0) ellipse (0.5cm and 0.1cm); 
\node[dotnode] (n1) at (-0.75,1) {}; 
\draw[overline={1.5}, midarrow={0.6}]
  (0.75,0.5) arc(0:-180:0.75cm and 0.15cm);
\node at (0,0.5) {\tiny $i$};
\node[small_morphism] (n2) at (-0.75,0.5) {\tiny $\id$};
\draw (n1) -- (n2);
\draw[overline={1}] (0,1) ellipse (1cm and 0.2cm);
\draw (0,1) ellipse (0.5cm and 0.1cm);
\draw (-1,0) -- (-1,1);
\draw (1,0) -- (1,1);
\node at (-0.75,1.3) {\tiny $X_i X_i^*$};
\end{scope}
\end{tikzpicture}
\;\;
\text{, or more intuitively,}\;\;
\sum_{i\in \Irr(\cA)}
\frac{\sqrt{d_i}}{\sqrt{\cD}}
\begin{tikzpicture}
\begin{scope}[shift={(0,-0.5)}]
\draw (0,0) ellipse (1cm and 0.2cm);
\node[dotnode] (n1) at (-0.75,1) {}; 
\node[dotnode] (n2) at (-0.6,0.97) {};
\draw (n1) .. controls +(down:0.4cm) and +(-160:0.1cm) ..
  (-0.65,0.5) .. controls +(20:0.5cm) and +(up:0.1cm) ..
  (0.75,0.5);
\draw[overline={1}] (-0.5,0) -- (-0.5,1);
\draw[overline={1}] (0.5,0) -- (0.5,1);
\draw (0,0) ellipse (0.5cm and 0.1cm); 
\draw[overline={1}]
  (n2) .. controls +(down:0.5cm) and +(170:0.1cm) ..
  (-0.5,0.4) .. controls +(-10:0.3cm) and +(down:0.15cm) ..
  (0.75,0.5);
\draw[overline={1}] (0,1) ellipse (1cm and 0.2cm);
\draw (0,1) ellipse (0.5cm and 0.1cm);
\draw (-1,0) -- (-1,1);
\draw (1,0) -- (1,1);
\node at (-0.75,1.3) {\tiny $X_i X_i^*$};
\end{scope}
\end{tikzpicture}
\;\;\;.
\]

As an application of \thmref{t:main},
we describe $\ZCY(\torus)$ purely algebraically in terms
of $\cA$.
We can produce $\torus$  from $\Ann$ by
gluing (neighborhoods of) $S^1 \times \{0\}$
and $S^1 \times \{1\}$.
The excision property of $\ZCY$ as stated 
in the introduction doesn't work as is,
but \ocite{KT}*{Theorem 7.5}
actually proves an apparently slightly more general but ultimately
equivalent form of excision,
which allows $\Sigma$ to be obtained by
gluing two boundaries of a single surface $\Sigma'$;
the balanced tensor product is then replaced by
the center of $\ZCY(\Sigma')$
as a $\CYA$-bimodule
(as defined in \ocite{GNN}*{Definition 2.1},
repeated in \ocite{KT}*{Definition 3.1};
for applications here,
it suffices to know that this notion of center
for a monoidal category as a bimodule over itself
coincides with the Drinfeld center.)

We can view $\CYA$ as a bimodule category
over itself, thinking of the left and right actions
as ``insertions" from the left
($S^1 \times \{0\}$) and right ($S^1 \times \{0\}$).
Thus we have the following corollary
of \ocite{KT}*{Theorem 7.5}:
\begin{proposition} \label{p:ZCY_torus}
\[
\ZCY(\torus) \simeq \cZ((\CYA,\tnsrst))
\]
as abelian categories.
\end{proposition}
\begin{proof}
Take $X' = S^1 \times (0,3)$ in \ocite{KT}*{Theorem 7.5}.
\end{proof}

As an immediate corollary of this
and \thmref{t:main}, we have
\begin{corollary} \label{c:torus_ZA}
\[
\ZCY(\torus) \simeq \cZ((\ZA,\tnsrbar))
\]
as abelian categories.
\end{corollary}
We will study $\ZCY(\torus)$ for $\cA = \Rep(G)$
in \secref{s:symm_torus}.\\

Now we can also give a topological interpretation
of the $\ZZ/2$-actions discussed in previous sections
(see \prpref{p:aut_ZA}, \prpref{p:aut_hA}).
An orientation-preserving diffeomorphism
from the annulus to itself naturally
gives an automorphism of $\CYA$.
In particular,
consider the following diffeomorphism $\UCY$:

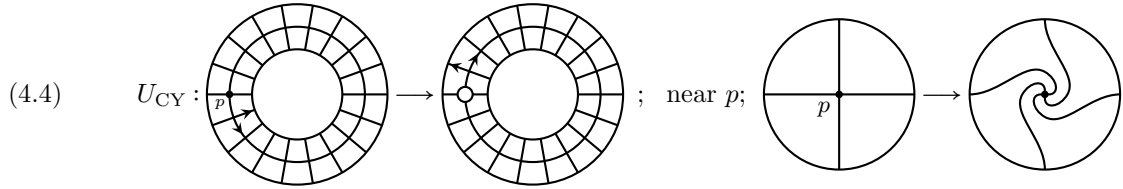
\begin{equation}
\UCY: 
\begin{tikzpicture}
\tikzmath{
  \r1 = 0.6;
  \r2 = 1.2;
  \rmid = (\r1+\r2)/2;
  \a1 = -140;
  \adel = 20;
  \n = 16;
}
\foreach \x in {0,...,\n}
  \draw (\a1+\x*\adel:\r1 cm) -- (\a1+\x*\adel:\r2 cm);
\draw[midarrow_rev={0.3}]
  (-160:\r1 cm) -- (-160:\r2 cm);
\draw (0,0) circle (\r1 cm);
\draw (0,0) circle (\r2 cm);
\draw[midarrow={0.6}] (0,0) circle (\rmid cm);
\node[dotnode] at (-\rmid,0) {};
\node at (-\rmid - 0.125, -0.125) {\tiny $p$};
\end{tikzpicture}
\longrightarrow
\begin{tikzpicture}
\tikzmath{
  \r1 = 0.6;
  \r2 = 1.2;
  \rmid = (\r1+\r2)/2;
  \a1 = 180;
  \adel = 20;
  \n = 16;
}
\foreach \x in {0,...,\n}
  \draw (\a1+\x*\adel:\r1 cm) -- (\a1+\x*\adel:\r2 cm);
\draw[midarrow={0.95}]
  (160:\r1 cm) -- (160:\r2 cm);
\draw (0,0) circle (\r1 cm);
\draw (0,0) circle (\r2 cm);
\draw[midarrow_rev={0.425}] (0,0) circle (\rmid cm);
\draw[fill=white] (-\rmid,0) circle (0.1cm);
\end{tikzpicture}
\;\;
\text{;\hspace{5pt} near }p;
\;\;
\begin{tikzpicture}
\draw (0,1) -- (0,-1);
\draw (-1,0) -- (1,0);
\node[dotnode] at (0,0) {};
\node at (-0.2,-0.2) {\small $p$};
\draw (0,0) circle (1cm);
\end{tikzpicture}
\longrightarrow
\begin{tikzpicture}
\foreach \x in {0,90,180,270}
\draw[rotate=\x] (0,0) --
  (0,0.05) .. controls +(up:0.15cm) and +(60:0.2cm) ..
  (-0.3,0) .. controls +(-120:0.3cm) and +(up:0.4cm) ..
  (0,-1);
\draw (0,0) circle (1cm);
\draw (0,0.99) -- (0,1);
\node[dotnode] at (0,0) {};
\end{tikzpicture}
\end{equation}

where we recall that $p$ was some fixed based point used
to define $H$.
The diffeomorphism $\UCY$ fixes $p$ and its tangent space,
swaps boundary components,
and reverses the direction of the core circle of the annulus.
The pair of arrows in the leftmost diagram
is sent to the pair of arrows in the next diagram.
The two diagrams on the right describe the behavior of $\UCY$
near $p$;
the tiny empty circle in the second-from-left diagram
should be filled in with the rightmost diagram.

$\UCY$ acts on objects of $\CYA$ by
applying $\UCY$ to the marked points,
and on morphisms, it acts by applying
$\UCY \times \id_{[0,1]}$ to graphs.

\begin{proposition}
$\UCY$ generates a $\ZZ/2$-action on
$\hCYA$ and $\CYA$.
\end{proposition}
\begin{proof}
It suffices to consider $\hCYA$.
$\UUCY$ is isotopic to the identity
by ``untwisting" around $p$.
This gives us a natural isomorphism
$\overline{\Theta}: \UUCY \simeq \id_{\hCYA}$ as follows:
for a marked point $b$, the isotopy takes
$\UUCY(b)$ back to $b$;
thinking of the interval direction of the cylinder
$\Ann \times [0,1]$ as time,
this traces out a curve interpolating
$(\UUCY(b)),0)$ and $(b,1)$.

It is a simple topology exercise
to check that $\overline{\Theta}$ is
a natural isomorphism;
it is helpful to think of $\overline{\Theta}$
as relating $\Ann \times [0,1]$
to the mapping cylinder of $(\UUCY)^\inv$.
\end{proof}

In particular, when the object is $H(X)$,
that is, $p$ labeled with $X$,
the natural isomorphism
is the vertical graph with a full negative twist,
i.e. $\overline{\Theta}_{H(X)}$
is simply $\thetabar_X$,
hence the suggestive name for $\overline{\Theta}$.

\begin{proposition}
$H: \hA \simeq \hCYA$
and $\Kar(H): \Kar(\hA) \simeq \CYA$
are $\ZZ/2$-equivariant.
\end{proposition}
\begin{proof}
It is easy to check that
$H$ commutes with $U$'s on the nose,
i.e. $\UCY \circ H = H \circ \hU$,
and likewise for Karoubi envelopes.
(Indeed, $\UCY$ was cooked up to have this property.)
The observation above shows
$\overline{\Theta}H = H\thetabar$
as natural isomorphisms
$\UUCY \circ H = H \circ \hU^2 \simeq H$,
and likewise for Karoubi envelopes.
Thus, $H$ is $\ZZ/2$-equivariant.
\end{proof}

\begin{remark} \label{r:tnsr_topology}
Since $\tnsrbar$ on $\ZA$ has a nice topological
interpretation as $\tnsrst$ on $\CYA$,
it is also nice to have a topological interpretation
for the standard tensor product on $\ZA$.
This comes from the (thickened) pair of pants,
denoted $\POP$, in the following manner.
For $(X,\ga),(Y,\mu) \in \ZA$,
with corresponding objects
$\VV = (H(X),H(\hP_\ga)), \WW = (H(Y),H(\hP_\mu))
\in \CYA$,
the object $\VV \tnsr \WW$
which corresponds to $(X,\ga)\tnsr(Y,\mu)$
is characterized by a natural isomorphism
\[
\Hom_\CYA(\VV \tnsr \WW, \mathbf{X})
\simeq \Skein(\POP; \overline{\VV},\overline{\WW},
  \mathbf{X})
= 
\Skein (
\begin{tikzpicture}[scale=0.6]
\tikzmath{
  \ellx = 1;
  \elly = 0.2;
  \ellxs = 0.8;
  \ellys = 0.16;
}
\draw (-1.5,1) ellipse (\ellx cm and \elly cm);
\draw (-1.5,1) ellipse (\ellxs cm and \ellys cm);
\draw (1.5,1) ellipse (\ellx cm and \elly cm);
\draw (1.5,1) ellipse (\ellxs cm and \ellys cm);
\draw (0,-1) ellipse (\ellx cm and \elly cm);
\draw (0,-1) ellipse (\ellxs cm and \ellys cm);
\draw (-2.5,1) to[out=-90,in=90] (-1,-1);
\draw (-2.3,1) to[out=-90,in=90] (-0.8,-1);
\draw (2.5,1) to[out=-90,in=90] (1,-1);
\draw (2.3,1) to[out=-90,in=90] (0.8,-1);
\draw (-0.5,1) .. controls +(down:0.5cm) and +(down:0.5cm)
  .. (0.5,1);
\draw (-0.7,1) .. controls +(down:0.6cm) and +(down:0.6cm)
  .. (0.7,1);
\node at (-1.5,1.4) {\tiny $\overline{\VV}$};
\node at (1.5,1.4) {\tiny $\overline{\WW}$};
\node at (0,-1.4) {\tiny $\mathbf{X}$};
\end{tikzpicture}
)
\]
for all $\mathbf{X} \in \CYA$.
(This is well-known for the extended Turaev-Viro theory \ocite{TV},
where given spherical fusion $\cA$, one has
$Z_{\text{TV}}(S^1) = \ZA$ \ocite{stringnet},
and the standard tensor product on $\ZA$ is given by
the pair of pants just as above.)
The stacking product can also be described this way,
but instead of the usual pair of pants $\POP$,
we use a different cobordism, $\YPOP$,
depicted below:
\[
\Hom_\CYA(\VV \tnsrst \WW, \mathbf{X})
\simeq \Skein(\YPOP; \overline{\VV},\overline{\WW},
  \mathbf{X})
=
\Skein (
\begin{tikzpicture}[scale=0.6]
\draw (0,1) ellipse (2cm and 0.4cm);
\draw (0,1) ellipse (1.8 cm and 0.36cm);
\draw (0,1) ellipse (1cm and 0.2cm);
\draw (0,1) ellipse (0.8cm and 0.16cm);
\draw (0,-1) ellipse (1.5cm and 0.3cm);
\draw (0,-1) ellipse (1.3cm and 0.26cm);
\draw (-2,1) to[out=-90,in=90] (-1.5,-1);
\draw (-0.8,1) to[out=-90,in=90] (-1.3,-1);
\draw (2,1) to[out=-90,in=90] (1.5,-1);
\draw (0.8,1) to[out=-90,in=90] (1.3,-1);
\draw (-1.8,1) .. controls +(down:1cm) and +(down:1cm)
  .. (-1,1);
\draw (1.8,1) .. controls +(down:1cm) and +(down:1cm)
  .. (1,1);
\node at (-2.2,1) {\tiny $\overline{\VV}$};
\node at (-1.2,1) {\tiny $\overline{\WW}$};
\node at (0,-1.5) {\tiny $\mathbf{X}$};
\end{tikzpicture}
)
\;\;.
\]
$\YPOP$ is a thickened `Y' crossed with $S^1$.
We do not prove these claims here,
which are not hard to prove after all
the work in this section.
\rmkend
\end{remark}
\begin{remark} \label{r:2fold_cy}
The topological interpretations of $\tnsr$ and $\tnsrst$
above can also elucidate the structure morphisms
mentioned in \rmkref{r:2fold}.
Consider the two cobordisms below.
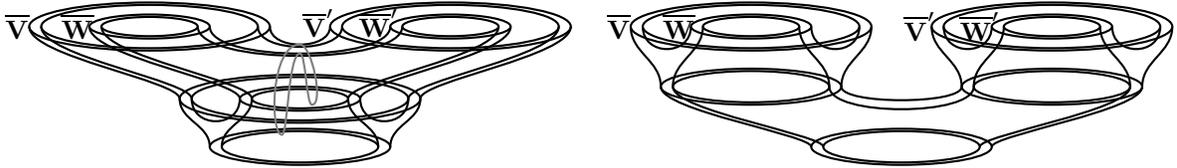
\[
\begin{tikzpicture}[scale=0.8]
\tikzmath{
  \md = -0.2;
}
\draw (-2.5,1) ellipse (2cm and 0.4cm);
\draw (-2.5,1) ellipse (1.8 cm and 0.36cm);
\draw (-2.5,1) ellipse (1cm and 0.2cm);
\draw (-2.5,1) ellipse (0.8cm and 0.16cm);
\draw (2.5,1) ellipse (2cm and 0.4cm);
\draw (2.5,1) ellipse (1.8 cm and 0.36cm);
\draw (2.5,1) ellipse (1cm and 0.2cm);
\draw (2.5,1) ellipse (0.8cm and 0.16cm);
\draw (0,\md) ellipse (2cm and 0.4cm);
\draw (0,\md) ellipse (1.8 cm and 0.36cm);
\draw (0,\md) ellipse (1cm and 0.2cm);
\draw (0,\md) ellipse (0.8cm and 0.16cm);
\draw (0,-1) ellipse (1.5cm and 0.3cm);
\draw (0,-1) ellipse (1.3cm and 0.26cm);
\draw (-4.5,1)
  .. controls +(down:0.5cm) and +(up:0.3cm) .. (-2,\md);
\draw (4.5,1)
  .. controls +(down:0.5cm) and +(up:0.3cm) .. (2,\md);
\draw (-4.3,1)
  .. controls +(down:0.5cm) and +(up:0.3cm) .. (-1.8,\md);
\draw (4.3,1)
  .. controls +(down:0.5cm) and +(up:0.3cm) .. (1.8,\md);
\draw (-3.5,1)
  .. controls +(down:0.5cm) and +(up:0.3cm) .. (-1,\md);
\draw (3.5,1)
  .. controls +(down:0.5cm) and +(up:0.3cm) .. (1,\md);
\draw (-3.3,1)
  .. controls +(down:0.5cm) and +(up:0.3cm) .. (-0.8,\md);
\draw (3.3,1)
  .. controls +(down:0.5cm) and +(up:0.3cm) .. (0.8,\md);
\draw (-0.5,1)
  .. controls +(down:0.4cm) and +(down:0.4cm) .. (0.5,1);
\draw (-0.7,1)
  .. controls +(down:0.45cm) and +(down:0.45cm) .. (0.7,1);
\draw (-1.5,1)
  .. controls +(down:0.6cm) and +(down:0.6cm) .. (1.5,1);
\draw (-1.7,1)
  .. controls +(down:0.65cm) and +(down:0.65cm) .. (1.7,1);
\draw (-2,\md) to[out=-90,in=90] (-1.5,-1);
\draw (-0.8,\md) to[out=-90,in=90] (-1.3,-1);
\draw (2,\md) to[out=-90,in=90] (1.5,-1);
\draw (0.8,\md) to[out=-90,in=90] (1.3,-1);
\draw (-1.8,\md)
  .. controls +(down:0.5cm) and +(down:0.5cm) .. (-1,\md);
\draw (1.8,\md)
  .. controls +(down:0.5cm) and +(down:0.5cm) .. (1,\md);
\draw[gray]
(0,0.7) .. controls +(left:0.5cm) and +(left:0.2cm) ..
(-0.3,-0.8) .. controls +(right:0.1cm) and +(left:0.2cm) ..
(0,0.55) .. controls +(right:0.1cm) and +(left:0.1cm) ..
(0.2,-0.3) .. controls +(right:0.2cm) and +(right:0.2cm) ..
(0,0.7);
\node at (-4.7,1) {$\overline{\VV}$};
\node at (-3.7,1) {$\overline{\WW}$};
\node at (0.3,1.05) {$\overline{\VV}'$};
\node at (1.3,1.05) {$\overline{\WW}'$};
\end{tikzpicture}
\;\;\;\;
\begin{tikzpicture}[scale=0.8]
\draw (-2.5,1) ellipse (2cm and 0.4cm);
\draw (-2.5,1) ellipse (1.8 cm and 0.36cm);
\draw (-2.5,1) ellipse (1cm and 0.2cm);
\draw (-2.5,1) ellipse (0.8cm and 0.16cm);
\draw (2.5,1) ellipse (2cm and 0.4cm);
\draw (2.5,1) ellipse (1.8 cm and 0.36cm);
\draw (2.5,1) ellipse (1cm and 0.2cm);
\draw (2.5,1) ellipse (0.8cm and 0.16cm);
\draw (-2.5,0) ellipse (1.5cm and 0.3cm);
\draw (-2.5,0) ellipse (1.3cm and 0.26cm);
\draw (2.5,0) ellipse (1.5cm and 0.3cm);
\draw (2.5,0) ellipse (1.3cm and 0.26cm);
\draw (0,-1) ellipse (1.5cm and 0.3cm);
\draw (0,-1) ellipse (1.3cm and 0.26cm);
\draw (-4.5,1) to[out=-90,in=90] (-4,0);
\draw (-3.3,1) to[out=-90,in=90] (-3.8,0);
\draw (-0.5,1) to[out=-90,in=90] (-1,0);
\draw (-1.7,1) to[out=-90,in=90] (-1.2,0);
\draw (-4.3,1)
  .. controls +(down:0.5cm) and +(down:0.5cm) .. (-3.5,1);
\draw (-0.7,1)
  .. controls +(down:0.5cm) and +(down:0.5cm) .. (-1.5,1);
\draw (4.5,1) to[out=-90,in=90] (4,0);
\draw (3.3,1) to[out=-90,in=90] (3.8,0);
\draw (0.5,1) to[out=-90,in=90] (1,0);
\draw (1.7,1) to[out=-90,in=90] (1.2,0);
\draw (4.3,1)
  .. controls +(down:0.5cm) and +(down:0.5cm) .. (3.5,1);
\draw (0.7,1)
  .. controls +(down:0.5cm) and +(down:0.5cm) .. (1.5,1);
\draw (-4,0)
  .. controls +(down:0.4cm) and +(up:0.2cm) .. (-1.5,-1);
\draw (4,0)
  .. controls +(down:0.4cm) and +(up:0.2cm) .. (1.5,-1);
\draw (-3.8,0)
  .. controls +(down:0.4cm) and +(up:0.2cm) .. (-1.3,-1);
\draw (3.8,0)
  .. controls +(down:0.4cm) and +(up:0.2cm) .. (1.3,-1);
\draw (-1,0)
  .. controls +(down:0.3cm) and +(down:0.3cm) .. (1,0);
\draw (-1.2,0)
  .. controls +(down:0.5cm) and +(down:0.5cm) .. (1.2,0);
\node at (-4.7,1) {$\overline{\VV}$};
\node at (-3.7,1) {$\overline{\WW}$};
\node at (0.3,1) {$\overline{\VV}'$};
\node at (1.3,1) {$\overline{\WW}'$};
\end{tikzpicture}
\]
The cobordism on the left (ignoring the gray curve)
corresponds to 
$(\VV \tnsr \VV') \tnsrst (\WW \tnsr \WW')$,
or $((X,\ga) \tnsr (X',\ga')) \tnsrbar
  ((Y,\mu) \tnsr (Y',\mu'))$,
while the cobordism on the right corresponds to
$(\VV \tnsrst \WW) \tnsr (\VV' \tnsrst \WW')$,
or $((X,\ga) \tnsrbar (Y,\mu)) \tnsr
  ((X',\ga') \tnsrbar (Y',\mu'))$.
They are different, but not by much:
the right can be obtained from the left
by surgery, specifically,
by attaching a 2-handle along the gray curve
(the gray curve starts in the outside $\POP$ for $\VV$'s,
goes down into the bottom $\YPOP$,
back up into the $\POP$ for $\WW$'s,
then goes down again into $\YPOP$, and closes up,
all the while staying close to the ``inner boundary",
i.e. keeping as tight as possible like a rubber band).
One way to visualize this is to consider the
reverse process of removing a 2-handle:
start with the right side, push the two ``troughs"
- between $\VV$ and $\WW$ and between $\VV'$ and $\WW'$
- downward and into the bottom pair of pants,
and when they are about to meet in the middle,
drill a hole through the wall.

On the level of skein modules,
this surgery is the same as
adding the gray curve colored by the regular coloring
(up to a factor, see \eqnref{e:sliding});
this is the topological interpretation of $\eta$.
Conversely, graphs in the right cobordism can be
lifted to a graph in the left cobordism
plus the gray curve with regular coloring;
this is the topological interpretation of $\zeta$.

The other structure morphisms can also be described
very easily. For example, one of them is
a morphism $v_2: \onebar \tnsr \onebar \to \onebar$,
which is simply the empty graph in $\POP$
(up to a factor)
interpreted as a morphism
$\onest \tnsrst \onest \to \onest$.

Thus, checking the compatibility axioms
becomes an exercise in topology.
Once again, we do not show the work in this paper.

As mentioned in the introduction, such 2-fold monoidal
structures are related to iterated loop spaces.
It may be interesting to see if these two topological
aspects of $(\ZA,\tnsr,\tnsrbar)$ are directly related.
\rmkend
\end{remark}

\section{$\cA$ Modular}
\label{s:modular}

When $\cA$ is modular, we have
(see for example \ocite{EGNO}*{Proposition 8.20.12}):

\begin{align*}
\cAAbop &\simeq_{\tnsr, \text{br}} (\ZA,\tnsr) \\
X \boxtimes Y &\mapsto (XY, c c^\inv) = (X,c) \tnsr (Y,c^\inv)
\end{align*}
where $\cA^\bop$ is $\cA$ with the opposite braiding,
and $c$ is the braiding on $\cA$.
Here the monoidal structure on $\cAAbop$ is defined component-wise.
In particular, duals are given by
$(X\boxtimes Y)^* = X^* \boxtimes Y^*$.

It is natural to ask: what is the reduced tensor product on
$\cAAbop$ under this equivalence?
We claim (proven below in \thmref{t:AA_ZA})
that the following definition is the answer,
which justifies the repeated use of the name ``reduced"
and notation like $\tnsrbar$ and $\onebar$.
(The reduced tensor product on $\ZA$ cannot in general be
braided, as we shall see soon,
so we will ignore the difference in braiding on $\cA$.)

\begin{definition}
\label{d:AA_red}
Let $W_1 \boxtimes Y_1, W_2 \boxtimes Y_2 \in \cAA$.
Define their \emph{reduced tensor product}
to be
\[
(W_1 \boxtimes Y_1) \tnsrbar (W_2 \boxtimes Y_2)
:= \eval{Y_1,W_2} \cdot W_1 \boxtimes Y_2
\]
where recall $\eval{Y_1,W_2} :=
  \Hom_\cA(\one,Y_1 W_2)$.
$\tnsrbar$ naturally extends to direct sums,
and is clearly associative.

For morphisms $f_1 \boxtimes g_1 :
W_1 \boxtimes Y_1 \to W_1' \boxtimes Y_1'$,
$f_2 \boxtimes g_2 :
W_2 \boxtimes Y_2 \to W_2' \boxtimes Y_2'$,
their \emph{reduced tensor product} is 
given by
\[
(f_1 \boxtimes g_1) \tnsrbar (f_2 \boxtimes g_2)
:= \eval{g_1,f_2} \cdot f_1 \boxtimes g_2
=
\begin{tikzpicture}
\draw (-0.2,0.5) -- (-0.2,-0.5)
  node[pos=0,above] {\small $Y_1$}
  node[below] {\small $Y_1'$};
\node[small_morphism] at (-0.2,0) {\tiny $g_1$};
\draw (0.2,0.5) -- (0.2,-0.5)
  node[pos=0,above] {\small $W_2$}
  node[below] {\small $W_2'$};
\node[small_morphism] at (0.2,0) {\tiny $f_2$};
\end{tikzpicture}
\cdot
\begin{tikzpicture}
\draw (-0.3,0.5) -- (-0.3,-0.5)
  node[pos=0,above] {\small $W_1$}
  node[below] {\small $W_1'$};
\node[small_morphism] at (-0.3,0) {\tiny $f_1$};
\draw (0.3,0.5) -- (0.3,-0.5)
  node[pos=0,above] {\small $Y_2$}
  node[below] {\small $Y_2'$};
\node[small_morphism] at (0.3,0) {\tiny $g_2$};
\node at (0,0) {$\boxtimes$};
\end{tikzpicture}
\]
where the left side, the ``coefficient" $\eval{g_1,f_2}$,
is to be interpreted as a linear map
$\eval{Y_1,W_2} \to \eval{Y_1',W_2'}$
by composition.
\defend
\end{definition}

For example,
\begin{equation} \label{e:cAA_simple}
(X_i \boxtimes X_j^*) \tnsrbar (X_k \boxtimes X_l^*)
\simeq \delta_{j,k} X_i \boxtimes X_l^*
\;\;.
\end{equation}
In particular, when $i\neq j$,
\begin{equation} \label{e:cAA_not_symm}
\begin{split}
(X_i \boxtimes X_i^*) \tnsrbar (X_i \boxtimes X_j^*)
&\simeq X_i \boxtimes X_j^* \\
(X_i \boxtimes X_j^*) \tnsrbar (X_i \boxtimes X_i^*)
&\simeq 0
\end{split}
\end{equation}
so $\tnsrbar$ cannot be braided.

\begin{proposition}
$(\cAA, \tnsrbar)$ is a pivotal multifusion category.
More precisely,
\begin{itemize}
\item The unit object, denoted $\onebar$, is
  $\bigoplus_{i\in \Irr(\cA)} X_i \boxtimes X_i^*$;

\item $(X \boxtimes Y)^\vee = Y^* \boxtimes X^*$,
  $\rvee (X \boxtimes Y) = \rvee Y \boxtimes \rdual{X}$,
  the (co)evaluation maps are described in the proof;

\item The pivotal structure is defined component-wise:
  $\delta_{X\boxtimes Y} = \delta_X \boxtimes \delta_Y$.

\end{itemize}
\end{proposition}

\begin{proof}
Fix $X \boxtimes Y \in \cAA$.
The left and right unit constraint are given by
\[
l_{X\boxtimes Y} :=
\sum_{k \in \Irr(\cA)} \sqrt{d_k}
\begin{tikzpicture}
\draw (-0.3,0.3) .. controls +(-90:0.5cm) and +(-90:0.5cm) .. (0.3,0.3)
  node[pos=0,above] {\small $X_k^*$}
  node[above] {\small $X$};
\node[small_morphism] at (-0.25,0.1) {\tiny $\al$};
\end{tikzpicture}
\cdot
\begin{tikzpicture}
\draw (-0.3,-0.3) -- (-0.3,0.3)
  node[above] {\small $X_k$}
  node[pos=0,below] {\small $X$};
\node[small_morphism] at (-0.3,0) {\tiny $\al$};
\node at (0,0) {$\boxtimes$};
\draw (0.3,-0.3) -- (0.3,0.3)
  node[above] {\small $Y$}
  node[pos=0,below] {\small $Y$};
\end{tikzpicture}
\;\; ; \;\;
r_{X\boxtimes Y} :=
\sum_{k \in \Irr(\cA)} \sqrt{d_k}
\begin{tikzpicture}
\draw (-0.3,0.3) .. controls +(-90:0.5cm) and +(-90:0.5cm) .. (0.3,0.3)
  node[pos=0,above] {\small $Y$}
  node[above] {\small $X_k$};
\node[small_morphism] at (-0.25,0.1) {\tiny $\al$};
\end{tikzpicture}
\cdot
\begin{tikzpicture}
\draw (-0.3,-0.3) -- (-0.3,0.3)
  node[above] {\small $X$}
  node[pos=0,below] {\small $X$};
\node at (0,0) {$\boxtimes$};
\draw (0.3,-0.3) -- (0.3,0.3)
  node[above] {\small $X_k^*$}
  node[pos=0,below] {\small $Y$};
\node[small_morphism] at (0.3,0) {\tiny $\al$};
\end{tikzpicture}
\]
where we recall that $\al$ is a sum over
a pair of dual bases
(see \eqnref{e:summation_convention}).
Their inverses are given by
flipping the diagram upside down.

The left (co)evaluation maps are given by
\[
\ev_{X\boxtimes Y} :=
\sum_{k \in \Irr(\cA)} \sqrt{d_k}
\begin{tikzpicture}
\draw (-0.3,0.3) .. controls +(-90:0.5cm) and +(-90:0.5cm) .. (0.3,0.3)
  node[pos=0,above] {\small $X^*$}
  node[above] {\small $X$};
\end{tikzpicture}
\cdot
\begin{tikzpicture}
\draw (-0.3,-0.3) -- (-0.3,0.3)
  node[above] {\small $Y^*$}
  node[pos=0,below] {\small $X_k$};
\node[small_morphism] at (-0.3,0) {\tiny $\al$};
\node at (0,0) {$\boxtimes$};
\draw (0.3,-0.3) -- (0.3,0.3)
  node[above] {\small $Y$}
  node[pos=0,below] {\small $X_k^*$};
\node[small_morphism] at (0.3,0) {\tiny $\al$};
\end{tikzpicture}
\;\; ; \;\;
\coev_{X \boxtimes Y} :=
\sum_{k \in \Irr(\cA)} \sqrt{d_k}
\begin{tikzpicture}
\draw (-0.3,-0.3) .. controls +(90:0.5cm) and +(90:0.5cm) .. (0.3,-0.3)
  node[pos=0,below] {\small $Y$}
  node[below] {\small $Y^*$};
\end{tikzpicture}
\cdot
\begin{tikzpicture}
\draw (-0.3,-0.3) -- (-0.3,0.3)
  node[above] {\small $X_k$}
  node[pos=0,below] {\small $X$};
\node[small_morphism] at (-0.3,0) {\tiny $\al$};
\node at (0,0) {$\boxtimes$};
\draw (0.3,-0.3) -- (0.3,0.3)
  node[above] {\small $X_k^*$}
  node[pos=0,below] {\small $X^*$};
\node[small_morphism] at (0.3,0) {\tiny $\al$};
\end{tikzpicture}
\;\;.
\]
The right (co)evaluation maps are given by
similar diagrams.
It is straightforward to check that these have
the right properties.
\end{proof}

\begin{theorem} \label{t:AA_ZA}
There is an equivalence of pivotal multifusion categories
\begin{align*}
K: \cAA &\simeq_\tnsr (\ZA,\tnsrbar)
\\
X \boxtimes Y &\mapsto (XY, c c^\inv)
\;\;.
\end{align*}
\end{theorem}

\begin{proof}
The tensor structure $L$ on $K$ is given as follows:
for $W_1 \boxtimes Y_1, W_2 \boxtimes Y_2$,
the isomorphism
$L : K(W_1\boxtimes Y_1) \tnsrbar K(W_2\boxtimes Y_2)
\simeq K((W_1\boxtimes Y_1) \tnsrbar (W_2 \boxtimes Y_2))$
is given by
\begin{align*}
L: (W_1 Y_1, c c^\inv) \tnsrbar (W_2 Y_2, c c^\inv)
&\simeq 
\eval{Y_1,W_2} \cdot (W_1 Y_2,c c^\inv)
\\
L =
\begin{tikzpicture}
\node[small_morphism] (al) at (0,0) {\tiny $\al$};
\draw (al) to[out=-135,in=90] (-0.2,-0.5)
  node[below] {\tiny $Y_1$};
\draw (al) to[out=-45,in=90] (0.2,-0.5)
  node[below] {\tiny $W_2$};
\end{tikzpicture}
&\cdot
\begin{tikzpicture}
\node[small_morphism] (al) at (0,0) {\tiny $\al$};
\draw (al) to[out=135,in=-90] (-0.2,0.5)
  node[above] {\tiny $Y_1$};
\draw (al) to[out=45,in=-90] (0.2,0.5)
  node[above] {\tiny $W_2$};
\draw (-0.5,0.5) .. controls +(down:0.3cm) and +(up:0.3cm) .. (-0.3,-0.5)
  node[pos=0,above] {\tiny $W_1$}
  node[below] {\tiny $W_1$};
\draw (0.5,0.5) .. controls +(down:0.3cm) and +(up:0.3cm) .. (0.3,-0.5)
  node[pos=0,above] {\tiny $Y_2$}
  node[below] {\tiny $Y_2$};
\end{tikzpicture}
\;\;.
\end{align*}

The inverse to $L$ is given by flipping the diagram upside down.
The following observation is helpful:
for $(W_1 Y_1,c c^\inv),(W_2 Y_2, c c^\inv) \in \ZA$,
we have
\[
W_1 Y_1 \tnsrproj{c c^\inv}{c c^\inv} W_2 Y_2
=
\im\Big(
\frac{1}{\cD}
\begin{tikzpicture}
\tikzmath{
  \toprow = 0.6;
  \bottom = -0.6;
}
\draw (0.5,\bottom) -- (0.5,\toprow)
  node[above] {\tiny $Y_2$};
\draw (-0.2,\bottom) -- (-0.2,\toprow)
  node[above] {\tiny $Y_1$};
\draw (0.2,\bottom) -- (0.2,\toprow)
  node[above] {\tiny $W_2$};
\draw[regular,overline={1.5}] (0,0) ellipse (0.7cm and 0.3cm);
\draw[overline={1.5}] (-0.5,\bottom) -- (-0.5,\toprow)
  node[above] {\tiny $W_1$};
\draw[overline={1.5}] (-0.2,0) -- (-0.2,\toprow);
\draw[overline={1.5}] (0.2,0) -- (0.2,\toprow);
\end{tikzpicture}
\Big)
=
\im\Big( \;
\begin{tikzpicture}
\tikzmath{
  \toprow = 0.6;
  \bottom = -0.6;
}
\draw (-0.5,\bottom) -- (-0.5,\toprow);
\draw (0.5,\bottom) -- (0.5,\toprow);
\node[small_morphism] (al1) at (0,0.2) {\tiny $\al$};
\node[small_morphism] (al2) at (0,-0.2) {\tiny $\al$};
\draw (al1) to[out=135,in=-90] (-0.2,\toprow);
\draw (al1) to[out=45,in=-90] (0.2,\toprow);
\draw (al2) to[out=-135,in=-90] (-0.2,\bottom);
\draw (al2) to[out=-45,in=-90] (0.2,\bottom);
\end{tikzpicture}
\; \Big)
\;\;.
\]
It is easy to check that $L$ satisfies the hexagon axiom.

Note that $K$ does not send unit to unit - the half-braiding on
$K(\onebar) = (\bigoplus X_i X_i^*,cc^\inv)$
is not the same as
$\onebar = (\bigoplus X_i X_i^*, \Gamma)$;
the isomorphism $K(\onebar) \simeq \onebar$
is essentially given by the $S$-matrix:
\[
S = \sum_{i,j\in \Irr(\cA)} \sqrt{d_i}\sqrt{d_j}
\begin{tikzpicture}
\tikzmath{
  \tp = 0.4;
  \bt = -0.4;
  \lf = -0.15;
  \rt = 0.15;
  \midtp = 0.1;
  \middn = -0.1;
}
\draw (\rt,\tp) to[out=-90,in=0] (0,\middn);
\draw[midarrow_rev={0.4}] (\lf,\bt) to[out=90,in=180] (0,\midtp);
\node at (-0.3,0.3) {\tiny $i$};
\draw[overline={1.5}] (\rt,\bt) to[out=90,in=0] (0,\midtp);
\draw[overline={1.5},midarrow={0.4}]
  (\lf,\tp) to[out=-90,in=180] (0,\middn);
\node at (-0.3,-0.3) {\tiny $j$};
\end{tikzpicture}
\;\;.
\]
Clearly the pivotal structures agree.
\end{proof}

The equivalence given above has a nice interpretation
in $\CYA$. Namely, the composition
$\Kar(H) \circ \Kar(G)^\inv \circ K :
  \cAA \simeq \CYA$
is naturally isomorphic to the following functor:
\begin{align*}
\cAA &\simeq \CYA
\\
\begin{tikzpicture}
\node (a) at (0,0.65) {$X\boxtimes Y$};
\node (b) at (0,-0.65) {$X' \boxtimes Y'$};
\draw[midarrow={0.55}] (a) -- (b) node[pos=0.5,right]
  {$f \boxtimes g$};
\end{tikzpicture}
&
\begin{tikzpicture}
\node at (0,0.65) {$\mapsto$};
\node at (0,0) {$\mapsto$};
\node at (0,-0.65) {$\mapsto$};
\end{tikzpicture}
\begin{tikzpicture}
\tikzmath{
  \toprow = 1;
  \bottom = -1;
  \midrow = 0;
  \qttop = 0.5;
  \qtbot = -0.5;
}
\draw (0,\bottom) ellipse (1cm and 0.2cm);
\draw[regular] (0.75,\qttop) arc(0:180:0.75cm and 0.15cm);
\draw[regular] (0.75,\qtbot) arc(0:180:0.75cm and 0.15cm);
\draw[overline={1}] (-0.5,\bottom) -- (-0.5,\toprow);
\draw[overline={1}] (0.5,\bottom) -- (0.5,\toprow);
\draw (0,\bottom) ellipse (0.5cm and 0.1cm); 
\node[dotnode] (n1) at (-0.85,\toprow) {};
\node[dotnode] (n2) at (-0.65,\toprow) {};
\node[dotnode] (n3) at (-0.85,\bottom) {};
\node[dotnode] (n4) at (-0.65,\bottom) {};
\draw[overline={1}] (n1) -- (n3);
\draw[overline={1}] (n2) -- (n4);
\node[dotnode] at (-0.85,\midrow) {};
\node at (-0.95,0.05) {\tiny $f$};
\node[dotnode] at (-0.65,\midrow) {};
\node at (-0.55,0.05) {\tiny $g$};
\draw[regular, overline={1.5}]
  (-0.75,\qttop) arc(180:360:0.75cm and 0.15cm);
\draw[regular, overline={1.5}]
  (-0.75,\qtbot) arc(180:360:0.75cm and 0.15cm);
\draw[overline={1}] (0,\toprow) ellipse (1cm and 0.2cm);
\draw (0,\toprow) ellipse (0.5cm and 0.1cm);
\draw (-1,\bottom) -- (-1,\toprow);
\draw (1,\bottom) -- (1,\toprow);
\node at (-0.75,1.3) {\tiny $X Y$};
\node at (-0.75,-1.35) {\tiny $X' Y'$};
\draw[decoration={brace,amplitude=3pt},decorate]
  (-1.2,0.3) -- (-1.2,\toprow);
\node at (-1.55,0.65) {im};
\node at (1.3,0.65) {$\frac{1}{\cD}$};
\draw[decoration={brace,amplitude=3pt},decorate]
  (-1.2,\bottom) -- (-1.2,-0.3);
\node at (-1.55,-0.65) {im};
\node at (1.3,-0.65) {$\frac{1}{\cD}$};
\end{tikzpicture}
\;\;.
\end{align*}

The composition $\Kar(H) \circ \Kar(G)^\inv \circ K$
itself only hits objects with one marked point $p$.
The functor presented above is more intuitive
from the following perspective.
Restricting to the first factor,
i.e. setting $Y=\one$,
this is like including $\cA$ into $\CYA$
along the outer boundary; likewise,
restricting to the second factor is like including $\cA$ into $\CYA$
along the inner boundary
(these are depicted below).
Then the functor presented above
is the $\tnsrst$-tensor product of these two restricted functors.
\[
X \mapsto (X, c) \mapsto
\im \big( \frac{1}{\cD}
\begin{tikzpicture}
\tikzmath{
  \toprow = 0.5;
  \bottom = -0.5;
  \midrow = 0;
}
\draw (0,\bottom) ellipse (1cm and 0.2cm);
\draw[regular] (0.6,\midrow) arc(0:180:0.6cm and 0.1cm);
\draw[overline={1}] (-0.5,\bottom) -- (-0.5,\toprow);
\draw[overline={1}] (0.5,\bottom) -- (0.5,\toprow);
\draw (0,\bottom) ellipse (0.5cm and 0.1cm); 
\node[dotnode] (n1) at (-0.75,\toprow) {};
\node[dotnode] (n2) at (-0.75,\bottom) {};
\draw[overline={1}] (n1) -- (n2);
\draw[regular, overline={1.5}]
  (0.6,\midrow) arc(0:-180:0.6cm and 0.1cm);
\draw[overline={1}] (0,\toprow) ellipse (1cm and 0.2cm);
\draw (0,\toprow) ellipse (0.5cm and 0.1cm);
\draw (-1,\bottom) -- (-1,\toprow);
\draw (1,\bottom) -- (1,\toprow);
\node at (-0.75,0.8) {\tiny $X$};
\node at (-0.75,-0.8) {\tiny $X$};
\end{tikzpicture}
\big)
\;\;\;
;
\;\;\;
Y \mapsto (Y, c^\inv) \mapsto
\im \big( \frac{1}{\cD}
\begin{tikzpicture}
\tikzmath{
  \toprow = 0.5;
  \bottom = -0.5;
  \midrow = 0;
}
\draw (0,\bottom) ellipse (1cm and 0.2cm);
\draw[regular] (0.875,\midrow) arc(0:180:0.875cm and 0.15cm);
\draw[overline={1}] (-0.5,\bottom) -- (-0.5,\toprow);
\draw[overline={1}] (0.5,\bottom) -- (0.5,\toprow);
\draw (0,\bottom) ellipse (0.5cm and 0.1cm); 
\node[dotnode] (n1) at (-0.75,\toprow) {};
\node[dotnode] (n2) at (-0.75,\bottom) {};
\draw[overline={1}] (n1) -- (n2);
\draw[regular, overline={1.5}]
  (0.875,\midrow) arc(0:-180:0.875cm and 0.1cm);
\draw[overline={1}] (0,\toprow) ellipse (1cm and 0.2cm);
\draw (0,\toprow) ellipse (0.5cm and 0.1cm);
\draw (-1,\bottom) -- (-1,\toprow);
\draw (1,\bottom) -- (1,\toprow);
\node at (-0.75,0.8) {\tiny $Y$};
\node at (-0.75,-0.8) {\tiny $Y$};
\end{tikzpicture}
\big)
\;\;.
\]
This picture also elucidates the definition of
$\tnsrbar$ on $\cAA$: if we take the tensor product
of the two functors above in opposite order,
essentially looking at
$(\one \boxtimes X) \tnsrbar (Y \boxtimes \one)$,
we get
\[
K(\one \boxtimes X) \tnsrst K(Y \boxtimes \one)
=
\im \big(
\frac{1}{\cD^2}
\begin{tikzpicture}
\tikzmath{
  \toprow = 0.5;
  \bottom = -0.5;
  \midrow = 0;
}
\draw (0,\bottom) ellipse (1cm and 0.2cm);
\draw[regular] (0.55,\midrow) arc(0:180:0.55cm and 0.1cm);
\draw[regular] (0.95,\midrow) arc(0:180:0.95cm and 0.2cm);
\draw[overline={1}] (-0.5,\bottom) -- (-0.5,\toprow);
\draw[overline={1}] (0.5,\bottom) -- (0.5,\toprow);
\draw (0,\bottom) ellipse (0.5cm and 0.1cm); 
\node[dotnode] (n1) at (-0.85,\toprow) {};
\node[dotnode] (n2) at (-0.65,\toprow) {};
\node[dotnode] (n3) at (-0.85,\bottom) {};
\node[dotnode] (n4) at (-0.65,\bottom) {};
\draw[overline={1}] (n1) -- (n3);
\draw[overline={1}] (n2) -- (n4);
\draw[regular, overline={1.5}]
  (-0.55,\midrow) arc(180:360:0.55cm and 0.1cm);
\draw[regular, overline={1.5}]
  (-0.95,\midrow) arc(180:360:0.95cm and 0.2cm);
\draw[overline={1}] (0,\toprow) ellipse (1cm and 0.2cm);
\draw (0,\toprow) ellipse (0.5cm and 0.1cm);
\draw (-1,\bottom) -- (-1,\toprow);
\draw (1,\bottom) -- (1,\toprow);
\node at (-0.75,0.8) {\tiny $XY$};
\node at (-0.75,-0.85) {\tiny $XY$};
\end{tikzpicture}
\big)
=
\im \big(
\frac{1}{\cD^2}
\begin{tikzpicture}
\tikzmath{
  \toprow = 0.5;
  \bottom = -0.5;
  \midrow = 0;
}
\draw (0,\bottom) ellipse (1cm and 0.2cm);
\draw[regular] (0.55,-0.1) arc(0:180:0.55cm and 0.1cm);
\draw[regular] (-0.55,0.1) arc(0:180:0.2cm and 0.04cm);
\draw[overline={1}] (-0.5,\bottom) -- (-0.5,\toprow);
\draw[overline={1}] (0.5,\bottom) -- (0.5,\toprow);
\draw (0,\bottom) ellipse (0.5cm and 0.1cm); 
\node[dotnode] (n1) at (-0.85,\toprow) {};
\node[dotnode] (n2) at (-0.65,\toprow) {};
\node[dotnode] (n3) at (-0.85,\bottom) {};
\node[dotnode] (n4) at (-0.65,\bottom) {};
\draw[overline={1}] (n1) -- (n3);
\draw[overline={1}] (n2) -- (n4);
\draw[regular, overline={1.5}]
  (-0.55,-0.1) arc(180:360:0.55cm and 0.1cm);
\draw[regular, overline={1.5}]
  (-0.95,0.1) arc(180:360:0.2cm and 0.04cm);
\draw[overline={1}] (0,\toprow) ellipse (1cm and 0.2cm);
\draw (0,\toprow) ellipse (0.5cm and 0.1cm);
\draw (-1,\bottom) -- (-1,\toprow);
\draw (1,\bottom) -- (1,\toprow);
\node at (-0.75,0.8) {\tiny $XY$};
\node at (-0.75,-0.85) {\tiny $XY$};
\end{tikzpicture}
\big)
=
\im \big(
\frac{1}{\cD}
\begin{tikzpicture}
\tikzmath{
  \toprow = 0.5;
  \bottom = -0.5;
  \midrow = 0;
}
\draw (0,\bottom) ellipse (1cm and 0.2cm);
\draw[regular] (0.55,0) arc(0:180:0.55cm and 0.1cm);
\draw[overline={1}] (-0.5,\bottom) -- (-0.5,\toprow);
\draw[overline={1}] (0.5,\bottom) -- (0.5,\toprow);
\draw (0,\bottom) ellipse (0.5cm and 0.1cm); 
\node[dotnode] (n1) at (-0.85,\toprow) {};
\node[dotnode] (n2) at (-0.65,\toprow) {};
\node[dotnode] (n3) at (-0.85,\bottom) {};
\node[dotnode] (n4) at (-0.65,\bottom) {};
\node[dotnode] (a1) at (-0.75,0.1) {};
\node at (-0.9,0.1) {\tiny $\al$};
\node[dotnode] (a2) at (-0.75,-0.1) {};
\node at (-0.9,-0.1) {\tiny $\al$};
\draw[overline={1}] (n1) to[out=-90,in=120] (a1);
\draw[overline={1}] (n2) to[out=-90,in=60] (a1);
\draw[overline={1}] (n3) to[out=90,in=-120] (a2);
\draw[overline={1}] (n4) to[out=90,in=-60] (a2);
\draw[regular, overline={1.5}]
  (-0.55,0) arc(180:360:0.55cm and 0.1cm);
\draw[overline={1}] (0,\toprow) ellipse (1cm and 0.2cm);
\draw (0,\toprow) ellipse (0.5cm and 0.1cm);
\draw (-1,\bottom) -- (-1,\toprow);
\draw (1,\bottom) -- (1,\toprow);
\node at (-0.75,0.8) {\tiny $XY$};
\node at (-0.75,-0.85) {\tiny $XY$};
\end{tikzpicture}
\big)
\simeq
\eval{X,Y} \cdot \onest
\]
where we used \eqnref{e:sliding} and
\eqnref{e:charge_conservation}.

\begin{remark}
\label{r:coincidence}
Note that the equivalence
$\cAAbop \simeq_{\tnsr, \text{br}} (\ZA,\tnsr)$
mentioned in the beginning of this section
is also built by tensoring the same two functors together;
it just happens that
\footnote{Coincidence? I think NOT!}
\[
K(X \boxtimes \one) \tnsr K(\one \boxtimes Y)
= (X,c) \tnsr (Y,c^\inv)
= (XY, c c^\inv)
= K(X \boxtimes Y)
\simeq K(X \boxtimes \one) \tnsrbar K(\one \boxtimes Y)
= (X,c) \tnsrbar (Y,c^\inv)
\;\;.
\]

One way to make sense of this coincidence is to give a
topological interpretation of $K$,
which we shall do now without much justification.
(See also Remarks \ref{r:tnsr_topology} and \ref{r:2fold_cy}.)

We claim that $K$, more precisely
$\wdtld{K} = \Kar(H) \circ \Kar(G)^\inv \circ K$,
is defined by the following cobordism:
\[
M = 
\begin{tikzpicture}
\begin{scope}[shift={(0,-0.3)}]
\node[emptynode] (a) at (0,0) {};
\node[emptynode] (b) at (0.5,0) {};
\node[emptynode] (c) at (1.5,0) {};
\node[emptynode] (d) at (2,0) {};
\draw (a) .. controls +(90:0.4cm) and +(90:0.4cm) .. (d);
\draw (a) .. controls +(-90:0.4cm) and +(-90:0.4cm) .. (d);
\draw (b) .. controls +(90:0.2cm) and +(90:0.2cm) .. (c);
\draw (b) .. controls +(-90:0.2cm) and +(-90:0.2cm) .. (c);
\draw (a) .. controls +(90:1.2cm) and +(90:1.2cm) .. (d);
\draw (b) .. controls +(90:0.6cm) and +(90:0.6cm) .. (c);
\end{scope}
\end{tikzpicture}
\;\;
\text{, i.e., }
\Hom_{\CYA}(\wdtld{K}(X \boxtimes Y), \VV)
=
\Skein(M; X, Y, \VV)
=
\Skein(\,
\begin{tikzpicture}
\begin{scope}[shift={(0,-0.3)}]
\node[emptynode] (a) at (0,0) {};
\node[emptynode] (b) at (0.5,0) {};
\node[emptynode] (c) at (1.5,0) {};
\node[emptynode] (d) at (2,0) {};
\draw (a) .. controls +(90:0.4cm) and +(90:0.4cm) .. (d);
\draw (a) .. controls +(-90:0.4cm) and +(-90:0.4cm) .. (d);
\draw (b) .. controls +(90:0.2cm) and +(90:0.2cm) .. (c);
\draw (b) .. controls +(-90:0.2cm) and +(-90:0.2cm) .. (c);
\draw (a) .. controls +(90:1.2cm) and +(90:1.2cm) .. (d);
\draw (b) .. controls +(90:0.6cm) and +(90:0.6cm) .. (c);
\node[small_dotnode] at (1,0.4) {};
\node at (1.1,0.55) {\tiny $Y$};
\node[small_dotnode] at (1,0.85) {};
\node at (1.1,1) {\tiny $X$};
\node at (0.3,-0.4) {\tiny $\VV$};
\end{scope}
\end{tikzpicture}
\,)
\]

Then $\wdtld{K}(X \boxtimes \one) \tnsrst \wdtld{K}(\one \boxtimes Y)
= \Kar(H) \circ \Kar(G)^\inv (K(X \boxtimes \one) \tnsrbar K(\one \boxtimes Y))$
may be interpreted as stacking one $M$ on top of the other:
\[
\begin{tikzpicture}
\begin{scope}[shift={(0,-0.3)}]
\node[emptynode] (a) at (0,0) {};
\node[emptynode] (b) at (0.5,0) {};
\node[emptynode] (c) at (1.5,0) {};
\node[emptynode] (d) at (2,0) {};
\draw (a) .. controls +(90:0.4cm) and +(90:0.4cm) .. (d);
\draw (a) .. controls +(-90:0.4cm) and +(-90:0.4cm) .. (d);
\draw (b) .. controls +(90:0.2cm) and +(90:0.2cm) .. (c);
\draw (b) .. controls +(-90:0.2cm) and +(-90:0.2cm) .. (c);
\draw (a) .. controls +(90:1.2cm) and +(90:1.2cm) .. (d);
\draw (b) .. controls +(90:0.6cm) and +(90:0.6cm) .. (c);
\node[small_dotnode] at (1,0.4) {};
\node at (1.1,0.55) {\tiny $Y$};
\end{scope}
\begin{scope}[shift={(0,-0.3)}]
\node[emptynode] (e) at (-0.7,0) {};
\node[emptynode] (f) at (-0.2,0) {};
\node[emptynode] (g) at (2.2,0) {};
\node[emptynode] (h) at (2.7,0) {};
\draw (e) .. controls +(90:0.7cm) and +(90:0.7cm) .. (h);
\draw (e) .. controls +(-90:0.7cm) and +(-90:0.7cm) .. (h);
\draw (f) .. controls +(90:0.5cm) and +(90:0.5cm) .. (g);
\draw (f) .. controls +(-90:0.5cm) and +(-90:0.5cm) .. (g);
\draw (e) .. controls +(90:2.1cm) and +(90:2.1cm) .. (h);
\draw (f) .. controls +(90:1.5cm) and +(90:1.5cm) .. (g);
\node[small_dotnode] at (1,1.5) {};
\node at (1.1,1.7) {\tiny $X$};
\end{scope}
\end{tikzpicture}
\]
which, upon gluing, becomes $M$ again;
on the other hand, we may consider the pair of pants product
$\wdtld{K}(X \boxtimes \one) \tnsr \wdtld{K}(\one \boxtimes Y)$:
\[
\begin{tikzpicture}
\tikzmath{
  \ellx = 1;
  \elly = 0.2;
  \ellxs = 0.5;
  \ellys = 0.1;
}
\draw (-1.5,1) ellipse (\ellx cm and \elly cm);
\draw (-1.5,1) ellipse (\ellxs cm and \ellys cm);
\draw (1.5,1) ellipse (\ellx cm and \elly cm);
\draw (1.5,1) ellipse (\ellxs cm and \ellys cm);
\draw (0,-1) ellipse (\ellx cm and \elly cm);
\draw (0,-1) ellipse (\ellxs cm and \ellys cm);
\draw (-2.5,1) to[out=-90,in=90] (-1,-1);
\draw (-2.0,1) to[out=-90,in=90] (-0.5,-1);
\draw (2.5,1) to[out=-90,in=90] (1,-1);
\draw (2.0,1) to[out=-90,in=90] (0.5,-1);
\draw (-0.5,1) .. controls +(down:0.5cm) and +(down:0.5cm)
  .. (0.5,1);
\draw (-1,1) .. controls +(down:1cm) and +(down:1cm)
  .. (1,1);
\begin{scope}[shift={(-2.5,1)}]
\node[emptynode] (a) at (0,0) {};
\node[emptynode] (b) at (0.5,0) {};
\node[emptynode] (c) at (1.5,0) {};
\node[emptynode] (d) at (2,0) {};
\draw (a) .. controls +(90:1.2cm) and +(90:1.2cm) .. (d);
\draw (b) .. controls +(90:0.6cm) and +(90:0.6cm) .. (c);
\node[small_dotnode] at (1,0.85) {};
\node at (1.1,1) {\tiny $X$};
\end{scope}
\begin{scope}[shift={(0.5,1)}]
\node[emptynode] (a) at (0,0) {};
\node[emptynode] (b) at (0.5,0) {};
\node[emptynode] (c) at (1.5,0) {};
\node[emptynode] (d) at (2,0) {};
\draw (a) .. controls +(90:1.2cm) and +(90:1.2cm) .. (d);
\draw (b) .. controls +(90:0.6cm) and +(90:0.6cm) .. (c);
\node[small_dotnode] at (1,0.4) {};
\node at (1.1,0.55) {\tiny $Y$};
\end{scope}
\end{tikzpicture}
\]
which is once again $M$.
\end{remark}

\begin{remark}
One can consider a similar tensor product
on $\cAA$ when $\cA$ is not modular.
\defref{d:AA_red} of $\tnsrbar$ will look the same,
but $\eval{Y_1,W_2}$ would be replaced by
the symmetric part of $Y_1 W_2$,
that is, the direct summands of $Y_1 W_2$
that belong to the symmetric center of $\cA$.
The functor $K: \cA \boxtimes \cA \to \ZA$
will still respect $\tnsrbar$,
but $K$ will not be an equivalence.
Furthermore, it is not clear whether
$\tnsrbar$ on $\cAA$ possesses a unit.
As evidence suggestive of this absence of unit,
recall that in the modular case,
even showing $K(\onebar) \simeq \onebar$
required the non-degeneracy of the $S$-matrix.
\rmkend
\end{remark}

The $\ZZ/2$-action is particularly simple
on $\cAA$.

\begin{proposition}
There is a tensor equivalence
\begin{align*}
U_M : (\cAA,\tnsrbar) &\simeq_\tnsr (\cAA,\tnsrbar^\op)
\\
X \boxtimes Y &\mapsto Y \boxtimes X
\;\;.
\end{align*}
We have $U^2_M = \id$, so
$U_M$ generates a $\ZZ/2$-action on $\cAA$
(but not a tensor action).
\end{proposition}
\begin{proof}
Let $W_1 \boxtimes Y_1, W_2 \boxtimes Y_2
\in \cAA$.
We have
\[
U_M(W_1 \boxtimes Y_1) \tnsrbar^\op
U_M(W_2 \boxtimes Y_2)
= \eval{W_2,Y_1} \cdot Y_2 \boxtimes W_1
\]
and
\[
U_M(W_1 \boxtimes Y_1 \tnsrbar W_2 \boxtimes Y_2)
= \eval{Y_1,W_2} \cdot Y_2 \boxtimes W_1
\;\;.
\]
We have the natural isomorphism
$z:\eval{W_2,Y_1} \simeq \eval{Y_1,W_2}$
(see \eqnref{e:cyclic}).
The tensor structure is given by
$z \cdot \id_{Y_2 \boxtimes W_1}$.
\end{proof}

\begin{proposition}
The natural isomorphism
$u_M : U \circ K \simeq K \circ U_M$
given by
\[
(u_M)_{X \boxtimes Y} = (\id_Y \tnsr \thetabar_X)
  \circ (c^\inv)_{X,Y}
=
\begin{tikzpicture}
\tikzmath{
  \a = 0.4;
  \b = -0.4;
  \lf = -0.3;
  \rt = 0.3;
}
\draw (\rt,\a) .. controls +(down:0.2cm) and +(up:0.4cm) .. (\lf,\b)
  node[pos=0,above] {\small $Y$};
\draw[overline={1.5}] (\lf,\a) .. controls +(down:0.2cm) and +(up:0.4cm) .. (\rt,\b)
  node[pos=0,above] {\small $X$};
\node[small_morphism] at (0.25,-0.2) {\tiny $\thetabar$};
\end{tikzpicture}
\]
makes $K : \cAA \simeq \ZA$
a $\ZZ/2$-equivariant equivalence.
\end{proposition}

\begin{proof}
Straightforward.
(The $\thetabar$ is needed
so that $u^2_M = \thetabar_{XY}$,
because $\thetabar : U^2 \simeq \id$.)
\end{proof}

There is a more topological reason
why $\thetabar$ is needed.
In the equivalence $\cAA \simeq \CYA$ shown above,
$X \boxtimes Y$ is sent to an object with two marked
points $b_1$ and $b_2$, labeled $X$ and $Y$.
We can choose to make the marked points
very close to $p$ so that they are fixed by $\UCY$.
Then $u_M$ would be a graph in $\CYA$
between the images of $X\boxtimes Y$ and $Y\boxtimes X$,
and $u^2_M$ is a full inverse twist.

Finally, we compare $(\cAA,\tnsrbar)$ to
$\matvec{}$,
the category of $\Vctsp$-valued matrices
(see \ocite{EGNO}*{Example 4.1.3}).
Let $\cS$ be some finite set.
The objects of $\matvec{\cS}$
are bigraded vector spaces
$V = \bigoplus_{i,j\in \cS} V_i^j$,
and the tensor product is given by
\[
(V \tnsr W)_i^j = \bigoplus_{k\in \cS} V_i^k \tnsr W_k^j
\;\;.
\]
Let $\kk_i^j$ be $\kk$ with bigrading $i,j$.
Then the unit is $\bigoplus_{i\in \cS} \kk_i^i$.
Duals are given by transposing the matrix
and then taking duals componentwise,
i.e. $(V^*)_i^j = (V_j^i)^*$.

\begin{proposition} \label{p:ZA_matvec}
There is a tensor equivalence
\begin{align*}
\matvec{\Irr(\cA)} &\simeq_\tnsr (\cAA,\tnsrbar)
\\
\kk_i^j &\mapsto X_i \boxtimes X_j^*
\;\;.
\end{align*}
\end{proposition}

\begin{proof}
Clearly the functor above is an equivalence of
abelian categories,
and sends unit to unit.

Denote $X_i^j := X_i \boxtimes X_j^*$,
so $(X_i^j)^\vee = X_j^i$.
One has
$X_i^j \tnsrbar X_j^l = \eval{X_j^*,X_k} X_i^l$,
which is 0 if $j \neq k$.
When $j=k$, define $J_k$ to be the map
\[
(\ev_{X_k} \circ \; -) \cdot \id_{X_i^l}:
  X_i^k \tnsrbar X_k^l \simeq X_i^l
\;\;.
\]
Then it is easy to check that
\[
J_{\kk_i^j,\kk_k^l} = \delta_{j,k} J_k
\]
is a tensor structure on the functor above.
\end{proof}

While $(\cAA,\tnsrbar)$ and $\matvec{\Irr(\cA)}$ are tensor equivalent,
they are not pivotal tensor equivalent
(assuming the natural pivotal structure on $\matvec{\Irr(\cA)}$
coming from $\Vctsp$).
This can be seen from computing dimensions.
For example, the left trace of $\id_{X_i^j}$ is
\[
\sum_{k\in \Irr(\cA)} d_k
\begin{tikzpicture}
\draw[midarrow] (0,0) circle (0.3cm);
\node at (-0.35,0.2) {\tiny $j^*$};
\end{tikzpicture}
\; \cdot
\begin{tikzpicture}
\draw[midarrow={0.55}] (-0.3,0.6) -- (-0.3,-0.6)
  node[pos=0,above] {\small $X_k$}
  node[below] {\small $X_k$};
\node at (-0.45,0) {\tiny $i$};
\node[small_morphism] at (-0.3,0.3) {\tiny $\al$};
\node[small_morphism] at (-0.3,-0.3) {\tiny $\beta$};
\node at (0,0) {$\boxtimes$};
\draw[midarrow_rev={0.55}] (0.3,0.6) -- (0.3,-0.6)
  node[pos=0,above] {\small $X_k^*$}
  node[below] {\small $X_k^*$};
\node at (0.45,0) {\tiny $i$};
\node[small_morphism] at (0.3,0.3) {\tiny $\al$};
\node[small_morphism] at (0.3,-0.3) {\tiny $\beta$};
\end{tikzpicture}
=
\sum_k \delta_{i,k} \frac{d_j}{d_k}
  \id_{X_k^k}
\]
so the left dimension of $X_i^j$ is
$d_{X_i^j}^L = \frac{d_j}{d_i}$,
which cannot always be 1 for all pairs $i,j$.
Its right dimension is
$d_{X_i^j}^R = d_{X_j^i}^L = \frac{d_i}{d_j}$.
Thus $(\cAA,\tnsrbar)$ can be thought of
as (the result of)
a deformation of the pivotal structure on $\matvec{}$.

\section{$\cA$ Symmetric}
\label{s:symmetric}

Next we consider the other extreme case,
when $\cA$ is symmetric.
The main contribution of this section is to give
a more explicit description of $\ZCY(\torus)$
when $\cA$ is symmetric.

As mentioned before, when $\cA$ is symmetric,
our reduced tensor product construction
coincides with Wasserman's symmetric tensor product
in \ocite{wassermansymm}.
While we arrived at these results independently,
many of the results here can be found there,
which we include for completeness;
we also add several new results,
e.g. \thmref{t:super_same} and \thmref{t:torus_DG}.
We thank Thomas Wasserman for pointing out
the subtlety of the braidings that we missed
(see \rmkref{r:super_diff}).

In this section, all tensor products
(not including $\tnsrbar$) and diagrams
are considered in the category of
vector spaces over $\kk$.

We start with taking
$\cA = \Rep(G)$ for some finite group $G$;
we deal with the supergroup case later
by making some minor modifications.
Of course, $\Rep(G)$ can be treated by the supergroup
case by taking $z = e$,
but we find exposition this way to be clearer.
We will also show (\thmref{t:super_same})
that the supergroup case is essentially the same as
the non-super case (but see \rmkref{r:super_diff}),
so one may choose to
ignore the supergroup case and its complexities
for the purposes of this paper.

Let us first recall the following
(see e.g. \ocite{BakK}*{Section 3.2},
\ocite{EGNO}*{Section 7.14}):
\begin{proposition}
\label{p:ZA_DG}
We have an equivalence of modular categories
\[
T: \cZ(\Rep(G)) \simeq_{\tnsr,br} \cD(G)-\mmod
\]
where $\cD(G)$ is the Drinfeld double of $G$.
\end{proposition}
\begin{proof}
We recall the details of the equivalence
to serve as a scaffold for subsequent discussion.
Recall that $\cD(G) \cong \kk[G] \tnsr F(G)$
as coalgebras, multiplication is defined by
$(g_1 \tnsr \delta_{h_1}) \cdot (g_2 \tnsr \delta_{h_2})
= g_1 g_2 \tnsr \delta_{g_2^\inv h_1 g_2} \delta_{h_2}$,
and the $R$-matrix is given by
$\sum_{g\in G} g \tnsr \delta_g \in \cD(G) \tnsr \cD(G)$.

Given $(V,\ga) \in \cZ(\Rep(G))$,
we put an $F(G)$ action on $V$ as follows:
for $f\in F(G), v\in V$,
\[
f \cdot v :=
  \eval{\id_V \tnsr f, \ga_{\kk[G]} (e \tnsr v)}
\;\;
\text{, or in diagrammatic terms,}
\;\;
f \cdot (-) := 
\begin{tikzpicture}
\node[small_morphism] (ga) at (0,0) {\tiny $\ga$};
\draw (ga) -- (0,0.5) node[above] {\small $V$};
\draw (ga) -- (0,-0.5);
\node[small_morphism] (eta) at (-0.3,0.3) {\tiny $e$};
\node[small_morphism] (f) at (0.3,-0.3) {\tiny $f$};
\draw (ga) -- (eta);
\draw (ga) -- (f);
\end{tikzpicture}
\]
where $e$ is the identity element of $G$,
here thought of as the (not $G$-)linear map
$\kk \to \kk[G]$ sending $1 \mapsto e$;
the diagram on the right is not a morphism in
$\cA$, it is just a linear map of vector spaces.

Conversely, given a $\cD(G)$-module $V$,
one has the half-braiding $\ga$ on $V$ defined as follows:
for $A \in \cZ(\Rep(G))$, and $a\in A, v\in V$,
\begin{align*}
\ga_A : A \tnsr V &\simeq V \tnsr A
\\
a \tnsr v &\mapsto
\sum_{g\in G} \delta_g \cdot v \tnsr g \cdot a
\;\;.
\end{align*}
In other words, $\ga = P \circ R$,
where $P$ is the canonical swap in vector spaces over
$\kk$.

$\cD(G)-\mmod$ can be thought of as the category of
$G$-equivariant bundles over $G$,
where the action of $G$ on the base $G$ is conjugation.
Namely, $V = \bigoplus_{g\in G} V_g$,
where $V_g = \delta_g \cdot V$.
The commutation relation between $\kk[G]$ and $F(G)$ inside $\cD(G)$
means that the action of $g\in G$ gives an isomorphism
$g: V_h \simeq V_{g h g^\inv}$;
then $V_g$ is naturally a $Z(g)$-representation
(where $Z(g) \subseteq G$ is the centralizer of $g$).
The orbit of $g$ in the base
is the conjugacy class $\bar{g}$,
which determines a central idempotent
$\delta_{\bar{g}} := \sum_{h\in \bar{g}} \delta_h$.
This decomposes $\cD(G)-\mmod$ (as an abelian category)
as a direct sum
\[
\cD(G)-\mmod \sim \bigoplus_{\bar{g}\in G//G}
\cD(G)\delta_{\bar{g}}-\mmod
\]

There is a \emph{convolution} tensor product
on $\cD(G)-\mmod$ given by
\[
(V \tnsr W)_g = \bigoplus_h V_{gh} \tnsr W_{h^\inv}
\]
which arises from the coproduct structure on $\cD(G)$,
\[
g \mapsto g \tnsr g
\;\;,\;\;
\delta_g \mapsto \sum_h \delta_{gh} \tnsr \delta_{h^\inv}
\]
and it is easy to check that it agrees with
the usual tensor product on $\cZ(\Rep(G))$
under this equivalence. The dual is given by
\[
(V^*)_g = (V_{g^\inv})^*
\]
and the pivotal structure is the natural
pivotal structure of vector spaces $V_g \simeq (V_g)^{**}$.
\end{proof}

In particular, $T$ is an equivalence of abelian categories,
so we want to interpret the reduced tensor product
$\tnsrbar$ as a tensor product on
$\cD(G)-\mmod$.
As in the modular case, we use the same symbol $\tnsrbar$,
which will be justified
when we prove their equivalence (\thmref{t:ZA_DG_red}).

\begin{definition}
Let $V,W \in \cD(G)-\mmod$. Their
\emph{reduced tensor product},
or \emph{fiberwise tensor product},
denoted $V \tnsrbar W$, is defined
first as a $G$-graded vector space by
\[
(V \tnsrbar W)_g = V_g \tnsr W_g
\]
and then as a $G$-module,
by $h\in G$ acting diagonally:
\[
h = h \tnsr h : V_g \tnsr W_g \simeq
	V_{hgh^\inv} \tnsr W_{hgh^\inv}
\;\;.
\]
For morphisms $\ph: V \to V', \psi: W \to W'$,
\[
(\ph \tnsrbar \psi)_g = \ph_g \tnsr \psi_g :
	V_g \tnsr W_g \to V_g' \tnsr W_g'
\;\;.
\]
\defend
\end{definition}

Note that the $G$-grading on $V \tnsrbar W$
implicitly defines the action of $\delta_g$ as
\begin{equation}
\delta_g(v_{h} \tnsrbar w_{h'}) =
	\delta_{g,h} \delta_{g,h'} \cdot v_{h} \tnsrbar w_{h'}
\end{equation}

\begin{proposition}
$(\cD(G)-\mmod,\tnsrbar)$ is a
pivotal multifusion category,
with
\begin{itemize}
\item Unit $\onebar = \bigoplus_{g\in G} \kk_g$,
	i.e. $(\onebar)_g = \kk$ for all $g\in G$,
  where $\kk$ is the trivial $Z(g)$ representation;
\item Duals: $(V^\vee)_g = (V_g)^*$,
so $V^\vee = \rvee V$;
\item Pivotal structure is the natural one
	$V_g \simeq (V_g)^{**} = (V^{\vee\vee})_g$
\end{itemize}
\end{proposition}
\begin{proof}
Straightforward
\end{proof}

Note that the unit is $\kk[G]$
with action by conjugation,
in other words $\bigoplus_i \End_\kk(X_i)$,
and not the regular representation $\kk[G]$,
which is $\bigoplus_i X_i^{\oplus d_i}$.
This should be reminiscent of the unit
$\onebar = \bigoplus X_i X_i^*$ for
$(\ZA,\tnsrbar)$.
Indeed, we have the following result
(see also \ocite{wassermansymm}*{Theorem 44}):

\begin{theorem}
\label{t:ZA_DG_red}
The equivalence $T$ of \prpref{p:ZA_DG}
is a pivotal tensor equivalence
\[
T : (\ZA,\tnsrbar) \simeq_\tnsr (\cD(G)-\mmod,\tnsrbar)
\;\;.
\]
\end{theorem}

\begin{proof}
The regular representation contains
$d_i = \dim_\kk(X_i)$ copies of $X_i$,
and thus the identity morphism on the regular representation
$\id_{\kk[G]}$ is effectively the same as
the dashed line \eqnref{e:regular_color}:
\[
\sum_{i \in \Irr(\Rep(G))} d_i \cdot \id_{X_i}
= \id_{\kk[G]}
\;\;.\footnotemark
\]
\footnotetext{This equation, strictly speaking,
does not make sense,
but it is easy to check that in the computations we perform
for $Q_{\ga,\mu}$,
we may replace the regular coloring (dashed line)
of the loop in the diagram \eqnref{e:Qproj}
by the regular representation,
because it gets traced away.}
Consider objects $(V,\ga),(W,\mu) \in \cZ(\Rep(G))$,
their reduced tensor product is the
image of the projector $Q_{\ga,\mu}$.
We compute the projection $T(Q_{\ga,\mu})$
on $V \tnsr W \in \cD(G)-\mmod$
(all sums are over $G$):
\begin{align*}
v \tnsr w
&\mapsto \frac{1}{|G|}
\sum_g g \tnsr \delta_g \tnsr v \tnsr w
\;\;\;\;
(\text{from }
\coev: 1 \mapsto \sum_g g \tnsr \delta_g \in \kk[G] \tnsr \kk[G]^*)
\\
&\mapsto \frac{1}{|G|}
\sum_g g \tnsr v \tnsr \delta_g \tnsr w
\\
&\mapsto \frac{1}{|G|}
\sum_{g,h_1,h_2} \delta_{h_1} \cdot v \tnsr h_1 g
	\tnsr \delta_{h_2} \cdot w \tnsr \delta_{h_2 g}
\;\;\;\;
(\text{applying half-braidings }\ga, \mu)
\\
&\mapsto \frac{1}{|G|}
\sum_{g,h_1,h_2} \delta_{h_1} \cdot v \tnsr \delta_{h_2} \cdot w
	\tnsr h_1 g \tnsr \delta_{h_2 g}
\\
&\mapsto \frac{1}{|G|}
\sum_{g,h_1,h_2} 
	\eval{h_1 g,\delta_{h_2 g}}
	\delta_{h_1} \cdot v \tnsr \delta_{h_2} \cdot w
\\
&= \sum_h \delta_h \cdot v \tnsr \delta_h \cdot w
\;\;\;.
\end{align*}

In particular,
\[
T(Q_{\ga,\mu})|_{V_g \tnsr W_h} =
\delta_{g,h} \; \id_{V_g \tnsr W_h}
\;\;.
\]
It follows easily that
\[
\im(T(Q_{\ga,\mu})) = V \tnsrbar W
\;\;.
\]
\end{proof}

\begin{remark} \label{r:T_symm}
We note that both $(\cZ(\Rep(G)),\tnsrbar)$
and $(\cD(G)-\mmod,\tnsrbar)$
are symmetric:
$(\cD(G)-\mmod,\tnsrbar)$ inherits the braiding
on vector spaces $P: V_g \tnsr W_g \simeq W_g \tnsr V_g$,
while $(\cZ(\Rep(G)),\tnsrbar)$ can inherit
the symmetric braiding $c$ on $\Rep(G)$,
as $c$ intertwines projections $Q_{\ga,\mu}, Q_{\mu,\ga}$
and half-braidings
to give
$c : (X,\ga) \tnsrbar (Y,\mu) \simeq (Y,\mu) \tnsrbar (X,\ga)$.
Then the above equivalence $T$ is in fact braided
(see \ocite{wassermansymm}*{Theorem 44}).
\rmkend
\end{remark}

It is possible to describe $\tnsrbar$ in terms of
a coproduct structure on $\cD(G)$.
Recall that a coproduct $\Delta$ on an algbera $A$,
as part of a bialgebra structure on $A$,
defines, by pullback along $\Delta$,
a module structure on the
tensor product (as vector spaces) of two representations
of $A$,
thus defining a tensor product on the category of representations
$\Rep(A)$.
In the following, we define a coproduct $\ov{\Delta}$ on $\cD(G)$,
but it does not preserve units,
so for $V,W \in \cD(G)-\mmod$,
the unit $1 \in \cD(G)$ does not act as the identity
on $V \tnsr W$
(tensor product as vector spaces);
to remedy this, we perform a projection by the action of
$\ov{\Delta}(1)$ on $V \tnsr W$.
This perspective will be particularly useful
in the supergroup case.

\begin{proposition} \label{p:coprod}
Consider the map
\begin{align*}
\overline{\Delta} : \cD(G) &\mapsto \cD(G) \tnsr \cD(G)
\\
g &\mapsto g \tnsr g
\\
\delta_g &\mapsto \delta_g \tnsr \delta_g
\end{align*}
for all $g\in G$.
This is an algebra homomorphism, 
and for $V,W \in \cD(G)-\mmod$,
\[
V \tnsrbar W = \overline{\Delta}(1) \cdot (V \tnsr W)
\]
where $1 = e \tnsr \delta_* = e \tnsr \sum_h \delta_h$
is the unit of $\cD(G)$.
\end{proposition}
\begin{proof}
Straightforward.
\end{proof}

The methods above can be adapted for finite supergroups.
Recall that $\Rep(G,z)$ is the same as
$\Rep(G)$ as a pivotal fusion category,
but has a $\ZZ/2$-grading determined by the
action of the order 2 central element $z$,
and the braiding on $\Rep(G,z)$ is twisted by
the Koszul sign convention.

For $V \in \Rep(G,z)$,
the actions of $e^0 := \frac{1+z}{2}$ and
$e^1 := \frac{1-z}{2}$
are projections to its even and odd part,
which we denote by $V^0$ and $V^1$, respectively.
Similarly, for $v\in V$, we denote by $v^0$ and $v^1$
its even and odd part.
In particular, for elements of $\cD(G)$
as a $G$-module by left multiplication,
$g^\sigma = e^\sigma g = (g + (-1)^\sigma gz)/2$
and
$\delta_g^\sigma = e^\sigma \delta_g$.
We warn the reader not to confuse the left
multiplication action of $G$ on $\cD(G)$
with the $G$-action on $\kk[G]^* = F(G) \subseteq \cD(G)$,
which is by precomposition,
as used in the proof of \thmref{t:ZA_DG_red_super};
in the latter,
$\delta_g^\sigma = 
(\delta_g + (-1)^\sigma \delta_{gz})/2$,
which, under the inclusion $F(G) \hookrightarrow \cD(G)$,
does not agree with
$e^\sigma \delta_g = (\delta_g + (-1)^\sigma z \delta_g)/2$.

Letting $\cA = \Rep(G,z)$,
there is no change in the proof of \prpref{p:ZA_DG},
so that we still have an equivalence of abelian categories
$\cZ(\Rep(G,z)) \simeq \cD(G)-\mmod$.
However, the different braiding on $\Rep(G,z)$
affects $\tnsrbar$ on $\cZ(\Rep(G,z))$,
so we will need to modify the definition
of the reduced/fiberwise tensor product on $\cD(G)-\mmod$
(see also \ocite{wassermansymm}*{Definition 49}):

\begin{definition}
\label{d:symm_red_super}
Let $V,W \in \cD(G)-\mmod$.
Their \emph{reduced tensor product
with respect to $z$},
or \emph{$z$-twisted fiberwise tensor product},
denoted $V \tnsrz{z} W$,
is defined first as a $G$-graded vector space by
\[
(V\tnsrz{z} W)_g :=
(V^0_g \tnsr W^0_g)
\oplus (V^1_{gz} \tnsr W^1_{gz})
\oplus (V^0_{gz} \tnsr W^1_g)
\oplus (V^1_g \tnsr W^0_{gz})
=
\bigoplus_{\sigma,\tau \in \ZZ/2}
  V^\sigma_{gz^\tau} \tnsr W^\tau_{gz^\sigma}
\]
and then as a $G$-module, $h\in G$ acts diagonally, i.e.
\[
h = h \tnsr h :
  V^\sigma_{gz^\tau} \tnsr W^\tau_{gz^\sigma}
  \simeq
  V^\sigma_{hgz^\tau h^\inv} \tnsr W^\tau_{hgz^\sigma h^\inv}
\;\;.
\]
For morphisms,
it is just defined by the tensor product
of the appropriate restrictions.
\defend
\end{definition}

The definition of $\tnsrz{z}$ looks complicated.
It is useful to observe the following:
the tensor product mixes fibers over $g$ and $gz$
together, and of the terms in
$(\bigoplus_{\sigma,\tau} V^\sigma_{gz^\tau})
\tnsr
(\bigoplus_{\sigma',\tau'} W^{\sigma'}_{gz^{\tau'}})$,
only those with $\sigma + \tau + \sigma' + \tau' = 0$
appear in $V \tnsrz{z} W$;
such a term will have $G$-grading $gz^{\sigma+\tau'}$,
and have parity $\sigma + \sigma'$.
(It may seem, from $(V\tnsrz{z} W)_g$
in \defref{d:symm_red_super},
that we must have $\sigma = \tau'$ and $\sigma' = \tau$,
but here we are also considering $(V\tnsrz{z} W)_{gz}$.)

\begin{proposition}
$(\cD(G)-\mmod,\tnsrz{z})$ is a
pivotal multifusion category,
with
\begin{itemize}
\item Unit $\onebar = \bigoplus_{g\in G} \kk_g$
  is the same as for $(\cD(G)-\mmod,\tnsrbar)$.
\item Duals: $(V^\vee)_g = (V_g)^*$,
so $V^\vee = \rvee V$; the evaluation map is 0
on the odd part of $V^\vee \tnsrz{z} V$,
and the obvious one on the even part;
similarly for the coevaluation
(see proof for details);
\item Pivotal structure is the natural one
  $V_g \simeq (V_g)^{**} = (V^{\vee\vee})_g$
\end{itemize}
\end{proposition}
\begin{proof}
This is more or less straightforward, if not
a little tricky for one not accustomed to the world
of super vector spaces.
Let us check that \;
$\ev \circ \coev = \id : V_g \to
  (V \tnsrz{z} V^\vee \tnsrz{z} V)_g \to V_g$.
Let $\{v_i\},\{\bar{v}_j\}$ be bases for
$V^0_g,V^1_g$ respectively,
and $\{w_k\},\{\bar{w}_l\}$ bases for $V^0_{gz},V^1_{gz}$
respectively. We raise indices for their dual bases
(not to be confused with parity).

Let $x\in V^0_g, \bar{x}\in V^1_g$. The coevaluation
is given by (sums over bases are implicit)
\[
1 \mapsto
v_i \tnsr v^i + \bar{w}_l \tnsr \bar{w}^l
+ w_k \tnsr w^k + \bar{v}_j \tnsr \bar{v}^j
\]
the first two terms are in $(V \tnsrz{z} V^\vee)_g$
and next two terms are in $(V \tnsrz{z} V^\vee)_{gz}$.
Then
\begin{align*}
x \mapsto v_i \tnsr v^i \tnsr x +
  \bar{w}_l \tnsr \bar{w}^l \tnsr x
\mapsto v_i\eval{v^i,x} + 0 = x
\\
\bar{x} \mapsto w_k \tnsr w^k \tnsr \bar{x} +
  \bar{v}_j \tnsr \bar{v}^j \tnsr \bar{x}
\mapsto 0 + \bar{v}_j\eval{\bar{v}^j,\bar{x}}
= \bar{x}
\;\;.
\end{align*}
\end{proof}

We have the analog of \thmref{t:ZA_DG_red}
(see also \ocite{wassermansymm}*{Theorem 51}):

\begin{theorem}
\label{t:ZA_DG_red_super}
The equivalence $T$ of \prpref{p:ZA_DG}
is a pivotal tensor equivalence
\[
T : (\cZ(\Rep(G,z)),\tnsrbar)
\simeq_\tnsr (\cD(G)-\mmod,\tnsrz{z})
\;\;.
\]
\end{theorem}
\begin{proof}
Essentially the same as before,
but again slightly tricky.
Recall that the actions of $e^0 = (1+z)/2$ and $e^1 = (1-z)/2$
are projections onto even and odd parts respectively,
and that we denote the even and odd parts of basis elements
$g \in \kk[G]$ and $\delta_g \in F(G) = \kk[G]^*$ by
\begin{align*}
g^\sigma &= e^\sigma \cdot g = (g + (-1)^\sigma gz)/2
\in \kk[G]
\\
\delta_g^\sigma &= e^\sigma \cdot \delta_g = 
(\delta_g + (-1)^\sigma \delta_{gz})/2
\in F(G)
\;\;.
\end{align*}
Then the braiding $c: \kk[G] \tnsr V^\sigma
\simeq V^\sigma \tnsr \kk[G]$
for $\sigma \in \ZZ/2$ is given by
\[
g \tnsr v =
(g^0 + g^1) \tnsr v
\overset{c}{\mapsto}
v \tnsr (g^0 + (-1)^\sigma g^1)
=
v \tnsr gz^\sigma
\;\;.
\]
Similarly, $c: F(G) \tnsr V^\sigma
\simeq V^\sigma \tnsr F(G)$ is given by
\[
\delta_g \tnsr v =
(\delta_g^0 + \delta_g^1) \tnsr v
\overset{c}{\mapsto}
v \tnsr
(\delta_g^0 + (-1)^\sigma \delta_g^1)
=
v \tnsr \delta_{gz^\sigma}
\;\;.
\]
We perform the same computation
as in the proof of \thmref{t:ZA_DG_red},
that is,
$T(Q_{\ga,\mu})$ on $V\tnsr W$;
once again the dashed line in $Q_{\ga,\mu}$
may be replaced by
the regular representation $\kk[G]$,
with which the braiding is given by the nice formulas above.
For $v\in V^\sigma_g, w\in W^\tau_h$,
(sum over $a \in G$ is implicit)
\begin{align*}
v \tnsr w \mapsto \frac{1}{|G|} v \tnsr w
&\mapsto \frac{1}{|G|}
a \tnsr \delta_a \tnsr v \tnsr w
\;\;\;
(\text{from }
\coev : 1 \mapsto g \tnsr \delta_g \in \kk[G] \tnsr F(G)\;)
\\
&\mapsto \frac{1}{|G|}
a \tnsr v \tnsr \delta_{az^\sigma} \tnsr w
\;\;\;
(\text{applying braiding }c)
\\
&\mapsto \frac{1}{|G|}
v \tnsr ga \tnsr w \tnsr h\cdot\delta_{az^\sigma}
\;\;\;\;
(\text{applying half-braidings }\ga, \mu)
\\
&\mapsto \frac{1}{|G|}
v \tnsr w \tnsr gaz^\tau \tnsr \delta_{haz^\sigma}
\;\;\;
(\text{applying braiding }c)
\\
&\mapsto
\delta_{gz^\tau,hz^\sigma} v \tnsr w
\;\;\;
(\text{pairing by } \ev_{\kk[G]})
\;\;.
\end{align*}
The action of $\delta_b$ on $v\tnsr w$,
i.e. projection to its $b$-graded component,
is the composition
\[
v \tnsr w
\mapsto
e \tnsr v \tnsr w
\overset{\ga}{\mapsto}
v \tnsr g \tnsr w
\overset{c}{\mapsto}
v \tnsr w \tnsr gz^\tau
\overset{\delta_b}{\mapsto}
\delta_{gz^\tau,b} \; v \tnsr w
\;\;.
\]
Thus, the projection of $T(Q_{\ga,\mu})$ restricted to
$V^\sigma_g \tnsr W^\tau_h$ has $G$-grading
$gz^\tau$,
so the $b$-graded component of $\im(T(Q_{\ga,\mu}))$
is $\bigoplus V^\sigma_g \tnsr W^\tau_h$,
where the sum is over all $g,h$ with
$gz^\tau = hz^\sigma = b$, which agrees with
\defref{d:symm_red_super}.
\end{proof}

\begin{remark} \label{r:T_symm_super}
Similar to \rmkref{r:T_symm},
we note that both $(\cZ(\Rep(G,z)),\tnsrbar)$
and $(\cD(G)-\mmod,\tnsrz{z})$
are symmetric: the braiding on $\Rep(G,z)$,
being symmetric, is naturally
a symmetric braiding for $(\cZ(\Rep(G,z)),\tnsrbar)$,
and $(\cD(G)-\mmod,\tnsrz{z})$ inherits the
braiding of super vector spaces
$(-1)^{\sigma \cdot \tau}P:
V^\sigma_{gz^\tau} \tnsr W^\tau_{gz^\sigma}
\simeq W^\tau_{gz^\sigma} \tnsr V^\sigma_{gz^\tau}$.
Then the above equivalence $T$ is in fact symmetric
(see \ocite{wassermansymm}*{Theorem 51}).
\rmkend
\end{remark}

We have an analogue of \prpref{p:coprod}:
\begin{proposition} \label{p:coprod_super}
Consider the map
\begin{align*}
\overline{\Delta}_z : \cD(G) &\mapsto \cD(G) \tnsr \cD(G)
\\
g &\mapsto g \tnsr g
\\
\delta_g &\mapsto \bigoplus_{\sigma,\tau}
  \delta^\sigma_{gz^\tau} \tnsr \delta^\tau_{gz^\sigma}
\end{align*}
for all $g\in G$, where
$\delta_g^\sigma = e^\sigma \delta_g$
is projection to the even/odd $g$-graded component.
This is an algebra homomorphism, 
and for $V,W \in \cD(G)-\mmod$,
\[
V \tnsrbar W = \overline{\Delta}_z(1) \cdot (V \tnsr W)
\]
where $1 = e \tnsr \delta_* = e \tnsr \sum_h \delta_h$
is the unit of $\cD(G)$.
\end{proposition}
\begin{proof}
Straightforward. We note that by parity consideration,
\begin{align*}
\overline{\Delta}_z(\delta_g^0)
&= \delta_g^0 \tnsr \delta_g^0
+ \delta_{gz}^1 \tnsr \delta_{gz}^1
\\
\overline{\Delta}_z(\delta_g^1)
&= \delta_g^1 \tnsr \delta_{gz}^0
+ \delta_{gz}^0 \tnsr \delta_g^1
\end{align*}
which reflects the observation made
after \defref{d:symm_red_super}.
\end{proof}

It turns out that these coproducts are the same
up to an algebra automorphism of $\cD(G)$:
\begin{proposition} \label{p:coprod_same}
The map
\begin{align*}
\lmb : \cD(G) &\simeq \cD(G)
\\
g &\mapsto g
\\
\delta_g &\mapsto \delta_g^0 + \delta_{gz}^1
\end{align*}
defines an isomorphism of algebras.
Furthermore, this isomorphism intertwines
the two coproducts $\overline{\Delta}$
and $\overline{\Delta}_z$,
i.e. the following diagram commutes:
\[
\begin{tikzcd}
\cD(G) \ar[r,"\lmb"] \ar[d,"\overline{\Delta}"]
& \cD(G) \ar[d,"\overline{\Delta}_z"]
\\
\cD(G) \tnsr \cD(G) \ar[r,"\lmb \tnsr \lmb"]
& \cD(G) \tnsr \cD(G)
\end{tikzcd}
\;\;\;
\begin{tikzcd}
\textcolor{white}{.}
\\
.
\end{tikzcd}
\]
\end{proposition}
\begin{proof}
Straightforward.
The computation simplifies greatly following the
observation in the proof of \prpref{p:coprod_super}.
\end{proof}

The following result is the corresponding categorical
statement for \prpref{p:coprod_same},
which says that the two tensor products $\tnsrbar$
and $\tnsrz{z}$ are essentially the same:

\begin{theorem} \label{t:super_same}
Let $\Lmb : \cD(G)-\mmod \simeq \cD(G)-\mmod$
be the pullback functor induced from
$\lmb: \cD(G) \simeq \cD(G)$ as defined
in \prpref{p:coprod_same}.
More explicitly, for $V\in \cD(G)-\mmod$,
$\Lmb(V)$ is the same as $V$ as $G$-module,
but the $F(G)$-action is changed so gradings are mixed:
\[
\Lmb(V)_g = V_g^0 \oplus V_{gz}^1
\;\;.
\]
Then $\Lmb$ is a pivotal tensor equivalence
\[
\Lmb : (\cD(G)-\mmod,\tnsrz{z}) \simeq_\tnsr
(\cD(G)-\mmod,\tnsrbar)
\;\;.
\]
\end{theorem}
\begin{proof}
Clearly the unit $\onebar = \bigoplus_{g\in G} \kk_g$
is preserved (all $\kk_g$ are even).
By \prpref{p:coprod} and \prpref{p:coprod_super},
the tensor products $\tnsrbar$ and $\tnsrz{z}$
arise from the coproducts $\overline{\Delta}$
and $\overline{\Delta}_z$.
Since $\lmb$ intertwines these coproducts,
we see that $\Lmb$ respects tensor products:
\begin{align*}
(\Lmb(V) \tnsrbar \Lmb(W))_g
&= \Lmb(V)_g \tnsr \Lmb(W)_g
\\
&= \overline{\Delta}(\delta_g) \cdot
(\Lmb(V) \tnsr \Lmb(W))
\\
&= ((\lmb \tnsr \lmb) \circ \overline{\Delta})(\delta_g)
\cdot (V \tnsr W)
\\
&= (\overline{\Delta}_z \circ \lmb)(\delta_g)
\cdot (V \tnsr W)
\\
&= \Lmb(V \tnsrz{z} W)_g
\;\;.
\end{align*}
Of course, a direct computation works too.
The pivotal structure is obviously respected.
\end{proof}

\begin{remark} \label{r:super_diff}
Thomas Wasserman pointed out that the equivalence
$\Lmb$ above does not respect the symmetric braidings
described in Remarks \ref{r:T_symm} and \ref{r:T_symm_super}.
We may readily see this by considering odd representations
$V^1, W^1$:
in $(\cD(G)-\mmod,\tnsrz{z})$, the braiding is given by
\[
(V^1 \tnsrz{z} W^1)_g
= V_{gz}^1 \tnsr W_{gz}^1
\xrightarrow{c \, = \, (-1)P}
= W_{gz}^1 \tnsr V_{gz}^1
(W^1 \tnsrz{z} V^1)_g
\]
while in $(\cD(G)-\mmod,\tnsrbar)$,
\[
(\Lmb(V^1) \tnsrbar \Lmb(W^1))_g
= \Lmb(V^1)_g \tnsr \Lmb(W^1)_g
= V_{gz}^1 \tnsr W_{gz}^1
\xrightarrow{c \, = \, P}
= W_{gz}^1 \tnsr V_{gz}^1
= (\Lmb(W^1) \tnsrbar \Lmb(V^1))_g
\;\;.
\]
Thus, $\Lmb$ is not a braided tensor equivalence.
\rmkend
\end{remark}

Next we study the $\ZZ/2$-action on $\cD(G)-\mmod$.
Here $\cA = \Rep(G,z)$, where we treat the non-super
case as setting $z=e$.

\begin{proposition}
Let $U_S : \cD(G)-\mmod \to \cD(G)-\mmod$
that changes the grading by the inverse map on $G$.
That is,
\[
(U_S(V))_g = V_{g^\inv}
\;\;.
\]
Then $U_S$ is naturally a tensor automorphism,
and $U^2_S = \id$, so $U_S$ generates a $\ZZ/2$-action.
\end{proposition}
\begin{proof}
Trivial.
\end{proof}

\begin{proposition}
The equivalence $T$ of \prpref{p:ZA_DG} is
$\ZZ/2$-equivariant.
\end{proposition}
\begin{proof}
Let $(V,\ga) \in \cZ(\Rep(G,z))$.
$U$ on $\cZ(\Rep(G,z))$ only affects the half-braiding,
so $T((V,\ga)) = V = T((V,\wdtld{\ga}))$
are the same as $G$-modules.
It suffices to check that the action of
$\delta_g$ on $T((V,\wdtld{\ga}))$
is the same as the action of
$\delta_{g^\inv}$ on $T((V,\ga))$.
For $v\in T((V,\wdtld{\ga}))^\sigma$, the action
of $\delta_g$ is
\begin{align*}
\begin{tikzpicture}
\node[small_morphism] (ga) at (0,0) {\small $\ga$};
\draw (ga) -- (0,-0.8);
\draw (ga) .. controls +(-30:0.3) and +(-30:0.4cm) ..
  (0,0.3) -- +(150:0.4cm)
  node[small_morphism] {$e$};
\draw
  (ga) .. controls +(150:0.3cm) and +(150:0.4cm) ..
  (0,-0.3) -- +(-30:0.4cm)
  node[small_morphism] {\tiny $\delta_g$};
\draw (ga) -- (0,0.8) node[above] {\small $v$};
\end{tikzpicture}
\;\;\;\;
\begin{split}
v &\mapsto
e \tnsr v
\\ &\mapsto
\sum_h h \tnsr \delta_h \tnsr v \tnsr z^\sigma
\\ &\mapsto
\sum_{h,h'} h \tnsr \delta_{h'} \tnsr h' \delta_h 
  \cdot v \tnsr z^\sigma
  \;\;
  (\text{using }\delta_{h'}\text{ action of }T((V,\ga)))
\\ &\mapsto
\sum_{h,h'} \delta_g(hz^\sigma) 
  \delta_{hh'}(z^\sigma) \delta_{h'} \cdot v
= \delta_{g^\inv} \cdot v
\;\;.
\end{split}
\end{align*}
Thus $T$ commutes with the $\ZZ/2$-actions.
It remains to note that for simple $(V,\ga)$,
$V$ is either all even or odd,
so the twist operator is $\pm 1$,
so $\wdtld{\wdtld{\ga}} = \ga$
(see end of proof of \prpref{p:aut_ZA})
and hence $U^2 = \id$.
\end{proof}

\subsection{$\ZCY(\torus)$ for $\cA$ symmetric}
\label{s:symm_torus}
\par \noindent
\\
Here we describe $\ZCY(\torus)$ when $\cA$ is symmetric.
Recall that we had \corref{c:torus_ZA}, which stated
\[
\ZCY(\torus) \simeq \cZ((\ZA,\tnsrbar))
\;\;.
\]

By \thmref{t:super_same}, it suffices to deal with
the non-super case $\cA = \Rep(G)$.
The following is taken from
\ocite{tham}*{Definition-Proposition 5.2}:
\begin{definition} \label{d:elliptic}
The \emph{elliptic double} of $G$ is the algebra
$\cD^{\text{el}}(G)$
whose underlying vector space is 
\[
\cD^{\text{el}}(G) = \kk[G] \tnsr F(G) \tnsr F(G)
\]
with algebra structure
\[
g \tnsr \delta_{h_1} \tnsr \delta_{h_2}
\cdot
g' \tnsr \delta_{h_1'} \tnsr \delta_{h_2'}
=
gg' \tnsr \delta_{h_1}\delta_{g'^\inv h_1' g'}
\tnsr \delta_{h_2}\delta_{g'^\inv h_2' g'}
\;\;.
\]
Let $\delta_* = \sum_{h\in G} \delta_h$ be the
unit of $F(G)$.
We will denote $g \tnsr \delta_* \tnsr \delta_*$
simply by $g$, $e \tnsr \delta_h \tnsr \delta_*$
by $\delta_h^1$, and
$e \tnsr \delta_* \tnsr \delta_h$ by $\delta_h^2$.
We may also denote $\delta_h^1 \delta_h^2$ 
by $\delta_{(h,h')}$
\defend
\end{definition}

The $*$ notation in the subscript will be used often,
and will generally mean to sum over all possible
entries.

Observe that any orbit $c$ of the diagonal action
of $G$ on $G \times G$ (acting by conjugation
on each factor) gives rise to a central idempotent
$\sum_{(h,h') \in c} \delta_{(h,h')}$.
In particular, the set of commuting pairs
\[
\Omega = \{(h,h') \in G \times G \; | \; [h,h'] = e\}
\]
is a union of conjugacy classes,
so gives rise to the following subalgebra:
\begin{definition}
$\cD_{\torus}(G)$ is the subalgebra of
$\cD^{\text{el}}(G)$
\[
\cD_{\torus}(G) :=
\{
g\delta_{(h,h')} | (h,h') \in \Omega
\}
= \cD^{\text{el}}(G) \cdot \delta_\Omega
\]
where $\delta_\Omega
= \sum_{(h,h')\in \Omega} \delta_{(h,h')}$.
We will use the same notation as in
\defref{d:elliptic}, so for example
$\delta_h^1$ stands for $\delta_h^1 \delta_\Omega
= \sum_{h'\,:\,[h,h'] = e}
  \delta_{(h,h')}$
\defend
\end{definition}
It is easy to see that $\cD_\torus(G)-\mmod$ is the
category of $G$-equivariant bundles over $\Omega$.

\begin{theorem} \label{t:torus_DG}
As abelian categories.
\[
\cZ((\cD(G)-\mmod,\tnsrbar)) \simeq \cD_{\torus}(G)-\mmod
\;\;.
\]
\end{theorem}
\begin{proof}
Observe that a $G$-equivariant bundle over $G$,
$V \in \cD(G)-\mmod$,
naturally decomposes into bundles where each is
supported on a single conjugacy class:
\[
V = \bigoplus_{\kappa \in G//G} V_\kappa,
\;\;
V_\kappa := \bigoplus_{g\in \kappa} V_g
\;\;.
\]
Furthermore, $V_\kappa$ is determined by $V_g$ for some
$g \in \kappa$ together with its restricted $Z(g)$-action.
To recover $V_\kappa$ from $V_g$, simply take the induction
$\kk[G] \tnsr_{Z(g)} V_g$, and specify that the grading of
$h \tnsr v$ is $hgh^\inv$.
Moreover, the fiberwise tensor product
restricts to the standard tensor product of
$Z(g)$-modules.
Thus we have, for choices $g\in \kappa$ for each conjugacy
class $\kappa \in G//G$,
\[
(\cD(G)-\mmod, \tnsrbar) \simeq_\tnsr
\bigoplus_{\kappa \in G//G} \Rep(Z(g))
\]
and hence,
\[
\cZ((\cD(G)-\mmod, \tnsrbar)) \simeq
\bigoplus_{\kappa \in G//G} \cZ(\Rep(Z(g))
\;\;.
\]
As we've seen before,
$\cZ(\Rep(Z(g)) \simeq \cD(Z(g))-\mmod$,
or the category of $Z(g)$-equivariant bundles over
$Z(g)$.

Now consider $W \in \cD_\torus(G)-\mmod$.
It can also be decomposed
\[
W = \bigoplus_{\kappa \in G//G} W_{(\kappa,*)},
\;\;
W_{(\kappa,*)} :=
\bigoplus_{h\in \kappa, (h,h')\in \Omega} W_{(h,h')}
\;\;.
\]
Similarly, $W_{(\kappa,*)}$ is determined by its restriction
to $W_{(h,*)}
=\bigoplus_{h'\,:\,(h,h')\in \Omega} W_{(h,h')}$
for some $h\in \kappa$,
together with its $Z(h)$-action.
Once again, induction $\kk[G] \tnsr_{Z(h)} W_{(h,*)}$
recovers $W_{(\kappa,*)}$,
with grading of $g \tnsr v$ for $v\in W_{(h,h')}$
given by $g\cdot (h,h') = (ghg^\inv,gh'g^\inv)$.
Note that $W_{(h,*)}$ is naturally a $Z(h)$-equivariant
bundle over $Z(h)$, and thus we have,
for choices $g\in \kappa$ for each conjugacy class
$\kappa \in G//G$,
\[
\cD_\torus(G)-\mmod \simeq \bigoplus_{\kappa \in G//G}
\cD(Z(g))-\mmod
\]
and hence we are done.
\end{proof}

\begin{remark}
Let us discuss a more topological
approach to/insight on $\cD_\torus(G)$.
In \ocite{tham}, we defined a category called the
``elliptic Drinfeld center" of $\cA$,
which we showed in \ocite{KT} to be equivalent
to $\ZCY(\punctorus)$,
where $\punctorus$ is the once-punctured torus.
When $\cA = \Rep(G)$, this category is equivalent
to the category of $G$-equivariant bundles over $G \times G$;
the equivalence depends on a choice of meridian and longitude
on $\punctorus$, so that each $G$ factor in $G\times G$
corresponds to one of the curves.
As we seal up the puncture in $\punctorus$ to get $\torus$,
the two $G$ factors are forced to commute,
thus $\CYA$ consists of bundles supported only on $\Omega$.
\end{remark}



\begin{bibdiv}
\begin{biblist}
\bib{AF}{article}{
  label={AF},
  author={Ayala, David},
  author={Francis, John},
  title={A factorization homology primer},
  date={2019},
  eprint={arXiv:1903.1096},
}
\bib{AFR}{article}{
  label={AFR},
  author={Ayala, David},
  author={Francis, John},
  author={Rozenblyum, Nick},
  title={Factorization homology I: Higher categories},
  journal={Adv. Math.},
  volume={333},
  date={2018},
  pages={1042--1177},
  issn={0001-8708},
}
\bib{BakK}{book}{
  label={BK},
  author={Bakalov, Bojko},
  author={Kirillov, Alexander, Jr.},
  title={Lectures on tensor categories and modular functors},
  series={University Lecture Series},
  volume={21},
  publisher={American Mathematical Society},
  place={Providence, RI},
  date={2001},
  pages={x+221},
  isbn={0-8218-2686-7},
  review={\MR{1797619 (2002d:18003)}},
}
\bib{BHLZ}{article}{
  label={BHLZ},
  author={Beliakova, Anna},
  author={Habiro, Kazuo},
  author={Lauda, Aaron D.},
  author={\v{Z}ivkovi\'{c}, Marko},
  title={Trace decategorification of categorified quantum $\mathfrak{sl}_2$},
  journal={Math. Ann.},
  volume={367},
  date={2017},
  number={1-2},
  pages={397--440},
  issn={0025-5831},
  eprint={arXiv:1404.1806},
}
\bib{BZBJ1}{article}{
  label={BZBJ1},
  author={Ben-Zvi, David},
  author={Brochier, Adrien},
  author={Jordan, David},
  title={Integrating quantum groups over surfaces},
  journal={J. Topol.},
  volume={11},
  date={2018},
  number={4},
  pages={874--917},
  issn={1753-8416},
  review={\MR{3847209}},
  doi={10.1112/topo.12072},
}
\bib{BZBJ2}{article}{
  label={BZBJ2},
  author={Ben-Zvi, David},
  author={Brochier, Adrien},
  author={Jordan, David},
  title={Quantum character varieties and braided module categories},
  journal={Selecta Math. (N.S.)},
  volume={24},
  date={2018},
  number={5},
  pages={4711--4748},
  issn={1022-1824},
  review={\MR{3874702}},
  doi={10.1007/s00029-018-0426-y},
}
\bib{Cooke}{article}{
  label={Co},
  author={Cooke, Juliet},
  title={Excision of Skein Categories and Factorisation Homology},
  date={2019},
  eprint={arXiv:1910.02630},
}
\bib{CY}{article}{
  author={Crane, Louis},
  author={Yetter, David},
  title={A categorical construction of $4$D topological quantum field
  theories},
  conference={
    title={Quantum topology},
  },
  book={
    series={Ser. Knots Everything},
    volume={3},
    publisher={World Sci. Publ., River Edge, NJ},
  },
  date={1993},
  pages={120--130},
}
\bib{CKYeval}{article}{
  label={CKY1},
  author={Crane, Louis},
  author={Kauffman, Louis H.},
  author={Yetter, David},
  title={Evaluating the Crane-Yetter invariant},
  conference={
    title={Quantum topology},
  },
  book={
    series={Ser. Knots Everything},
    volume={3},
    publisher={World Sci. Publ., River Edge, NJ},
  },
  date={1993},
  pages={131--138},
  review={\MR{1273570}},
  doi={10.1142/9789812796387\_0006},
}
\bib{CKY}{article}{
  label={CKY2},
  author={Crane, Louis},
  author={Kauffman, Louis H.},
  author={Yetter, David N.},
  title={State-sum invariants of $4$-manifolds},
  journal={J. Knot Theory Ramifications},
  volume={6},
  date={1997},
  number={2},
  pages={177--234},
  issn={0218-2165},
  review={\MR{1452438}},
  doi={10.1142/S0218216597000145},
}
\bib{EGNO}{book}{
  label={EGNO},
  author={Etingof, Pavel},
  author={Gelaki, Shlomo},
  author={Nikshych, Dmitri},
  author={Ostrik, Victor},
  title={Tensor categories},
  series={Mathematical Surveys and Monographs},
  volume={205},
  publisher={American Mathematical Society, Providence, RI},
  date={2015},
  pages={xvi+343},
  isbn={978-1-4704-2024-6},
}
\bib{ENO2005}{article}{
  label={ENO2005},
  author={Etingof, Pavel},
  author={Nikshych, Dmitri},
  author={Ostrik, Viktor},
  title={On fusion categories},
  journal={Ann. of Math. (2)},
  volume={162},
  date={2005},
  number={2},
  pages={581--642},
  issn={0003-486X},
  review={\MR{2183279 (2006m:16051)}},
  doi={10.4007/annals.2005.162.581},
}
\bib{GNN}{article}{
  label={GNN},
  author={Gelaki, Shlomo},
  author={Naidu, Deepak},
  author={Nikshych, Dmitri},
  title={Centers of graded fusion categories},
  journal={Algebra Number Theory},
  volume={3},
  date={2009},
  number={8},
  pages={959--990},
  issn={1937-0652},
}
\bib{stringnet}{article}{ 
  label={Kir},
  author={Kirillov, Alexander, Jr},
  title={String-net model of Turaev-Viro invariants},
  eprint={arXiv:1106.6033},
}
\bib{Mu}{article}{
  label={Mu},
  author={M\"{u}ger, Michael},
  title={On the structure of modular categories},
  journal={Proc. London Math. Soc. (3)},
  volume={87},
  date={2003},
  number={2},
  pages={291--308},
  issn={0024-6115},
  review={\MR{1990929}},
  doi={10.1112/S0024611503014187},
}

\bib{KT}{article}{
  label={KT},
  title={Factorization Homology and 4D TQFT},
  author={Alexander Kirillov Jr and Ying Hong Tham},
  year={2020},
  journal={Quantum Topology (accepted)},
  eprint={2002.08571},
  archivePrefix={arXiv},
  primaryClass={math.QA}
}
\bib{Ocn}{article}{
  label={O},
  author={Ocneanu, Adrian},
  title={Chirality for operator algebras},
  conference={
    title={Subfactors},
     address={Kyuzeso},
     date={1993},
  },
  book={
     publisher={World Sci. Publ., River Edge, NJ},
  },
  date={1994},
  pages={39--63},
  review={\MR{1317353}},
}
\bib{RT}{article}{
  label={RT},
  author={Reshetikhin, N. Yu.},
  author={Turaev, V. G.},
  title={Ribbon graphs and their invariants derived from quantum groups},
  journal={Comm. Math. Phys.},
  volume={127},
  date={1990},
  number={1},
  pages={1--26},
  issn={0010-3616},
  review={\MR{1036112}},
}
\bib{tham_thesis}{article}{
  author={Tham, Ying Hong},
  title={On the Category of Boundary Values in the Extended Crane-Yetter TQFT},
  year={2021},
  eprint={arXiv:2108.13467}
}
\bib{tham}{article}{
  label={Tham},
  title={The Elliptic Drinfeld Center of a Premodular Category},
  author={Ying Hong Tham},
  year={2019},
  eprint={1904.09511}
}
\bib{TV}{article}{
  label={TV},
  author={Turaev, V. G.},
  author={Viro, O. Ya.},
  title={State sum invariants of $3$-manifolds and quantum $6j$symbols},
  journal={Topology},
  volume={31},
  date={1992},
  number={4},
  pages={865--902},
  issn={0040-9383},
  review={\MR{1191386}},
  doi={10.1016/0040-9383(92)90015-A},
}
\bib{wassermansymm}{article}{
  label={Was},
  author={Wasserman, Thomas A.},
  title={The symmetric tensor product on the Drinfeld centre of a symmetric
  fusion category},
  journal={J. Pure Appl. Algebra},
  volume={224},
  date={2020},
  number={8},
  pages={106348},
  issn={0022-4049},
  review={\MR{4074586}},
  doi={10.1016/j.jpaa.2020.106348},
}
\bib{wasserman2fold}{article}{
  label={Was2},
  author={Wasserman, Thomas A.},
  title={The Drinfeld centre of a symmetric fusion category is 2-fold
  monoidal},
  journal={Adv. Math.},
  volume={366},
  date={2020},
  pages={107090},
  issn={0001-8708},
  review={\MR{4072794}},
  doi={10.1016/j.aim.2020.107090},
}
\end{biblist}
\end{bibdiv}
\end{document}